\documentclass[final]{zzz}
\usepackage{amsmath,amssymb,amsthm,amsfonts}
\usepackage{verbatim}
\usepackage{showkeys} 
\usepackage{enumerate}
\usepackage{soul}
\usepackage{xcolor}
\usepackage{hyperref}
\usepackage{faktor}
\usepackage{bm}
\usepackage{comment}

\hypersetup{
  colorlinks = true,
  urlcolor = blue,
  linkcolor = black,
  citecolor = red
}

\newcommand{\hyp}[5]{{}_{#1}F_{#2}\!\left(\genfrac{}{}{0pt}{}{#3}{#4};#5\right)}

\newcommand{\Ohyp}[5]{\,\mbox{}_{#1}{\bm{F}}_{#2}\!\left(
\genfrac{}{}{0pt}{}{#3}{#4};#5\right)}

\definecolor{Mycolor2}{HTML}{e85d04}
\sethlcolor{black}

\newcommand\boro[1]{{\textcolor{black}{#1}}}

\newcommand{\bfx}{{\bf x}}

\newcommand{\bfxp}{{{\bf x}^\prime}}
\newcommand{\N}{{\mathbf N}}

\newcommand{\Si}{{\mathbf S}}
\newcommand{\Pii}{{\mathbf P}}
\newcommand{\Hi}{{\mathbf H}}

\newcommand{\mcg}{{\mathcal G}}

\newcommand{\OO}{{{\mathbb O}}}
\newcommand{\HH}{{{\mathbb H}}}

\def\mcO{\mathcal{G}^{\OO\Hi^2_R}}
\def\mcH{\mathcal{G}^{\HH\Hi^n_R}}
\def\mcC{\mathcal{G}^{\CC\Hi^n_R}}
\def\mpO{\mathcal{G}^{\OO\Pii^2_R}}
\def\mpH{\mathcal{G}^{\HH\Pii^n_R}}
\def\mpC{\mathcal{G}^{\CC\Pii^n_R}}

\def\CH2{\CC\Hi^2_R}

\newcommand{\expe}{{\mathrm e}}
\newcommand{\dd}{{\mathrm d}}

\newcommand{\Z}{{\mathbf Z}}

\newcommand{\CC}{{{\mathbb C}}}

\newtheorem{thm}[lemma]{Theorem}
\newtheorem{cor}[lemma]{Corollary}
\newtheorem{rem}[lemma]{Remark}
\newtheorem{lem}[lemma]{Lemma}

\makeatletter
\def\eqnarray{\stepcounter{equation}\let\@currentlabel=\theequation
\global\@eqnswtrue
\tabskip\@centering\let\\=\@eqncr
$$\halign to \displaywidth\bgroup\hfil\global\@eqcnt\z@
 $\displaystyle\tabskip\z@{##}$&\global\@eqcnt\@ne
 \hfil$\displaystyle{{}##{}}$\hfil
 &\global\@eqcnt\tw@ $\displaystyle{##}$\hfil
 \tabskip\@centering&\llap{##}\tabskip\z@\cr}

\def\endeqnarray{\@@eqncr\egroup
   \global\advance\c@equation\m@ne$$\global\@ignoretrue}

\def\@yeqncr{\@ifnextchar [{\@xeqncr}{\@xeqncr[5pt]}}
\makeatother

\def\sideremark#1{\ifvmode\leavevmode\fi\vadjust{\vbox to0pt{\vss
 \hbox to 0pt{\hskip\hsize\hskip1em%
 \vbox{\hsize2cm\tiny\raggedright\pretolerance10000
 \noindent #1\hfill}\hss}\vbox to8pt{\vfil}\vss}}}

\newcommand{\R}{\mathbb{R}}
\newcommand{\C}{\mathbb{C}}

\begin{document}

\renewcommand{\PaperNumber}{***}

\FirstPageHeading

\ShortArticleName{
Double summation addition theorems for Jacobi functions of the first and second kind
}

\ArticleName{Double summation addition theorems for\\ Jacobi functions of the first and second kind}
\Author{Howard S.~Cohl\,$^\ast$, 
Roberto S.~Costas-Santos\,$^\S$,
Loyal Durand\,$^\dag$,
Camilo Montoya\,$^\ast$
and Gestur \'{O}lafsson\,$^\ddag$
}

\AuthorNameForHeading{H.~S.~Cohl, 
R.~S.~Costas-Santos,
L.~Durand,
C. Montoya, G. \'{O}lafsson}

\Address{$^\ast$~Applied and Computational Mathematics Division,
National Institute of Standards and Technology,
Gaithersburg, MD 20899-8910, USA
} 
\URLaddressD{
\href{http://www.nist.gov/itl/math/msg/howard-s-cohl.cfm}
{http://www.nist.gov/itl/math/msg/howard-s-cohl.cfm}
}
\EmailD{howard.cohl@nist.gov, camilo.montoya@nist.gov} 

\Address{$^\S$ Department of Quantitative Methods, Universidad Loyola Andaluc\'ia, Sevilla, Spain
} 
\URLaddressD{
\href{http://www.rscosan.com}
{http://www.rscosan.com}
}
\EmailD{rscosa@gmail.com} 

\Address{$^\dag$~Department of Physics, University of Wisconsin, Madison, WI 53706, USA
} 
\EmailD{ldurand@wisc.edu} 

\Address{$^\ddag$~Department of Mathematics, Louisiana State University, Baton Rouge, LA 70803-4918, USA
} 
\URLaddressD{
\href{http://www.math.lsu.edu/~olafsson}
{http://www.math.lsu.edu/$\sim$olafsson}
}
\EmailD{olafsson@lsu.edu} 

\ArticleDates{Received ???, in final form ????; Published online ????}

\Abstract{In this paper we review and derive hyperbolic and trigonometric double summation addition theorems for Jacobi functions of the first and second kind. In connection with these addition theorems, we perform a full analysis of the relation between symmetric, antisymmetric and odd-half-integer parameter values for the Jacobi functions with certain Gauss hypergeometric functions which satisfy a quadratic transformation, including associated Legendre, Gegenbauer and Ferrers functions of the first and second kind.
We also introduce Olver normalizations of the Jacobi functions which are particularly useful in the derivation of expansion formulas when the parameters are integers. We introduce an application of the addition theorems for the Jacobi functions of the second kind to separated eigenfunction expansions of a fundamental solution of the Laplace-Beltrami operator on the compact and noncompact rank one symmetric spaces.}


\Classification{33C05, 33C45, 53C22, 53C35}

\begin{flushright}
\begin{minipage}{70mm}
\it Dedicated to Dick Askey whose favorite function was the Jacobi polynomial.
\end{minipage}
\end{flushright}

\section{Introduction}

Jacobi polynomials (hypergeometric polynomials)
were introduced by the German mathematician Carl Gustav 
Jacob Jacobi (1804--1851).
These polynomials first appear in an article by Jacobi which was published posthumously in 1859 by Heinrich Eduard Heine \cite{Jacobi1859}. Jacobi polynomials, $P_n^{(\alpha,\beta)}(x)$, 
are polynomials which for $\Re \alpha, \Re\beta>-1$ are 
orthogonal on the real segment $[-1,1]$ \cite[(9.8.2)]{Koekoeketal} 
and can be defined in terms of a terminating sum as follows:
\begin{equation}
P_n^{(\alpha,\beta)}(\cos\theta):=\frac{\Gamma(\alpha+1+n)}
{\Gamma(n+1)\,\Gamma(\alpha+\beta+1+n)}
\sum_{k=0}^n(-1)^k\binom{n}{k}
\frac{\Gamma(\alpha+\beta+1+n+k)}{\Gamma(\alpha+1+k)}
\sin^{2k}(\tfrac12\theta),
\end{equation}
where $\Gamma$ is the gamma function \cite[\href{http://dlmf.nist.gov/5.2.E1}{(5.2.1)}]{NIST:DLMF}, and
$\binom{n}{k}$ the binomial 
coefficient \cite[\href{http://dlmf.nist.gov/1.2.E1}{(1.2.1)}]{NIST:DLMF}.
The above definition of the Jacobi polynomial is equivalent to 
the following Gauss hypergeometric representation \cite[\href{http://dlmf.nist.gov/18.5.E7}{(18.5.7)}]{NIST:DLMF}:
\begin{equation}
\label{Jacobipolydef}
P_n^{(\alpha,\beta)}(x)=
\frac{(\alpha+1)_n}{n!}
\hyp21{-n,n+\alpha+\beta+1}{\alpha+1}{\frac{1\!-\!x}{2}},
\end{equation}
where $x=\cos \theta$.
We will return to the notations used in \eqref{Jacobipolydef} 
in the following section.

Ultraspherical polynomials, traditionally defined by \cite[\href{http://dlmf.nist.gov/18.7.E1}{(18.7.1)}]{NIST:DLMF}
\begin{equation}
C_n^\lambda(\cos\theta):=\frac{\Gamma(\lambda+\frac12)
\Gamma(2\lambda+n)}{\Gamma(2\lambda)\Gamma(\lambda+\frac12+n)}
P_n^{(\lambda-\frac12,\lambda-\frac12)}(\cos\theta),
\end{equation}
are symmetric $\alpha\!=\!\beta$ Jacobi polynomials.
These polynomials are 
commonly referred 
to as Gegenbauer polynomials after Austrian mathematician 
Leopold Gegenbauer 
(1849--1903). However the Czech (Austrian) astronomer and mathematician Moriz All\'{e} discovered and used many of their fundamental properties including their generating function {\it and} addition theorem \cite{Alle1865} almost a decade prior to Gegenbauer \cite{Gegenbauer1874,Gegenbauer1893}, and Heine \cite[p.~455]{Heine1878}.
See the nice discussion of the history of the addition theorem for ultraspherical polynomials by Koornwinder in \cite[p.~383]{Koornwinder18}; see also \cite[\href{http://dlmf.nist.gov/18.18.E8}{(18.18.8)}]{NIST:DLMF}.
The addition theorem for ultraspherical polynomials is given by 
\begin{eqnarray}
&&\hspace{-0.5cm}C_n^\lambda(\cos\theta_1\cos\theta_2
\pm\sin\theta_1\sin\theta_2\cos\phi)\nonumber\\
&&\hspace{0.5cm}=\frac{n!}{(2\lambda)_n}\sum_{k=0}^n \frac{(\mp 1)^k 
(\lambda)_k(2\lambda)_{2k}}{(-n)_k(\lambda-\frac12)_k(2\lambda+n)_k}
(\sin\theta_1\sin\theta_2)^kC_{n-k}^{\lambda+k}(\cos\theta_1)
C_{n-k}^{\lambda+k}(\cos\theta_2)C_k^{\lambda-\frac12}(\cos\phi).
\label{Gegaddn}
\end{eqnarray}
This result is quite important all by itself. 
In the special case $\lambda=\frac12$, it becomes one way of writing the addition 
theorem for spherical harmonics on the two-dimensional sphere:
\begin{equation}
P_n(\cos\theta_1\cos\theta_2\pm \sin\theta_1\sin\theta_2\cos\phi)=
\sum_{k=-n}^n(\pm 1)^k
\frac{(n-k)!}{(n+k)!}{\sf P}_n^k(\cos\theta_1)
{\sf P}_n^k(\cos\theta_2)\cos(k\phi),
\end{equation}
where the ${\sf P}_n^k$ are Ferrers functions of the first 
kind \cite[\href{http://dlmf.nist.gov/14.3.E1}{(14.3.1)}]{NIST:DLMF}. 
See the foreword of Willard Miller (1977) \cite{Miller} written by Richard Askey for a 
beautiful discussion (on pp.~xix--xx) of addition theorems. 
These addition theorems are intimately related to separated 
eigenfunction expansions of spherical functions (reproducing kernels) 
on highly symmetric (isotropic) manifolds. 
In fact, on a $d$-dimensional hypersphere, the special argument 
of Gegenbauer's addition theorem is easily expressible 
in terms of the geodesic distance between two arbitrary points.

Given the addition theorem for ultraspherical polynomials \eqref{Gegaddn}, it was a natural problem to extend this to Jacobi polynomials for $\alpha\ne\beta$. It was a good match when Richard Askey, on sabbatical at the Mathematical Centre in Amsterdam during 1969--1970, met Tom Koornwinder there, who had some experience with group theoretical methods and was looking for a good subject for a Ph.D.~thesis. 
Askey suggested to Koornwinder the problem of finding an addition theorem for Jacobi polynomials, and he also arranged that Koornwinder could attend a special year at the Mittag-Leffler Institute in Sweden. There Koornwinder obtained the desired result \cite{Koornwinder1972AI,Koornwinder72B,Koornwinder72C}.
He later found that his group theoretic method and the resulting addition theorem in a special case were anticipated by two papers in Russian:~Vilenkin and \v{S}apiro \cite{VilenkinSapiro1967}
realized that disk polynomials 
\cite[\href{http://dlmf.nist.gov/18.37.E1}{(18.37.1)}]{NIST:DLMF} 
and in particular the Jacobi polynomials 
$P_n^{(\alpha,0)}$, $\alpha$ {an} integer, can be interpreted as spherical functions on the complex projective space $\mathrm{SU}(\alpha+2)/\mathrm{U}(\alpha+1)$ 
\cite{Cartan1929,Cartan1931} or as spherical functions on 
the complex unit sphere $\mathrm{U}(\alpha+2)/\mathrm{U}(\alpha+1)$ in
$\mathbb{C}^{\alpha+2}$ as a homogeneous space of the unitary group $\mathrm{U}(\alpha+2)$
(see references by Ikeda, Kayama and Seto in \cite{Koornwinder72A}).
\v{S}apiro 
obtained from that observation, the addition theorem for Jacobi polynomials in the $\beta=0$ case
\cite{Sapiro1968}.

Koornwinder initially presented his addition theorem
for Jacobi polynomials in a series of three papers in 1972
\cite{Koornwinder1972AI,Koornwinder72B,Koornwinder72C}.
Koornwinder gave four different proofs of the addition formula for Jacobi polynomials. His first proof {focused on the spherical functions of the Lie group}
$\mathrm{U}(d)/\mathrm{U}(d-1)$, $d\ge 2$ integer, and appeared in 
\cite{Koornwinder72B,Koornwinder72C}; his second proof that used
ordinary spherical harmonics
appeared in \cite{Koornwinder73};
his third proof was 
an analytic proof and it appeared in \cite{MR385197,Koornwinder74,Koornwinder75}; and a short proof using
orthogonal polynomials in three variables which appeared in 
\cite{Koornwinder77}.

Let us consider 
the trigonometric context of Koornwinder's addition theorem for 
the Jacobi polynomials. 
Let $n\in\mathbb N_0$, $\alpha>\beta>-\frac12$,
$\cos\theta_1=\frac12(\expe^{i\theta_1}+\expe^{-i\theta_1})$,
$\cos\theta_2=\frac12(\expe^{i\theta_2}+\expe^{-i\theta_2})$, 
$w\in(-1,1)$,
$\phi\in[0,\pi]$.
Then Koornwinder's addition theorem for Jacobi polynomials is given by
\begin{eqnarray}
&&\hspace{-0.45cm}{P}_n^{(\alpha,\beta)}
\left(2|\cos\theta_1\cos\theta_2\pm\expe^{i\phi}w\sin\theta_1
\sin\theta_1|^2-1\right)
=\frac{n!\Gamma(\alpha+1)}{\Gamma(\alpha+n+1)}\sum_{k=0}^n
\frac{(\alpha+1)_k(\alpha+\beta+n+1)_k}
{(\alpha+k)(\beta+1)_k(-n)_k}
\nonumber\\
&&\hspace{-0.25cm}\times\sum_{l=0}^k
(\mp 1)^{k-l}
\frac{(\alpha+k+l)(-\beta-n)_l}{(\alpha+n+1)_l}
(\cos\theta_1\cos\theta_2)^{k-l}\left(\sin\theta_1\sin\theta_2\right)^{k+l}
\nonumber\\[0.1cm]
&&\hspace{0.1cm}\times
{P}_{n-k}^{(\alpha+k+l,\beta+k-l)}\left(\cos(2\theta_1)\right)
{P}_{n-k}^{(\alpha+k+l,\beta+k-l)}\left(\cos(2\theta_2)\right)
w^{k-l}{P}_{l}^{(\alpha-\beta-1,\beta+k-l)}(2w^2-1)\frac{\beta\!+\!k\!-\!l}{\beta}C_{k-l}^\beta(\cos\phi) .
\label{addPtrigfirst}
\end{eqnarray}
As we will see in Section \ref{reffirstKoorn} below, this addition theorem and its various counterparts for Jacobi functions of the first and second kind are deeply connected to a 2-variable orthogonal polynomial system sometimes referred to as parabolic biangle polynomials ${\mathcal P}_{k,l}^{(\alpha,\beta)}(w,\phi)$.
Furthermore, this addition theorem for Jacobi polynomials represents a separated eigenfunction 
expansion of the spherical functions on highly symmetric 
(isotropic) manifolds (so-called compact symmetric spaces of rank one) and the special argument is again 
given by the geodesic distance between two points on these manifolds. 
We shall return to this later.

In the case of ultraspherical and Jacobi polynomials, the sum is terminating, 
as one would expect since the object of study is a polynomial. 
However, as Flensted-Jensen and Koornwinder realized \cite{FlenstedJensenKoorn79}, the addition theorem for Jacobi polynomials can 
be extended to Jacobi functions of the first kind by formally 
taking the outer sum limit to infinity. 
While Jacobi polynomials $P_n^{(\alpha,\beta)}$ have $n$ discrete, the Jacobi functions $\varphi_\lambda^{(\alpha,\beta)}$ have $\lambda$ continuous
(see \eqref{FJKJdef} below).
When one starts to consider Jacobi functions, then many 
new questions arise which must be understood for a full 
theory of the separated eigenfunctions expansions of Jacobi functions. 
First of all, one must consider two separate contexts, 
the trigonometric context 
where the arguments of the functions are analytically 
continued from the segment $(-1,1)$ and also the hyperbolic 
context, where the arguments of the functions are analytically 
continued from the segment $(1,\infty)$. 
On top of that, one must also consider the particular expansions 
of Jacobi functions of the second kind.
Gegenbauer and Jacobi functions are solutions to 
second-order ordinary differential equations. 
Therefore there are two linearly independent solutions, 
namely the functions of the first kind and the functions 
of the second kind. The separated eigenfunction expansions 
of the Gegenbauer functions of the first and second 
kind were treated quite extensively in a paper by 
Durand, Fishbane and Simmons (1976) \cite{DurandFishSim}. 
Durand extended Koornwinder's addition theorem to Jacobi (and other) functions of the second kind  in \cite{Durand79}.
The study of multi-summation addition theorems for Jacobi 
functions of the first and second kind seems not to have 
moved forward since 
the advances by 
Durand and by Flensted-Jensen and Koornwinder. 
In the remainder of this paper, we give the full 
multi-summation expansions of Jacobi functions of the 
second kind and bring the full theory of the expansions 
of Jacobi functions to a circle. However, there still 
remain open questions in the study, and we will return 
to these questions later.
\section{Preliminaries}

Throughout this paper we adopt the following set notations:
$\mathbb N_0:=\{0\}\cup\mathbb N=\{0, 1, 2, 3, \ldots\}$, and
we use the set $\mathbb C$ which represents the complex
numbers.
Jacobi functions (and their special cases such as Gegenbauer, associated Legendre and Ferrers functions) have representations given in terms of Gauss hypergeometric functions which can be defined in terms of an infinite series over ratios of 
shifted factorials (Pochhammer symbols). The shifted factorial
can be defined for $a\in\mathbb C$,
$n\in\mathbb N_0$ 
by \cite[\href{http://dlmf.nist.gov/5.2.E4}{(5.2.4)}, \href{http://dlmf.nist.gov/5.2.E5}{(5.2.5)}]{NIST:DLMF}
$
(a)_n:=(a)(a+1)\cdots(a+n-1).
$
The following ratio of two gamma functions \cite[\href{http://dlmf.nist.gov/5}{Chapter 5}]{NIST:DLMF} are related to 
the shifted factorial, namely
for $a\in\mathbb C\setminus-{\mathbb N}_0$, one
 has
\begin{equation}
(a)_n=\frac{\Gamma(a+n)}{\Gamma(a)},
\label{Pochdef}
\end{equation}
which allows one to extend the definition to non-positive 
integer values of $n$.
Some other properties of shifted factorial which we will use are
($n,k\in\mathbb N_0$, $n\ge k$)
\begin{eqnarray}
&&\hspace{-12.5cm}\Gamma(a-n)=\frac{(-1)^n\Gamma(a)}{(1-a)_n},\label{Poch1}\\
&&\hspace{-12.5cm}(-n)_k=\frac{(-1)^kn!}{(n-k)!}.\label{Poch2}
\end{eqnarray}
One also has the following
expression for the generalized 
binomial coefficient 
for $z\in\mathbb C$, $n\in\mathbb N_0$
\cite[\href{http://dlmf.nist.gov/1.2.E6}{(1.2.6)}]{NIST:DLMF}
\begin{equation}
\label{binomz}
\binom{z}{n}=\frac{(-1)^n (-z)_n}{n!}.
\end{equation}
Define the multisets ${\bf a}:=\{a_1,\ldots,a_r\}$, ${\bf b}:=\{b_1,\ldots,b_s\}$.
We will also use the common notational product convention,
$a_l\in\mathbb C$, $l\in\mathbb N$, $r\in\mathbb N_0$, e.g.,
\begin{eqnarray}
&&\hspace{-9.2cm}
({\bf a})_k:=(a_1,\ldots,a_r)_k:=(a_1)_k(a_2)_k\cdots(a_r)_k,\\[0.1cm]
&&\hspace{-9.2cm}\Gamma({\bf a}):=\Gamma(a_1,\ldots,a_r):=\Gamma(a_1)\cdots\Gamma(a_r).
\end{eqnarray}
Also define the multiset notation ${\bf a}+t:=\{a_1+t,\ldots,a_r+t\}$.

For any expression of the form
$(z^2-1)^\alpha$, we fix the branch of the power functions such that
\[
(z^2-1)^\alpha:=(z+1)^\alpha(z-1)^\alpha,
\]
for any fixed $\alpha\in{\mathbb C}$ and 
$z\in{\mathbb C}\setminus\{-1,1\}$.
The generalized hypergeometric
function \cite[\href{http://dlmf.nist.gov/16}{Chapter 16}]{NIST:DLMF} is defined by the 
infinite series \cite[\href{http://dlmf.nist.gov/16.2.E1}{(16.2.1)}]{NIST:DLMF}
\begin{equation} \label{genhyp}
{}_rF_s
({\bf a};{\bf b};z):=\hyp{r}{s}{{\bf a}}
{{\bf b}}{z}
:=
\sum_{k=0}^\infty
\frac{({\bf a})_k}
{({\bf b})_k}
\frac{z^k}{k!},
\end{equation}
where $|z|<1$, $b_j\not \in -\mathbb N_0$, for 
$j\in\{1, \dots, s\}$;
and elsewhere by analytic continuation.
Further define the Olver normalized (scaled or regularized) 
generalized hypergeometric series
${}_r\bm{F}_s({\bf a};{\bf b};z)$, given
by
\begin{equation}
\label{regrFs}
\hspace{-0.0cm}
{}_r\bm{F}_s
({\bf a};{\bf b};z):=
\Ohyp{r}{s}{{\bf a}}{{\bf b}}{z}:=
\frac{1}{\Gamma({\bf b})}\hyp{r}{s}{{\bf a}}{{\bf b}}{z}=
\sum_{k=0}^\infty \frac{(a_1,\ldots,a_r)_k}{\Gamma({\bf b}+k
)}
\frac{z^k}{k!},
\end{equation}
which is entire for all $a_l,b_j\in\mathbb C$, $l\in\{1,\ldots,r\}$, $j\in\{1,\ldots,s\}$.
Both the generalized and Olver normalized generalized hypergeometric series, if 
nonterminating, are entire if $r\le s$, convergent for $|z|<1$ if $r=s+1$ 
and divergent if $r\ge s+1$.
\medskip
The special case of the generalized hypergeometric function with $r=2$, $s=1$ is referred to as the Gauss hypergeometric function \cite[\href{http://dlmf.nist.gov/15}{Chapter 15}]{NIST:DLMF}, or simply the hypergeometric function. It has many interesting properties, including linear transformations which were discovered by Euler and Pfaff. 
Euler's linear transformation is \cite[\href{http://dlmf.nist.gov/15.8.E1}{(15.8.1)}]{NIST:DLMF}
\begin{equation}
\Ohyp21{a,b}{c}{z}=(1-z)^{c-a-b}\Ohyp21{c\!-\!a,c\!-\!b}{c}{z}
\label{Eulertran}
\end{equation}
and Pfaff's linear transformation is \cite[\href{http://dlmf.nist.gov/15.8.E1}{(15.8.1)}]{NIST:DLMF}
\begin{equation}
\Ohyp21{a,b}{c}{z}
=(1-z)^{-a}\Ohyp21{a,c\!-\!b}{c}{\frac{z}{z-1}}
=(1-z)^{-b}\Ohyp21{b,c\!-\!a}{c}{\frac{z}{z-1}}.\label{Pfafftran}
\end{equation}

\subsection{The Gegenbauer and associated Legendre functions}
The functions which satisfy quadratic transformations of the Gauss 
hypergeometric function are given by Gegenbauer and associated Legendre functions of the 
first and second kind. As we will see, these
functions correspond to Jacobi functions of the first and second kind when 
their parameters satisfy certain relations. We now describe some of the 
properties of these functions, which have a deep and long history.

Let $n\in\N_0$. The Gegenbauer (ultraspherical) polynomial which is an important
specialization of the Jacobi polynomial for
symmetric parameters values, is given
in terms of a terminating Gauss hypergeometric series
\cite[\href{http://dlmf.nist.gov/18.7.E1}{(18.7.1)}]{NIST:DLMF}
\begin{equation}\label{GegJac}
\hspace{0.0cm}C_n^\mu(z)=\frac{(2\mu)_n}{(\mu+\frac12)_n}
P_n^{(\mu-\frac12,\mu-\frac12)}(z)
=\frac{(2\mu)_n}
{n!}
\hyp21{-n,2\mu+n}{\mu+\frac12}{\frac{1-z}{2}}.
\end{equation}
Note that the ultraspherical polynomials
satisfy the following parity relation
\cite[\href{http://dlmf.nist.gov/18.6.T1}{Table 18.6.1}]{NIST:DLMF}
\begin{equation}
C_n^\mu(-z)=(-1)^n C_n^\mu(z).
\label{parGeg}
\end{equation}
Gegenbauer functions which generalize ultraspherical polynomials
with arbitrary degrees $n=\lambda\in\CC$ are solutions $w=w(z)=w_\lambda^\mu(z)$ to the 
Gegenbauer differential equation 
\cite[\href{http://dlmf.nist.gov/18.8.T1}{Table 18.8.1}]{NIST:DLMF}
\begin{equation}
\hspace{0.0cm}(z^2\!-\!1)\frac{\dd^2w(z)}{\dd z^2}+\left(2\lambda+1\right)z\frac{\dd w(z)}{\dd z}
-\lambda(\lambda+2\mu)w(z)=0.
\label{Gende}
\end{equation}
There are two linearly independent solutions to this second order ordinary differential equation which are referred to as Gegenbauer functions of the first and second  kind $C_\lambda^\mu(z)$, $D_\lambda^\mu(z)$.
A closely connected differential equation to the Gegenbauer differential equation \eqref{Gende} is the associated Legendre differential equation which is given by  \cite[\href{http://dlmf.nist.gov/14.2.E1}{(14.2.1)}]{NIST:DLMF}
\begin{equation}
\hspace{0.0cm}(1\!-\!z^2)\frac{\dd^2w(z)}{\dd z^2}-2z\frac{\dd w(z)}{\dd z}
+\left(\nu(\nu+1)-\frac{\mu^2}{1-z^2}\right)w(z)=0.
\label{Legde}
\end{equation}
Two linearly independent solutions to this equation are referred to as associated Legendre functions of the first and second kind $P_\nu^\mu(z)$, $Q_\nu^\mu(z)$. In the following subsection we will present the definitions of these important functions which are Gauss hypergeometric functions which satisfy a quadratic transformation.
\subsubsection{Hypergeometric representations of the Gegenbauer and associated Legendre functions}

The Gegenbauer function of the first kind is defined by \cite[\href{http://dlmf.nist.gov/15.9.E15}{(15.9.15)}]{NIST:DLMF}
\begin{equation}\label{Gegenbauerfuncdef}
\hspace{0.0cm}C_\lambda^\mu(z):=
\frac{\sqrt{\pi}\,\Gamma(\lambda+2\mu)}
{2^{2\mu-1}\Gamma(\mu)\Gamma(\lambda+1)}
\Ohyp21{-\lambda,2\mu+\lambda}{\mu+\frac12}{\frac{1-z}{2}},
\end{equation}
where $\lambda+2\mu\not\in-\N_0$.
{It is a} clear extension of the Gegenbauer
polynomial when the index is allowed to be 
a complex number as well as a non-negative integer.
Two representations
which will be useful for us in comparing to
the Jacobi function of the second kind are referred to as 
Gegenbauer functions of the second kind which have hypergeometric representations given 
with $\lambda+2\mu\not\in-\N_0$, \cite[(2.3)]{DurandFishSim}
\begin{eqnarray}
&&\hspace{-4.7cm}D_\lambda^\mu(z)
:={\frac{\expe^{i\pi\mu}
\Gamma(\lambda+2\mu)}{\Gamma(\mu)(2z)^{\lambda+2\mu}}\Ohyp21{\frac12\lambda+\mu,
\frac12\lambda+\mu+\frac12}{\lambda+\mu+1}{\frac{1}{z^2}}}
\\[0.05cm]
&&\hspace{-4.6cm}\hspace{1.1cm}
=
\frac{\expe^{i\pi\mu}2^{\lambda}\Gamma(\lambda+\mu+\frac12)
\Gamma(\lambda+2\mu)}{\sqrt{\pi}\,\Gamma(\mu)(z\!-\!1)^{\lambda+\mu+\frac12}
(z\!+\!1)^{\mu-\frac12}}\Ohyp21{\lambda\!+\!1,\lambda\!+\!\mu\!+\!\frac12}
{2\lambda\!+\!2\mu\!+\!1}{\frac{2}{1\!-\!z}},
\label{GegDhyper}
\end{eqnarray}
and in the second representation 
$\lambda+\mu+\frac12\not\in-\N_0$.
The equality of these two representations of the Gegenbauer function of the second kind follow from a quadratic transformation of the Gauss hypergeometric function from Group 3 to Group 1 in \cite[\href{http://dlmf.nist.gov/15.8.T1}{Table 15.8.1}]{NIST:DLMF}.
The associated Legendre function 
of the first kind  is defined as \cite[\href{http://dlmf.nist.gov/14.3.E6}{(14.3.6)} and \href{http://dlmf.nist.gov/14.21.i}{\S 14.21(i)}]{NIST:DLMF}
\begin{equation}
\hspace{0.2cm}P_\nu^\mu(z):=
\left(\frac{z+1}{z-1}\right)^{\frac12\mu}
\Ohyp21{-\nu,\nu+1}{1-\mu}{\frac{1-z}{2}},
\label{associatedLegendrefunctionP}
\end{equation}
where $|1-z|<2$, and elsewhere in $z$ by analytic continuation.
The associated Legendre function of the second kind
$Q_\nu^\mu:\mathbb C\setminus(-\infty,1]\to\mathbb C,$ $\nu+\mu\notin-\mathbb N$, has the following two single Gauss hypergeometric function representations 
\cite[\href{http://dlmf.nist.gov/14.3.E7}{(14.3.7)} and \href{http://dlmf.nist.gov/14.21}{\S 14.21}]{NIST:DLMF}, \cite[entry 24, p.~161]{MOS}, 
\begin{eqnarray}\label{for:lf2k}
&&\hspace{-3.0cm}\hspace{0.2cm}Q_\nu^\mu(z):=\frac{\sqrt{\pi}\,\expe^{i\pi\mu}\Gamma(\nu+\mu+1)
(z^2-1)^{\frac12\mu}}{2^{\nu+1}
z^{\nu+\mu+1}}\Ohyp21{\frac{\nu+\mu+1}{2},
\frac{\nu+\mu+2}{2}}{\nu+\frac32}{\frac{1}{z^2}}
\\[0.05cm]
&&\hspace{-2.7cm}\hspace{1.1cm}
=
\frac{2^\nu \expe^{i\pi\mu}\Gamma(\nu+1)\Gamma(\nu+\mu+1) (z+1)^{\frac12\mu}}
{(z-1)^{\frac12\mu+\nu+1}}
\Ohyp21{
\nu+1,\nu+\mu+1}{
2\nu+2}{\frac{2}{1-z}},
\label{Qdefntwodivide1mz}
\end{eqnarray}
and for the second representation, $\nu\not\in-\N$.
The first and second single Gauss hypergeometric representations are convergent as a Gauss hypergeometric series for
$|z|>1$, respectively $|z-1|>2$, and elsewhere in $z\in\C\setminus(-\infty,1]$ by analytic continuation
of the Gauss hypergeometric function.
\begin{rem}
The relations between the 
associated Legendre functions of the first and second kind 
are related
to the Gegenbauer functions of the first and second kind by \cite[\href{http://dlmf.nist.gov/14.3.E22}{(14.3.22)}]{NIST:DLMF}
\begin{eqnarray}
&&\hspace{-5.9cm}P_\nu^{\mu}(z)=\frac{\Gamma(\frac12-\mu)\Gamma(\nu+\mu+1)}
{2^\mu\sqrt{\pi}\,\Gamma(\nu-\mu+1)(z^2-1)^{\frac12\mu}}
C_{\nu+\mu}^{\frac12-\mu}(z),
\label{FerPtoGegC}\\
&&\hspace{-5.9cm}Q_\nu^\mu(z)=\frac{\expe^{2\pi i(\mu-\frac14)}\sqrt{\pi}\,\Gamma(\frac12-\mu)\Gamma(\nu+\mu+1)}{2^\mu\Gamma(\nu-\mu+1)(z^2-1)^{\frac12\mu}}
D_{\nu+\mu}^{\frac12-\mu}(z),
\end{eqnarray}
which are valid for $\mu\in\C\setminus\{\frac12,\frac32,\ldots\}$, $\nu+\mu\in\C\setminus-\N$.
Equivalently, the inverse relationships are given by 
\begin{eqnarray}
&&\hspace{-5.3cm} \label{ClammuFerPb}
C_\lambda^\mu(z)=\dfrac{\sqrt{\pi}\,\Gamma(\lambda+2\mu)}
{2^{\mu-\frac12}\Gamma(\mu)\Gamma(\lambda+1)(z^2-1)^{\frac{\mu}{2}-\frac14}}
P_{\lambda+\mu-\frac12}^{\frac12-\mu}(z),\\
&&\hspace{-5.3cm}D_\lambda^\mu(z)=\frac{\expe^{2\pi i(\mu-\frac14)}\Gamma(\lambda+2\mu)}
{\sqrt{\pi}\,2^{\mu-\frac12}\Gamma(\mu)\Gamma(\lambda+1)
(z^2-1)^{\frac12\mu-\frac14}}Q_{\lambda+\mu-\frac12}^{\frac12-\mu}(z),
\label{DQrel}
\end{eqnarray}
which are valid for all $\lambda+2\mu\in\C\setminus-\N_0$.
\end{rem}

\begin{rem}
By comparing Gauss hypergeometric representations of the various functions, one may express ${}_2\bm{F}_1(a,a+\frac12;c;z)$ in terms of associated Legendre functions of the first and second kind $P_\nu^\mu$, $Q_\nu^\mu$ and the Gegenbauer functions of the first and second kind $C_\nu^\mu$, $D_\nu^\mu$ using the following very useful formulas. Let
$z\in\mathbb C\setminus[1,\infty)$. Then
\begin{eqnarray}
&&\hspace{-4.4cm}\Ohyp21{a,a+\frac12}{c}{z}=
2^{c-1}z^{\frac12(1-c)}(1-z)^{\frac12c-a-\frac12}P_{2a-c}^{1-c}
\left(\frac{1}{\sqrt{1-z}}\right)\\
&&\hspace{-1.32cm}=\frac{2^{2c-2}\Gamma(c-\frac12)\Gamma(2 (a-c+1))}
{\sqrt{\pi}\,\Gamma(2a)(1-z)^a}C_{2a-2c+1}^{c-\frac12}
\left(\frac{1}{\sqrt{1-z}}\right),
\end{eqnarray}
where $2c\not\in\{1,-1,-3,\ldots\}$, $2a-2c\not\in\{-2,-3,\ldots\}$, and
\begin{eqnarray}
&&\hspace{-3.7cm}\Ohyp21{a,a+\frac12}{c}{z}=
\frac{\expe^{i\pi(c-2a-\frac12)}2^{c-\frac12}(1-z)^{\frac12c-a-\frac14}}
{\sqrt{\pi}\,\Gamma(2a)z^{\frac12c-\frac14}}Q_{c-\frac32}^{2a-c+\frac12}
\left(\frac{1}{\sqrt{z}}\right)\\
&&\hspace{-2.35cm}\hspace{1.75cm}={\frac{\expe^{i\pi(2a-c)}2^{2c-2a-1}
\Gamma(c-2a)(1-z)^{c-2a-\frac12}}{\Gamma(2c-2a-1)z^{c-a-\frac12}}
D_{2a-1}^{c-2a}\left(\frac{1}{\sqrt{z}}\right),}
\end{eqnarray}
where $c,c-2a\not\in-\mathbb N_0$.
\end{rem}

\subsubsection{The Gegenbauer functions on-the-cut $(-1,1)$ and the Ferrers Functions}

We will consider Jacobi functions of the second kind on-the-cut
in Section \ref{Jac2cut}.
As we will see, for certain combinations
of the parameters which we will describe below, the Jacobi functions 
of the first and second kind on the cut are related to the 
the Gegenbauer functions of the first and second kind on-the-cut and the associated Legendre 
functions of the first and second kind on-the-cut (Ferrers functions).

\medskip
The Gegenbauer functions of the first and second kind 
on-the-cut are defined in terms of the Gegenbauer functions immediately above and below the segment $(-1,1)$ in the complex plane. These definition are given by
\cite[(3.3), (3.4)]{Durand78}
\begin{eqnarray}
\label{GegCcutdef}
&&\hspace{-4.6cm}{\sf C}_\lambda^\mu(x):=D_\lambda^\mu(x+i0)+\expe^{-2\pi i\mu}D_\lambda^\mu(x-i0)=C_\lambda^\mu(x\pm i0),\quad x\in(-1,1]\\
&&\hspace{-4.6cm}
{\sf D}_\lambda^\mu(x):=-iD_\lambda^\mu(x+i0)+i\expe^{-2\pi i\mu}D_\lambda^\mu(x-i0),\quad x\in(-1,1).\label{GegDcutdef}
\end{eqnarray}
Note that ${\sf C}_\lambda^\mu(x)$ and ${\sf D}_\lambda^\mu$ are real for real values of $\lambda$ and $\mu$.

\medskip

The Ferrers functions of the first and second kind are defined
as  \cite[\href{http://dlmf.nist.gov/14.23.E1}{(14.23.1)}, \href{http://dlmf.nist.gov/14.23.E2}{(14.23.2)}]{NIST:DLMF}
\begin{eqnarray}
&&\hspace{-3.1cm}{\sf P}_\nu^\mu(x):=\expe^{\pm i\pi\mu}P_\nu^\mu(x\pm i0)=\frac{i\expe^{-i\pi\mu}}{\pi}\left(\expe^{-\frac12 i\pi\mu}Q_\nu^\mu(x+i0)-\expe^{\frac12 i\pi\mu}Q_\nu^\mu(x-i0)\right),
\label{FerrersPdef}\\
&&\hspace{-3.1cm}{\sf Q}_\nu^\mu(x):=\frac{\expe^{-i\pi\mu}}{2}\left(\expe^{-\frac12 i\pi\mu}Q_\nu^\mu(x+i0)+\expe^{\frac12 i\pi\mu}Q_\nu^\mu(x-i0)\right).
\label{FerrersQdef}
\end{eqnarray}

Using the above definition one can readily obtain a single hypergeometric representation of the Gegenbauer function of the first kind on-the-cut, namely
\begin{equation}\label{Gegenbauerfunccutdef}
\hspace{0.3cm}{\sf C}_\lambda^\mu(x)=
\frac{\sqrt{\pi}\,\Gamma(2\mu+\lambda)}
{2^{2\mu-1}\Gamma(\mu)\Gamma(\lambda+1)}
\Ohyp21{-\lambda,2\mu+\lambda}{\mu+\frac12}{\frac{1-x}{2}},
\end{equation}
which is identical to the Gegenbauer function of the first kind \eqref{Gegenbauerfuncdef} because this function analytically continues to the segment $(-1,1)$, see \eqref{GegCcutdef}.
For the Gegenbauer function of the second kind on-the-cut, one can readily obtain a double hypergeometric representation of by using the definition \eqref{GegDcutdef} and then using the interrelation between the Gegenbauer function of the second kind and the Legendre function of the second kind and then comparing to the Ferrers function of the second kind through its definition
\eqref{FerrersQdef}.

However, first we will give hypergeometric representations of the Ferrers function
of the first and second kind which are easily found in the literature. The first author recently co-authored a paper with Park and Volkmer where all double hypergeometric representations of the Ferrers function of the second kind were computed \cite{Cohletal2021}.
Using \eqref{FerrersPdef} one can derive hypergeometric representations of the Ferrers function of the first kind
(associated Legendre function of the first kind on-the-cut)
$\mathsf{P}_\nu^\mu:(-1,1)\to\mathbb C$. For instance,
one has a single hypergeometric representation given by 
\cite[\href{http://dlmf.nist.gov/14.3.E1}{(14.3.1)}]{NIST:DLMF}
\begin{equation}
\hspace{0.3cm}{\sf P}_\nu^\mu(x)=
\left(\frac{1+x}{1-x}\right)^{\frac12\mu}
\Ohyp21{-\nu,\nu+1}{1-\mu}{\frac{1\!-\!x}{2}}.
\label{FerrersPdefnGauss2F1}
\end{equation}
Let $\nu \in \C$, $\mu \in \C \setminus \Z$, $\nu+\mu\not\in-\N$, then a double hypergeometric representation of the Ferrers function of the second kind is given by
\cite[\href{http://dlmf.nist.gov/14.3.E2}{(14.3.2)}]{NIST:DLMF}
\begin{eqnarray}
&&\hspace{-3.4cm}{\sf Q}_\nu^\mu (x) = \frac{\pi}{2 \sin(\pi\mu)} \Bigg( \cos(\pi\mu)
{\left( \frac{1+x}{1-x} \right)}^{{\frac12\mu}}
\Ohyp21{-\nu, \nu+1}{1 - \mu}{\frac{1-x}{2}} \nonumber\\
&&\hspace{-3.0cm}\hspace{3cm}- \frac{\Gamma(\nu+\mu+1)}{\Gamma(\nu-\mu+1)}
{\left( \frac{1-x}{1+x} \right)}^{{\frac12\mu}}
\Ohyp21{-\nu, \nu+1}{1+\mu}{\frac{1-x}{2}} \Bigg).
\label{QcuthyprepA}
\end{eqnarray}

\begin{lem}
Let $x\in\C\setminus((-\infty,1]\cup(1,\infty))$, $\lambda,\nu,\mu\in\C$. Then
\begin{eqnarray}
&&\hspace{-5cm}{\sf D}_\lambda^\mu(x)=
\frac{\Gamma(\lambda+2\mu)}{2^{\mu-\frac32}\sqrt{\pi}\,\Gamma(\mu)\Gamma(\lambda+1)(1-x^2)^{\frac12\mu-\frac14}}
{\sf Q}_{\lambda+\mu
-\frac12}^{\frac12-\mu}(x),\label{relDQcut},
\end{eqnarray}
such that $\lambda+2\mu\not\in-\N_0$ and
\begin{eqnarray}
&&\hspace{-6.7cm}{\sf Q}_\nu^\mu(x)=\frac{\sqrt{\pi}\,\Gamma(\frac12-\mu)\Gamma(\nu+\mu+1)}
{2^{\mu+1}\Gamma(\nu-\mu+1)(1-x^2)^{\frac12\mu}}
{\sf D}_{\nu+\mu}^{\frac12-\mu}(x),
\label{relQDcut}
\end{eqnarray}
such that $\mu\not\in\{\frac12,\frac32,\ldots\}$ and $\nu+\mu\not\in-\N$.
\end{lem}
\begin{proof}
Start with the definition \eqref{GegDcutdef} and use the interrelation between the Gegenbauer function of the second kind on-the-cut and the Ferrers function of the second kind 
\eqref{DQrel}. Then applying this relation to the double hypergeometric representation 
completes the proof.
\end{proof}

\begin{thm}Let $x\in\C\setminus((-\infty,1]\cup(1,\infty))$, $\lambda,\mu\in\C$, such that $\lambda+2\mu\not\in-\N_0$. Then
\begin{eqnarray}
&&\hspace{-0.5cm}{\sf D}_\lambda^\mu(x)=\frac{\sqrt{\pi}}{\cos(\pi\mu)2^{\mu-\frac12}\Gamma(\mu)}
\Biggl(\frac{\sin(\pi\mu)\Gamma(\lambda+2\mu)}{\Gamma(\lambda+1)(1+x)^{\mu-\frac12}}\Ohyp21{\lambda+\mu+\frac12,\frac12-\lambda-\mu}{\frac12+\mu}{\frac{1-x}{2}}\nonumber\\
&&\hspace{6cm}-\frac{1}{(1-x)^{\mu-\frac12}}\Ohyp21{\lambda+\mu+\frac12,\frac12-\lambda-\mu}{\frac32-\mu}{\frac{1-x}{2}}\Biggr).
\end{eqnarray}
\end{thm}
\begin{proof}
Start with the definition \eqref{GegDcutdef} and use the interrelation between the Gegenbauer function of the second kind and the Legendre function of the second kind 
\eqref{DQrel}. Then comparing with the double hypergeometric representation given by \eqref{QcuthyprepA} completes the proof. 
\end{proof}

Note that we also have 
interrelation between 
the Ferrers function of the first kind 
and the Gegenbauer function of the first kind on-the-cut 
\cite[\href{http://dlmf.nist.gov/14.3.E21}{(14.3.21)}]{NIST:DLMF}
\begin{eqnarray} \label{relFerPGeg}
\hspace{-7.5cm}
&&{\sf P}_\nu^{\mu}(x)=\frac{\Gamma(\frac12-\mu)\Gamma(\nu+\mu+1)}
{2^\mu\sqrt{\pi}\,\Gamma(\nu-\mu+1)(1-x^2)^{\frac12\mu}}
{\sf C}_{\nu+\mu}^{\frac12-\mu}(x),
\end{eqnarray}
where $\mu\not\in\{\frac12,\frac32,\ldots\}$, $\nu+\mu\not\in- \mathbb N$,
or 
equivalently 
\begin{eqnarray} \label{ClammuFerP}
\hspace{-7.0cm}{\sf C}_\lambda^\mu(x)=
\dfrac{\sqrt{\pi}\,\Gamma(\lambda+2\mu)}
{2^{\mu-\frac12}\Gamma(\mu)\Gamma(\lambda+1)(1-x^2)^{\frac{1}{2}\mu-\frac14}}
{\sf P}_{\lambda+\mu-\frac12}^{\frac12-\mu}(x).
\end{eqnarray}
Finally we should add that the Legendre polynomial (the associated Legendre 
function of the first kind $P_\nu^\mu$ and the Ferrers 
function of the first kind ${\sf P}_\nu^\mu$ with $\mu=0$ 
and $\nu=n\in\mathbb Z$) is given by \cite[\href{http://dlmf.nist.gov/18.7.E9}{(18.7.9)}]{NIST:DLMF}
\[
\hspace{0.45cm}P_n(x):=P_n^0(x)={\sf P}_n^0(x)=C_n^{\frac12}(x)=P_n^{(0,0)}(x),
\]
which vanishes for $n$ negative.

\subsection{Brief introduction to Jacobi functions of the first and second kind}
Now we will discuss fundamental properties and special 
values and limits for the Jacobi functions.
Jacobi functions are
complex solutions $w=w(z)=w_\gamma^{(\alpha,\beta)}(z)$ 
to the Jacobi differential equation
\cite[\href{http://dlmf.nist.gov/18.8.T1}{Table 18.8.1}]{NIST:DLMF}
\begin{equation}
(1-z^2)\frac{\dd^2w}{\dd z^2}+\left(\beta-\alpha-z(\alpha+\beta+2)\right)
\frac{\dd w}{\dd z}
+\gamma(\alpha+\beta+\gamma+1)w=0,
\label{Jacde}
\end{equation}
which is a second order linear homogeneous differential equation.
Solutions to this differential equation satisfy the following three-term recurrence relation \cite[(10.8.11), p.~169]{ErdelyiHTFII}
\begin{equation}
B_\gamma^{(\alpha,\beta)}w_{\gamma}^{(\alpha,\beta)}(z)+A_\gamma^{(\alpha,\beta)} (z)w_{\gamma+1}^{(\alpha,\beta)}(z)+w_{\gamma+2}^{(\alpha,\beta)}(z)=0,
\label{ttrl}
\end{equation}
where
\begin{eqnarray}
&&\hspace{-3.0cm}A_\gamma^{(\alpha,\beta)}(z)=-\frac{(\alpha+\beta+2\gamma+3)\left(\alpha^2-\beta^2+
(\alpha+\beta+2\gamma+2)
(\alpha+\beta+2\gamma+4)
z\right)}
{2(\gamma+2)(\alpha+\beta+\gamma+2)(\alpha+\beta+2\gamma+2)},\\
&&\hspace{-3.0cm}B_\gamma^{(\alpha,\beta)}=\frac{(\alpha+\gamma+1)(\beta+\gamma+1)(\alpha+\beta+2\gamma+4)}
{(\gamma+2)(\alpha+\beta+\gamma+2)(\alpha+\beta+2\gamma+2)}.
\end{eqnarray}
This three-term recurrence relation is very useful for deriving various solutions to \eqref{Jacde} when solutions are known for values which have integer separations.

\subsubsection{The Jacobi function of the first kind}
{The Jacobi function of the first kind
is a generalization of the Jacobi polynomial (as given by \eqref{Jacobipolydef}) where 
the degree is no longer restricted to be an integer. 
In the following material we derive properties 
for the Jacobi function of the first kind.}
{In the following result we present the four 
single Gauss hypergeometric function representations 
of the Jacobi function of the first kind.}

\begin{thm}
\label{Firstthm}
Let
$\alpha,\beta,\gamma\in\mathbb C$ such that
$\alpha+\gamma\not\in-\mathbb N$. Then, 
the Jacobi function of the first kind 
$P_\gamma^{(\alpha,\beta)}:\mathbb C
\setminus(-\infty,-1]\to\mathbb C$ 
can be defined by
\begin{eqnarray}
&&\hspace{-1.8cm}P_\gamma^{(\alpha,\beta)}(z)=
\frac{\Gamma(\alpha+\gamma+1)}{
\Gamma(\gamma+1)}
\Ohyp21{-\gamma,\alpha+\beta+\gamma+1}
{\alpha+1}{\frac{1-z}{2}}
\label{Jac1}\\
&&\hspace{-0.2cm}=\frac{\Gamma(\alpha+\gamma+1)}{
\Gamma(\gamma+1)}
\left(\frac{2}{z+1}\right)^\beta\Ohyp21
{-\beta-\gamma,\alpha+\gamma+1}{\alpha+1}
{\frac{1-z}{2}}\label{Jac2}\\
&&\hspace{-0.2cm}=\frac{\Gamma(\alpha+\gamma+1)}
{
\Gamma(\gamma+1)}
\left(\frac{z+1}{2}\right)^\gamma\Ohyp21
{-\gamma,-\beta-\gamma}
{\alpha+1}{\frac{z-1}{z+1}}\label{Jac3}\\
&&\hspace{-0.2cm}=\frac{\Gamma(\alpha+\gamma+1)}
{
\Gamma(\gamma+1)}
\left(\frac{2}{z+1}\right)^{\alpha+\beta+\gamma+1}\Ohyp21
{\alpha+\gamma+1,\alpha+\beta+\gamma+1}{\alpha+1}
{\frac{z-1}{z+1}}.\label{Jac4}
\end{eqnarray}
\label{bigPthm}
\end{thm}

\begin{proof}
Start with \eqref{Jacobipolydef} and replace 
the shifted factorial by a ratio of gamma 
functions using \eqref{Pochdef},
the factorial $n!=\Gamma(n+1)$ and substitute 
$n\mapsto\gamma\in\mathbb C$, 
$x\mapsto z$. Application of Euler's transformation \eqref{Eulertran} and Pfaff's transformation \eqref{Pfafftran} provides the 
other three single hypergeometric representations. This completes the proof.
\end{proof}
There exist double Gauss hypergeometric 
representations of the Jacobi function of the 
first kind which can be obtained by using 
the linear transformation formulas for the Gauss hypergeometric function 
$z\mapsto z^{-1}$, $z\mapsto (1-z)^{-1}$, $z\mapsto 1-z$, $z\mapsto 1-z^{-1}$
\cite[\href{http://dlmf.nist.gov/14.3.E1}{(14.3.1)}--\href{http://dlmf.nist.gov/14.3.E5}{(14.3.5)}]{NIST:DLMF},
respectively. 
However, these in general 
will be given in terms of a sum of two Gauss 
hypergeometric functions. 
{We will} will not present {the} double hypergeometric representations of 
the Jacobi function of the first kind here.
\medskip

One has the following connection relation for the Jacobi function of the first kind.

\begin{cor}Let $\gamma,\alpha,\beta\in\CC$, $z\in\CC\setminus(-\infty,1]$, $\gamma\not\in-\N$, $\beta+\gamma\not\in\N_0$. Then
\begin{eqnarray}
&&\hspace{-6.6cm}P_{-\gamma-\alpha-\beta-1}^{(\alpha,\beta)}(z)=
\frac{\Gamma(-\beta-\gamma)\Gamma(\gamma+1)}{\Gamma(-\gamma-\alpha-\beta)
\Gamma(\alpha+\gamma+1)}P_\gamma^{(\alpha,\beta)}(z).\label{WolfP}
\end{eqnarray}
\end{cor}
\begin{proof}
This connection relation
can be derived by using \eqref{Jac1} and making the replacement 
$\gamma\mapsto-\gamma-\alpha-\beta-1$ 
which 
leaves the parameters and argument of the hypergeometric function unchanged. Comparing the prefactors completes the proof. 
\end{proof}

\begin{rem}
One of the consequences of the definition of the 
Jacobi function of the first kind is the following
special value:
\begin{equation}
P_\gamma^{(\alpha,\beta)}(1)
=\frac{\Gamma(\alpha+\gamma+1)}
{\Gamma(\alpha+1)\Gamma(\gamma+1)},
\label{Jacone}
\end{equation}
where $\alpha+\gamma\not\in-\mathbb N$.
For $\gamma=n\in\mathbb Z$ one has
\begin{equation}
P_n^{(\alpha,\beta)}(1)=\frac{(\alpha+1)_n}{n!},\quad 
P_n^{(\alpha,\beta)}(-1)=(-1)^n\frac{(\beta+1)_n}{n!},
\end{equation}
which is consistent with \eqref{Jacone}
and the parity relation for Jacobi polynomials 
(see \cite[\href{http://dlmf.nist.gov/18.6.T1}{Table 18.6.1}]{NIST:DLMF}).
From \eqref{Jac1} we have
\begin{equation}
P_0^{(\alpha,\beta)}(z)=1,
\label{Pzero}
\end{equation}
and $P_k^{(\alpha,\beta)}(z)=0$ for all $k\in-\mathbb N$.
\end{rem}
\medskip
\subsubsection{The Jacobi function of the second kind}
The Jacobi function of the second kind $Q_\gamma^{(\alpha,\beta)}(z)$, $\gamma\in\CC$
is a generalization of the Jacobi function
of the second kind $Q_n^{(\alpha,\beta)}(z)$, $n\in\mathbb N_0$ (as given by \cite[(10.8.18)]{ErdelyiHTFII}), where the degree is no longer restricted to be an integer. 
In the following material we
derive properties for the Jacobi function of the second kind.
Below we give the four single Gauss hypergeometric function representations 
of the Jacobi function of the second kind.
\begin{thm}\label{thmQ}
Let $\gamma,\alpha,\beta,z\in\mathbb C$ such that 
$z\in\mathbb C\setminus[-1,1]$,
$\alpha+\gamma,\beta+\gamma\notin-\mathbb N$.
Then, the Jacobi function of the second kind
has the following Gauss hypergeometric representations
\begin{eqnarray}
&&\hspace{-1.4cm}Q_\gamma^{(\alpha,\beta)}(z) :=
\frac{2^{\alpha+\beta+\gamma}\Gamma(\alpha+\gamma+1)\Gamma(\beta+\gamma+1)}
{(z-1)^{\alpha+\gamma+1}(z+1)^\beta}
\Ohyp21{\gamma+1,\alpha+\gamma+1}{\alpha+\beta+2\gamma+2}
{\frac{2}{1-z}}
\label{dJsk1}\\[0.2cm]
&&\hspace{0.0cm}=
\frac{2^{\alpha+\beta+\gamma}\Gamma(\alpha+\gamma+1)\Gamma(\beta+\gamma+1)}
{(z-1)^{\alpha+\beta+\gamma+1}}
\Ohyp21{\beta+\gamma+1,\alpha+\beta+\gamma+1}{\alpha+\beta+2\gamma+2}
{\frac{2}{1-z}}\label{dJsk4}\\[0.2cm]
&&\hspace{0.0cm}=
\frac{2^{\alpha+\beta+\gamma}\Gamma(\alpha+\gamma+1)\Gamma(\beta+\gamma+1)}
{(z-1)^{\alpha}(z+1)^{\beta+\gamma+1}}
\Ohyp21{\gamma+1,\beta+\gamma+1}{\alpha+\beta+2\gamma+2}
{\frac{2}{1+z}}\label{dJsk3}\\[0.2cm]
&&\hspace{0.0cm}=
\frac{2^{\alpha+\beta+\gamma}\Gamma(\alpha+\gamma+1)\Gamma(\beta+\gamma+1)}
{(z+1)^{\alpha+\beta+\gamma+1}}
\Ohyp21{\alpha+\gamma+1,\alpha+\beta+\gamma+1}{\alpha+\beta+2\gamma+2}
{\frac{2}{1+z}}.
\label{dJsk2}
\end{eqnarray}
\end{thm}
\begin{proof}
Start with \cite[(10.8.18)]{ErdelyiHTFII} and
let $n\mapsto\gamma\in\mathbb C$ and $x\mapsto z$.
Application of Pfaff's $(z\mapsto z/(z-1))$ and 
Euler's $(z\mapsto z)$ transformations
\cite[\href{http://dlmf.nist.gov/15.8.E1}{(15.8.1)}]{NIST:DLMF} provides the other 
three representations. This completes the proof.
\end{proof}

One has the following 
connection relation between Jacobi functions of the first kind and Jacobi functions of the second kind.

\begin{cor}
Let $\gamma,\alpha,\beta\in\CC$, $z\in\CC\setminus(-\infty,1]$, $\alpha+\gamma,\beta+\gamma\not\in-\N$, $\alpha+\beta+2\gamma\not\in\Z$. Then
\begin{eqnarray}
&&\hspace{-0.7cm}P_\gamma^{(\alpha,\beta)}(z)=
\frac{-2\sin(\pi(\beta+\gamma))}{\pi \sin(\pi(\alpha+\beta+2\gamma+1))}\nonumber\\
&&\hspace{1.3cm}\times\left(
\sin(\pi\gamma)
Q_\gamma^{(\alpha,\beta)}(z)
-\sin(\pi(\alpha+\gamma))\frac{\Gamma(\alpha+\gamma+1)\Gamma(\beta+\gamma+1)}{\Gamma(\gamma+1)\Gamma(\alpha+\beta+\gamma+1)}Q_{-\alpha-\beta-\gamma-1}^{(\alpha,\beta)}(z)\right).
\label{Durconrel}
\end{eqnarray}
\end{cor}
\begin{proof}
This can be derived by starting with \eqref{Jac1}, applying the linear transformation \cite[\href{http://dlmf.nist.gov/15.8.E2}{(15.8.2)}] {NIST:DLMF} $z\mapsto z^{-1}$ and then comparing twice with Theorem \ref{thmQ}.
\end{proof}

\begin{rem}
Using \eqref{Durconrel} one can see that
for $\gamma=n\in\N_0$, that $Q_{-\alpha-\beta-\gamma-1}^{(\alpha,\beta)}(z)$ is a Jacobi polynomial, namely
\begin{eqnarray}
&&\hspace{-3.2cm}Q_{-\alpha-\beta-1-n}^{(\alpha,\beta)}(z)=
\frac{\Gamma(-\alpha)\Gamma(-\beta)}{2\Gamma(-\alpha-\beta)}
\frac{n!(\alpha+\beta+1)_n}{(\alpha+1)_n(\beta+1)_n}
P_n^{(\alpha,\beta)}(z)\nonumber\\
&&\hspace{-0.6cm}
=-\frac{\pi}{2}
\frac{\sin(\pi(\alpha+\beta))}{\sin(\pi\alpha)\sin(\pi\beta)}
\frac{n!\,\Gamma(\alpha+\beta+1+n)}{\Gamma(\alpha+1+n)\Gamma(\beta+1+n)}
P_n^{(\alpha,\beta)}(z).
\label{QtoP}
\end{eqnarray}
\end{rem}

\begin{rem}
From Theorem 
\ref{thmQ} one can derive the following special values for $Q_{-1}^{(\alpha,\beta)}(z)$and $Q_{0}^{(\alpha,\beta)}(z)$, namely
\begin{eqnarray}
\label{JF2Kmone}
&&\hspace{-3.50cm}Q_{-1}^{(\alpha,\beta)}(z) =
\frac{2^{\alpha+\beta-1}\Gamma(\alpha)\Gamma(\beta)}
{\Gamma(\alpha+\beta)(z-1)^{\alpha}(z+1)^\beta},\\
&&\hspace{-3.50cm}Q_0^{(\alpha,\beta)}(z) =
\frac{2^{\alpha+\beta}\Gamma(\alpha+1)\Gamma(\beta+1)}
{(z+1)^{\alpha+\beta+1}}
\Ohyp21{\alpha+1,\alpha+\beta+1}{\alpha+\beta+2}
{\frac{2}{1+z}}.
\end{eqnarray}
Using the three-term recurrence relation \eqref{ttrl} one can derive values of the Jacobi function of the second kind at all negative integer values. For instance, one can derive
\begin{eqnarray}
\label{JF2Kmtwo}
&&\hspace{-3.50cm}Q_{-2}^{(\alpha,\beta)}(z) =
\frac{2^{\alpha+\beta-2}\Gamma(\alpha-1)\Gamma(\beta-1)}
{\Gamma(\alpha+\beta-1)(z-1)^{\alpha}(z+1)^\beta}(\alpha-\beta+(\alpha+\beta-2)z),
\end{eqnarray}
and also expressions for Jacobi functions of the second kind with further negative integer values of $\gamma$.
\end{rem}

If one examines the Gauss hypergeometric representations presented in Theorem \ref{thmQ} one can see that they are not defined for certain values of $\gamma$, $\alpha$, $\beta$ since we must avoid $\alpha+\gamma$ and $\beta+\gamma$ being a negative integer. In fact, these singularities are removable and one is able to compute the values of these Jacobi functions.  One can evaluate the Jacobi function of the second kind when the parameters $\alpha$, $\beta$,  and degree $\gamma$ is a non-negative integer in the following result{,} which was inspired by the work in \cite{WimpMcCabeConnor97}.

\begin{thm}
\label{Qnotcutsumthm}
Let $n,a,b\in\N_0$, $z\in\CC\setminus[-1,1]$. Then
\begin{eqnarray}
&&\hspace{-1cm} Q_n^{(a,b)}(z)=\frac{(-1)^{a+n}}{2^{n+1}}
\sum_{\substack{k=0\\k\ne n}}^{a+b+2n}\frac{(-2)^k}{(n-k)}
\left((z+1)^{n-k}-(z-1)^{n-k}\right)
P_k^{(a+n-k,b+n-k)}(z)\nonumber\\
&&\hspace{6cm}+\frac{(-1)^a}{2}\log\left(\frac{z+1}{z-1}\right)P_n^{(a,b)}(z).
\end{eqnarray}
\end{thm}

\begin{proof}
Start with the integral
representation for the Jacobi function of the second kind
\cite[(4.61.1)]{Szego}
\begin{equation}
\hspace{0.9cm}Q_\gamma^{(\alpha,\beta)}(z)=\frac{1}{2^{\gamma+1}(z-1)^\alpha(z+1)^\beta}\int_{-1}^{1}\frac{(1-t)^{\alpha+\gamma}(1+t)^{\beta+\gamma}}{(z-t)^{\gamma+1}}{\mathrm d}t,
\label{intrepQ}
\end{equation}
provided $\Re(\alpha+\gamma)$, $\Re(\beta+\gamma)>-1$ \cite[(2.5)]{WimpMcCabeConnor97}
and 
identify $(\gamma,\alpha,\beta)=(n,a,b)\in\N_0^3$.
Then consider 
\begin{eqnarray}
&&\hspace{-3.4cm}\mu_{n,k}^{(a,b)}(z):=
\frac{\dd^k}{\dd z^k}(1-z)^{n+a}(1+z)^{n+b}\nonumber\\[0.1cm]
&&\hspace{-1.8cm}=(-1)^k2^kk!(1-z)^{a+n-k}(1+z)^{b+n-k}P_k^{(a+n-k,b+n-k)}(z),\nonumber
\end{eqnarray}
where we have used the Rodrigues-type formula for Jacobi polynomials
\cite[\href{http://dlmf.nist.gov/18.5.T1}{Table 18.5.1}]{NIST:DLMF}.
It is easy to show that
\begin{equation}
\hspace{1cm}(1-t)^{n+a}(1+t)^{n+b}=\sum_{k=0}^{2n+a+b}
\mu_{n,k}^{(a,b)}(z)\frac{(t-z)^k}{k!},
\label{sumit}
\end{equation}
and the right-hand side is valid for all $z\in\C$.
Now start with \eqref{intrepQ}  and insert  \eqref{sumit} into the integrand and perform the integration over $t\in(-1,1)$ using
\[
\hspace{1cm}\int_{-1}^1 (z-t)^{k-n-1}\,\dd t
=
\left\{ \begin{array}{ll}
\displaystyle \frac{(z+1)^{k-n}-(z-1)^{k-n}}{k-n} & \qquad\mathrm{if}\ k\ne n,\\[10pt]
\displaystyle \log\left(\frac{z+1}{z-1}\right) & \qquad\mathrm{if}\ k=n,
\end{array} \right.
\]
which completes the proof.
\end{proof}

By using \eqref{dJsk3} we find that
if $|z|\sim 1+\epsilon$ then as $\epsilon\to0^{+}$ one has the following behavior of the Jacobi function of the second kind near the singularity at $z=1$, namely
\begin{equation}
\hspace{1cm}Q_\gamma^{(\alpha,\beta)}(1+\epsilon)\sim\frac{2^{\alpha-1}\Gamma(\alpha)\Gamma(\beta+\gamma+1)}{\Gamma(\alpha+\beta+\gamma+1)\epsilon^\alpha},
\label{Qnear1}
\end{equation}
where $\Re\alpha>0$, $\beta+\gamma\not\in-\N$. 
By using \eqref{dJsk3} we
see that as $|z|\to\infty$ one has 
\begin{equation}
\label{Qinfty}
\hspace{1cm}Q_\gamma^{(\alpha,\beta)}(z)\sim\frac{2^{\alpha+\beta+\gamma}\Gamma(\alpha+\gamma+1)\Gamma(\beta+\gamma+1)}{\Gamma(\alpha+\beta+2\gamma+2)z^{\alpha+\beta+\gamma+1}},
\end{equation}
where $\alpha+\gamma+1,\beta+\gamma\not\in-\N$.

\subsubsection{Jacobi functions of the first and second kind on-the-cut}
\label{Jac2cut}
\medskip
We now refer to the real segment $(-1,1)$ as the cut and the Jacobi functions 
of the first and second kind on-the-cut as ${\sf P}_\gamma^{(\alpha,\beta)},
{\sf Q}_\gamma^{(\alpha,\beta)}$.
The natural definitions of these Jacobi functions 
are due to Durand and can be found in 
\cite[(2.3), (2.4)]{Durand78} (see also 
\cite{Askeyetal86}). These are given as follows:
\begin{eqnarray}
\label{JacPcutdef}
&&\hspace{-2.1cm}{\sf P}_\gamma^{(\alpha,\beta)}(x)
:=\frac{i}{\pi}\left(\expe^{i\pi\alpha}
Q_\gamma^{(\alpha,\beta)}(x+i0)
-\expe^{-i\pi\alpha}Q_\gamma^{(\alpha,\beta)}
(x-i0)\right)
=P_\gamma^{(\alpha,\beta)}(x\pm i0),
\label{Pcutdef}\\
&&\hspace{-2.1cm} {\sf Q}_\gamma^{(\alpha,\beta)}(x)
:=\frac12\left(\expe^{i\pi\alpha}
Q_\gamma^{(\alpha,\beta)}(x+i0)
+\expe^{-i\pi\alpha}Q_\gamma^{(\alpha,\beta)}(x-i0)
\right).
\label{Qcutdef}
\end{eqnarray}
Note that the Jacobi function of the first kind 
on-the-cut \eqref{Pcutdef} is simply an analytic 
continuation of the Jacobi function of the first 
kind (see Theorem \ref{Firstthm}) since 
the complex-valued function is continuous across 
the real interval $(-1,1]$. On the other hand, 
the Jacobi function of the second kind is not an 
analytic continuation of the Jacobi function of 
the second kind
(see Theorem \ref{thmQ}). This is because 
$Q_\gamma^{(\alpha,\beta)}$ is not continuous 
across the real interval $(-1,1)$.
Hence, an `average' \eqref{Qcutdef} must be taken of the 
function values with infinitesimal positive and negative 
arguments in order to define it. Originally,
in Szeg\H{o}'s book \cite[\S4.62.9]{Szego} (see also 
\cite[(10.8.22)]{ErdelyiHTFII}) a definition for 
the Jacobi function of the second kind on-the-cut was 
given by ${\sf Q}_\gamma^{(\alpha,\beta)}(x):=\frac12
\left( Q_\gamma^{(\alpha,\beta)}(x+i0)
+Q_\gamma^{(\alpha,\beta)}(x-i0)\right)$, but as is pointed 
out by Durand \cite{Durand78}, Szeg\H{o}'s definition 
destroys the analogy between 
${\sf P}_\gamma^{(\alpha,\beta)}(\cos\theta)$, 
${\sf Q}_\gamma^{(\alpha,\beta)}(\cos\theta)$ and the 
trigonometric functions. Hence with the updated Durand 
definitions for the Jacobi functions of the first and 
second kind on-the-cut \eqref{Pcutdef}, \eqref{Qcutdef}, 
one has the following asymptotics as $n\to\infty$, {namely \cite[p.~77]{Durand78}}
\[
\hspace{0.8cm}Q_n^{(\alpha,\beta)}(\cos\theta
\pm i0)\sim\frac12\left(\frac{\pi}{n}\right)^\frac12
\left(\sin(\tfrac12\theta)\right)^{-\alpha-\frac12}
\left(\cos(\tfrac12\theta)\right)^{-\beta-\frac12}
\expe^{\mp iN\theta\mp i\frac{\pi}{2}(\alpha+\frac12)},
\]
where $N:=n+\frac12\alpha+\frac12\beta+\frac12$.

There are many double hypergeometric
representations of the Jacobi function
of the second kind on-the-cut 
${\sf Q}_\gamma^{(\alpha,\beta)}:\mathbb C\setminus((-\infty,1]
\cup[1,\infty))\to\mathbb C$. 
These hypergeometric representations follow 
by applying the definition \eqref{Qcutdef} to 
Theorem \ref{thmQ} which provides the Gauss hypergeometric 
representations for the Jacobi function of the second kind. 
The application of \eqref{Qcutdef} takes the argument of the
Gauss hypergeometric functions just above and below the ray 
$(1,\infty)$
in which it is known that the Gauss hypergeometric 
function is discontinuous.
The values of the Gauss hypergeometric function $z$ 
above and below this ray may then be transformed 
into a region where the Gauss hypergeometric function 
is continuous in a complex neighborhood of the argument 
of the Gauss hypergeometric function by utilizing the 
transformations which one can find in 
\cite[Appendix B]{Cohletal2021}.
These transformations, which map from Gauss hypergeometric 
functions with argument $x\pm i0$ to sums of Gauss 
hypergeometric functions with arguments given by 
$1/x$, $1-x$, $1-x^{-1}$ and $(1-x)^{-1}$. 
Eight Gauss hypergeometric function representations of 
the Jacobi function of the second kind on-the-cut can 
be obtained by starting with \eqref{dJsk1}-\eqref{dJsk2}, 
applying the transformation \cite[Theorem B.1]{Cohletal2021} 
$z\mapsto z^{-1}$ and by either utilizing 
the Euler \eqref{Eulertran} or Pfaff \eqref{Pfafftran} 
transformations as needed. 
There are certainly more Gauss hypergeometric 
representations that can be obtained for the Jacobi 
function of the second kind on-the-cut by applying 
\cite[Theorems B.2--B.4]{Cohletal2021}, but the derivation 
of these representations must be left to a later publication. 
{We} will give two of these here for
$\gamma,\alpha,\beta\in\mathbb C$ such that 
$\alpha, \beta\not\in\mathbb Z$, $\alpha+\gamma,\beta
+\gamma\not\in-\mathbb N$, namely
\begin{eqnarray}
&&\hspace{-0.4cm}
{\sf Q}_\gamma^{(\alpha,\beta)}(x)=
\frac{\pi}{2\sin(\pi\alpha)}
\Biggl(-
\cos(\pi\alpha)\frac{\Gamma(\alpha+\gamma+1)}
{\Gamma(\gamma+1)}
\Ohyp21{-\gamma,\alpha+\beta+\gamma+1}{1+\alpha}
{\frac{1\!-\!x}{2}}\nonumber\\[-0.0cm]
&&\hspace{3.7cm}
+\frac{\Gamma(\beta+\gamma+1)}
{\Gamma(\alpha+\beta+\gamma+1)}
\left(\frac{2}{1\!-\!x}\right)^\alpha
\left(\frac{2}{1\!+\!x}\right)^\beta
\Ohyp21{-\alpha-\beta-\gamma,\gamma+1}
{1-\alpha}{\frac{1\!-\!x}{2}}
\Biggr)
\label{Qcut1}\\
&&\hspace{1.25cm}=
\frac{\pi}{2^{\gamma+1}\sin(\pi\alpha)}
\biggl(-\cos(\pi\alpha)\frac{\Gamma(\alpha+\gamma+1)}{\Gamma(\gamma+1)}
(1+x)^\gamma
\Ohyp21{-\gamma,-\beta-\gamma}{1+\alpha}{\frac{x-1}{x+1}}\nonumber\\
&&\hspace{3.7cm}+\frac{\Gamma(\beta+\gamma+1)}{\Gamma(\alpha+\beta+\gamma+1)}
\frac{(1\!+\!x)^{\alpha+\gamma}}{(1\!-\!x)^\alpha}
\Ohyp21{-\alpha-\beta-\gamma,-\alpha-\gamma}{1-\alpha}{\frac{x-1}{x+1}}\biggr).
\end{eqnarray}

Just as we were able to compute the values of the Jacobi function of the second kind with non-negative integer parameters and degree, {the same evaluation can be accomplished} for the Jacobi function of the second kind on-the-cut which we present now.

\begin{thm}
\label{Qcutsub}
Let $n,a,b\in\N_0$, $x\in(-1,1)$. Then
\begin{eqnarray}
&&\hspace{-0.7cm}{\sf Q}_n^{(a,b)}(x)\!=\!\frac{(-1)^n}{2^{n+1}}\sum_{\substack{k=0\\k\ne n}}^{a+b+2n}\frac{(-2)^k}{(n-k)}\left((1\!+\!x)^{n-k}-(x\!-\!1)^{n-k}\right)P_k^{(a+n-k,b+n-k)}(x)
+\frac12\log\left(\frac{1\!+\!x}{1\!-\!x}\right)P_n^{(a,b)}(x).
\end{eqnarray}
\end{thm}

\begin{proof}
Start with Theorem \ref{Qnotcutsumthm} and use the definition \eqref{Qcutdef}
which completes the proof.
\end{proof}

\noindent 
Note that by setting $a=b$ in the above result we can 
obtain an interesting finite sum expression for the Ferrers functions of the second kind with non-negative integer degree  
and order given as a sum over ultraspherical polynomials.

\begin{cor}
Let $n,a\in\N_0$, $x\in(-1,1)$. Then
\begin{eqnarray}
&&\hspace{-0.5cm}{\sf Q}_n^{a}(x)=\frac{(-1)^a(1-x^2)^{\frac12a}}{2\sqrt{\pi}}
\biggl((-1)^{n+a}2^n(n+a)!\nonumber\\
&&\hspace{1cm}\times\sum_{\substack{k=0\\[0.10cm]k\ne n-a}}^{2n}\frac{(-1)^k\Gamma(n-k+\frac12)}{2^k(2n-k)!(n-a-k)}\left((1+x)^{n-a-k}-(x-1)^{n-a-k}\right)C_k^{n-k+\frac12}(x)\nonumber\\
&&\hspace{6cm}+2^a\Gamma(a+\tfrac12)\log\left(\frac{1+x}{1-x}\right)C_{n-a}^{a+\frac12}(x)\biggr).
\end{eqnarray}
\end{cor}

\begin{proof}
Start with \eqref{Qcutsub} and 
set $a=b$. Then utilizing
\eqref{Qcutsym} below with
\eqref{relJacGeg} completes the proof.
\end{proof}

By using \eqref{Qcut1} we
see that as $x=1-\epsilon$ one has 
as $\epsilon\to 0^{+}$,
\begin{equation}
\label{Qcutnear1}
\hspace{1cm}{\sf Q}_\gamma^{(\alpha,\beta)}(1-\epsilon)\sim\frac{2^{\alpha-1}\Gamma(\alpha)\Gamma(\beta+\gamma+1)}{\Gamma(\alpha+\beta+\gamma+1)\epsilon^\alpha},
\end{equation}
where $\beta+\gamma+1\not\in-\N_0$ and
$\Re\alpha>0$.

\subsection{Specializations to Gegenbauer, associated Legendre and Ferrers functions}
Here we discuss some limiting cases where the Jacobi functions 
reduce to more elementary functions such as Gegenbauer, 
associated Legendre, and Ferrers functions.

These identities involve 
symmetric and antisymmetric Jacobi functions of
the first kind.
The relation between the symmetric Jacobi function
of the first kind and the Gegenbauer function
of the first kind for $z\in\mathbb C\setminus(-\infty,-1]$
is given by
\begin{equation}
\label{relJacGeg}
\hspace{1.0cm}P_\gamma^{(\alpha,\alpha)}(z)=
\frac{\Gamma(2\alpha+1)\Gamma(\alpha+\gamma+1)}
{\Gamma(\alpha+1)\Gamma(2\alpha+\gamma+1)}
C_\gamma^{\alpha+\frac12}(z).
\end{equation}
This follows by starting with \eqref{Jac1} and then 
comparing it to the Gauss hypergeometric representation 
of the Gegenbauer function of the first kind on the 
right-hand side using {\eqref{Gegenbauerfuncdef}}.

\begin{rem}
The relation between the symmetric Jacobi function 
of the first kind and the Ferrers function of the 
first kind is
\begin{equation}
\label{relJacFer}
\hspace{1.0cm}P_\gamma^{(\alpha,\alpha)}(z)=
\frac{2^\alpha\Gamma(\alpha+\gamma+1)}{\Gamma(\gamma+1)
(1-x^2)^{\frac12\alpha}}{\sf P}_{\alpha+\gamma}^{-\alpha}(x),
\end{equation}
where $x\in\mathbb C\setminus((-\infty,-1]\cup[1,\infty))$ and the relation between 
the symmetric Jacobi function
of the first kind 
and the associated Legendre function 
of the first kind is
\begin{equation}
\label{relJacAssP}
\hspace{1.0cm}P_\gamma^{(\alpha,\alpha)}(z)=
\frac{2^\alpha\Gamma(\alpha+\gamma+1)}{\Gamma(\gamma+1)(z^2-1)^{\frac12\alpha}} 
P_{\alpha+\gamma}^{-\alpha}(z)
\end{equation}
where $z\in\mathbb C\setminus(-\infty,1]$.
These are easily obtained
through
\cite[\href{http://dlmf.nist.gov/18.7.E2}{(18.7.2)}]{NIST:DLMF}
and
\cite[\href{http://dlmf.nist.gov/14.3.E21}{(14.3.21)}, \href{http://dlmf.nist.gov/14.3.E22}{(14.3.22)}]{NIST:DLMF}.
\end{rem}

\begin{rem}The relation between the antisymmetric Jacobi function of the first kind {on-the-cut} and
the Ferrers function of the first kind and the Gegenbauer function {of the first kind on-the-cut} is
\begin{eqnarray}
&&\hspace{-5.7cm}{\sf P}_\gamma^{(\alpha,-\alpha)}(x)
=\frac{\Gamma(\alpha+\gamma+1)}{\Gamma(\gamma+1)}
\left(\frac{1+x}{1-x}\right)^{\frac12\alpha}
{\sf P}_\gamma^{-\alpha}(x)\nonumber\\
&&\hspace{-3.8cm}=
\frac{\Gamma(2\alpha+1)\Gamma(\gamma-\alpha+1)}
{2^\alpha\Gamma(\gamma+1)\Gamma(\alpha+1)}
(1+x)^{\alpha}
{\sf C}_{\gamma-\alpha}^{\alpha+\frac12}(x)
,
\end{eqnarray}
where $x\in\mathbb C\setminus((-\infty,-1]\cup[1,\infty))$ and the relation between 
the antisymmetric Jacobi function
of the first kind 
and the associated Legendre and
Gegenbauer function 
of the first kinds is
\begin{eqnarray}
&&\hspace{-5.7cm}P_\gamma^{(\alpha,-\alpha)}(z)=\frac{\Gamma(\alpha+\gamma+1)}
{\Gamma(\gamma+1)}\left(\frac{z+1}{z-1}\right)^{\frac12\alpha}P_\gamma^{-\alpha}(z)
\nonumber\\
&&\hspace{-3.825cm}=
\frac{\Gamma(2\alpha+1)
\Gamma(\gamma-\alpha+1)}{2^\alpha\Gamma(\gamma+1)\Gamma(\alpha+1)}
(z+1)^{\alpha}
C_{\gamma-\alpha}^{\alpha+\frac12}(z),
\end{eqnarray}
where $z\in\mathbb C\setminus(-\infty,1]$.
These are obtained by comparing \eqref{Jac1} with \eqref{FerrersPdefnGauss2F1} and
\eqref{associatedLegendrefunctionP}.
\end{rem}

\noindent
\begin{rem}One has the following quadratic transformations for the symmetric Jacobi functions of the first kind 
which can be found in \cite[Theorem 4.1]{Szego}. Let $z\in\CC\setminus(-\infty,1]$, $\gamma,\alpha\in\CC$, $\alpha+\gamma\not\in-\mathbb N$. Then
\begin{eqnarray}
&&\hspace{-5.65cm}P_{2\gamma}^{(\alpha,\alpha)}(z)=
\frac{\sqrt{\pi}\,\Gamma(\alpha+2\gamma+1)}{2^{2\gamma}\Gamma(\gamma+\frac12)\Gamma(\alpha+\gamma+1)}P_\gamma^{(\alpha,-\frac12)}(2z^2-1),
\label{P2g}
\end{eqnarray}
where $\alpha+2\gamma\not\in-\mathbb N$,
$\gamma\not\in-\mathbb N+\frac12$, and
\begin{eqnarray}
&&\hspace{-5.5cm}P_{2\gamma+1}^{(\alpha,\alpha)}(z)=\frac{\sqrt{\pi}\,\Gamma(\alpha+2\gamma+2)z}{2^{2\gamma+1}\Gamma(\gamma+\frac32)\Gamma(\alpha+\gamma+1)}P_\gamma^{(\alpha,\frac12)}(2z^2-1),
\label{P2gp}
\end{eqnarray}
where 
$\alpha+2\gamma+1\not\in-\mathbb N$,
$\gamma\not\in-\mathbb N-\frac12$.
The restrictions on the parameters 
come directly by applying the restrictions on the parameters in Theorem \ref{bigPthm} to the Jacobi functions of the first kind on both sides of the relations.
\end{rem}
Below we present some identities which involve 
symmetric and antisymmetric Jacobi functions of
the second kind
\begin{thm}
Two equivalent relations between the symmetric 
Jacobi function of the
second kind and the associated Legendre function of the second kind are given by
\begin{eqnarray}
&&\hspace{-7.1cm} Q_{\gamma}^{(\alpha,\alpha)}(z)
=\frac
{2^{\alpha}\expe^{i\pi\alpha}
\Gamma(\alpha+\gamma+1)}
{\Gamma(\gamma+1)(z^2-1)^{\frac12\alpha}}
Q_{\alpha+\gamma}^{-\alpha}(z),
\label{JacQLeg}\\
&&\hspace{-7.1cm}Q_{\gamma}^{(\alpha,\alpha)}(z)=\frac{2^{\alpha}\expe^{-i\pi\alpha}
\Gamma(\alpha+\gamma+1)}
{\Gamma(2\alpha+\gamma+1)(z^2-1)^{\frac12\alpha}}
Q_{\alpha+\gamma}^{\alpha}(z),
\label{JacQLegb}
\end{eqnarray}
where $\alpha+\gamma\not\in-\mathbb N$. Also,
two equivalent
relations between  antisymmetric Jacobi functions of the second kind and the associated Legendre function of the second kind are given by
\begin{eqnarray}
&& \hspace{-6.2cm}Q_{\gamma}^{(\alpha,-\alpha)}(z)
=\frac
{\expe^{-i\pi\alpha}
\Gamma(\gamma-\alpha+1)}
{\Gamma(\gamma+1)}
\left(\frac{z+1}{z-1}\right)^{\frac12\alpha}
Q_{\gamma}^{\alpha}(z),
\label{JacasQLeg}\\
&&\hspace{-6.2cm}{Q_\gamma^{(-\alpha,\alpha)}(z)=\frac{\expe^{-i\pi\alpha}\Gamma(\gamma-\alpha+1)}{\Gamma(\gamma+1)}\left(\frac{z-1}{z+1}\right)^{\frac12\alpha}Q_\gamma^\alpha(z),}
\end{eqnarray}
where $\gamma-\alpha\not\in-\mathbb N$.
\end{thm}

\begin{proof}
By comparing \eqref{dJsk1} and \eqref{dJsk2} with \eqref{Qdefntwodivide1mz} 
and by using the Legendre duplication formula \cite[\href{http://dlmf.nist.gov/4.4.E5}{(5.5.5)}]{NIST:DLMF} 
one can obtain all these formulas in a straightforward way.
\end{proof}

See \cite[Section 3, (A.14)]{Cohl12pow} for an interesting
application of the symmetric relation for associated Legendre functions of the second kind.
\begin{rem}
Observe that by identifying \eqref{JacQLeg} and \eqref{JacQLegb} and
for $z\in \mathbb C\setminus [-1,1]$ one has
\[
\hspace{1cm}Q_\gamma^{(\alpha,\alpha)}(z)=
\frac{2^{2\alpha}}{(z^2-1)^{\alpha}} 
\left\{ \begin{array}{ll}
\displaystyle Q_{\gamma+2\alpha}^{(-\alpha,-\alpha)}(z), & \qquad\mathrm{if}\ z\in\CC\setminus[-1,1]\ s.t.\ \Re z\ge 0,\\[7pt]
\displaystyle \expe^{2\pi i\alpha} Q_{\gamma+2\alpha}^{(-\alpha,-\alpha)}(z), & \qquad\mathrm{if}\ z\in\CC\setminus[-1,1]\ s.t. \ \Re z<0\ and \ \Im z<0, \nonumber\\[7pt]
\displaystyle \expe^{-2\pi i\alpha} Q_{\gamma+2\alpha}^{(-\alpha,-\alpha)}(z), & \qquad\mathrm{if}\ z\in\CC\setminus[-1,1]\ s.t. \ \Re z<0\ and \ \Im z\ge 0, \nonumber
\end{array} \right.
\]
where the principal branches of complex powers are taken.
\end{rem}
\medskip
\begin{thm}
Let $\alpha,\gamma\in\mathbb C$, $z\in\mathbb C\setminus[-1,1]$, 
$\alpha+\gamma\not\in-\mathbb N$.
Then the relations between the symmetric and antisymmetric Jacobi
functions of the second kind to the Gegenbauer function
of the second kind is given by
\begin{equation}
\hspace{1.0cm}Q_\gamma^{(\alpha,\alpha)}(z)=
\expe^{-i\pi(\alpha+\frac12)}\sqrt{\pi}\,2^{2\alpha}
\frac{\Gamma(\alpha+\frac12)\Gamma(\alpha+\gamma+1)}
{\Gamma(2\alpha+\gamma+1)}D_\gamma^{\alpha+\frac12}(z),
\label{GegDsym}
\end{equation}
where 
$\alpha\in\mathbb C\setminus\{-\frac12,-\frac32,-\frac52,\ldots\}$
and
\begin{equation}
\hspace{1.0cm}Q_\gamma^{(\alpha,-\alpha)}(z)=\expe^{i\pi(\alpha-\frac12)}2^{2\gamma-\alpha+1}
\frac{\Gamma(\alpha+\gamma+1)\Gamma(\frac12-\alpha)\Gamma(\gamma+\frac32)}
{\Gamma(2\gamma+2)(z-1)^\alpha}D_{\alpha+\gamma}^{\frac12-\alpha}(z),
\label{GegDasym}
\end{equation}
where $\alpha\in\mathbb C\setminus\{\frac12,\frac32,\frac52,\ldots\}$, 
$\gamma\in\mathbb C\setminus\{-\frac32,-\frac52,-\frac72,\ldots\}$.
\end{thm}
\begin{proof}
Start with the definition of the Jacobi function
of the second kind \eqref{dJsk1} and take
$\beta=\alpha$. Then comparing \eqref{GegDhyper} using
Euler's $(z\mapsto z)$ transformation
\cite[\href{http://dlmf.nist.gov/15.8.E1}{(15.8.1)}]{NIST:DLMF} produces 
\eqref{GegDsym}.
In order to produce \eqref{GegDasym}, start 
with \eqref{dJsk1} and take
$\beta=-\alpha$. Then compare \eqref{GegDhyper} using
Euler's $(z\mapsto z)$ transformation
\cite[\href{http://dlmf.nist.gov/15.8.E1}{(15.8.1)}]{NIST:DLMF}. This completes the
proof.
\end{proof}

\noindent{One has the following quadratic transformations for symmetric Jacobi functions of the second kind.
\begin{thm}Let $z\in\CC\setminus[-1,1]$, $\gamma,\alpha\in\CC$, $\alpha+\gamma\not\in-\mathbb N$. Then
\begin{eqnarray}
&&\hspace{-5.5cm}Q_{2\gamma}^{(\alpha,\alpha)}(z)=
\frac{\sqrt{\pi}\,\Gamma(\alpha+2\gamma+1)}{2^{2\gamma}\Gamma(\gamma+\frac12)\Gamma(\alpha+\gamma+1)}Q_\gamma^{(\alpha,-\frac12)}(2z^2-1),
\label{Q2g}
\end{eqnarray}
where $\alpha+2\gamma\not\in-\mathbb N$,
$\gamma\not\in-\mathbb N+\frac12$, and
\begin{eqnarray}
&&\hspace{-5.5cm}Q_{2\gamma+1}^{(\alpha,\alpha)}(z)=\frac{\sqrt{\pi}\,\Gamma(\alpha+2\gamma+2)z}{2^{2\gamma+1}\Gamma(\gamma+\frac32)\Gamma(\alpha+\gamma+1)}Q_\gamma^{(\alpha,\frac12)}(2z^2-1),
\label{Q2gp}
\end{eqnarray}
where 
$\alpha+2\gamma+1\not\in-\mathbb N$,
$\gamma\not\in-\mathbb N-\frac12$.
\end{thm}

\begin{proof}
Starting with the left-hand sides of 
\eqref{Q2g}, \eqref{Q2gp} using the Gauss hypergeometric definition \eqref{dJsk1},
the ${}_2F_1$'s become of a form 
where $c=2b$. Then for both equations we use the quadratic
transformation of the Gauss hypergeometric function \cite[\href{http://dlmf.nist.gov/15.8.E14}{(15.8.14)}]{NIST:DLMF}. This transforms the ${}_2F_1$ to a form which is recognizable with the right-hand sides through \eqref{dJsk2}, \eqref{dJsk1}, respectively. This completes the proof.
The restrictions on the parameters 
come directly by applying the restrictions on the parameters in Theorem \ref{thmQ} to the Jacobi functions of the second kind on both sides of the relations.
\end{proof}

There is also an interesting alternative additional 
quadratic transformation for the Jacobi function 
of the second kind with $\alpha=\pm \frac12$. Note that 
there does not seem to be a corresponding formula
for the Jacobi function of the first kind since in 
this case the functions which would appear 
on the left-hand side would be a sum of two Gauss hypergeometric
functions.

\begin{thm}
\label{altquadQalpha}
Let $z\in\mathbb C$ such that $|z|<1$, $\beta,\gamma\in\mathbb C$ such that 
$\beta+\gamma+\frac12\not\in-\mathbb N_0$. Then
\begin{eqnarray}
&&\hspace{-2.0cm}C_{2\gamma+1}^\beta(z)=\frac{2^{2\gamma+2}\Gamma(\beta+\gamma+\frac12)}{\Gamma(-\gamma-\frac12)\Gamma(2\gamma+2)\Gamma(\beta)(1-z^2)^{\beta+\gamma+\frac12}}
Q_{-\gamma-1}^{(-\frac12,\beta+2\gamma+1)}\left(\frac{1+z^2}{1-z^2}\right),\\
&&\hspace{-2.0cm}C_{2\gamma}^\beta(z)=\frac{2^{2\gamma+1}\Gamma(\beta+\gamma+\frac12)z}{\Gamma(-\gamma+\frac12)\Gamma(2\gamma+1)\Gamma(\beta)(1-z^2)^{\beta+\gamma+\frac12}}
Q_{-\gamma-1}^{(\frac12,\beta+2\gamma)}\left(\frac{1+z^2}{1-z^2}\right).
\end{eqnarray}
\end{thm}

\begin{proof}
The results are easily verified by starting with 
\eqref{dJsk1}, \eqref{dJsk3}, substituting the related 
values in the Jacobi function of the second kind and 
comparing with associated Legendre functions of the 
first kind with argument $\sqrt{(z-1)/(z+1)}$ and 
utilizing a quadratic transformation of the Gauss 
hypergeometric function which relates the two
completes the proof.
\end{proof}

\begin{rem}
Note that in Theorem \ref{altquadQalpha}, if the argument
of the Jacobi function of the second kind has modulus greater than unity then the argument of the Gegenbauer function of the first kind has modulus less than unity.
\end{rem}

\begin{cor}
Let $z,\beta,\gamma\in\mathbb C$ such that $z\in\mathbb C\setminus[-1,1]$. Then
\begin{eqnarray}
&&\hspace{-1.9cm}Q_\gamma^{(\frac12,\beta)}(z)=
\frac{2^{\beta+3\gamma+\frac52}\Gamma(-2\gamma-1)\Gamma(\gamma+\frac32)\Gamma(\beta+2\gamma+2)}
{\Gamma(\beta+\gamma+\frac32)(z-1)^\frac12(z+1)^{\beta+\gamma+1}}
C_{-2\gamma-2}^{\beta+2\gamma+2}\left(\sqrt{\frac{z-1}{z+1}}\right),
\end{eqnarray}
where $-2\gamma-1,\gamma+\frac32,\beta+2\gamma+2\not\in-\mathbb N_0$, and 
\begin{eqnarray}
&&\hspace{-2.3cm}Q_\gamma^{(-\frac12,\beta)}(z)=
\frac{2^{\beta+3\gamma+\frac12}\Gamma(-2\gamma)\Gamma(\gamma+\frac12)\Gamma(\beta+2\gamma+1)}
{\Gamma(\beta+\gamma+\frac12)(z+1)^{\beta+\gamma+\frac12}}
C_{-2\gamma-1}^{\beta+2\gamma+1}\left(\sqrt{\frac{z-1}{z+1}}\right),
\end{eqnarray}
where $-2\gamma,\gamma+\frac12,\beta+2\gamma+1\not\in-\mathbb N_0$.
\end{cor}

\begin{proof}
Inverting Theorem \ref{altquadQalpha} completes the proof.
\end{proof}

\noindent
Note that the above results imply the following corollary.
\begin{cor}
Let $z,\beta,\gamma\in\mathbb C$ such that $z\in\mathbb C\setminus[-1,1]$, 
$\gamma+\frac32,\beta+\gamma+1\not\in-\mathbb N_0$. Then
\begin{equation}
\hspace{0.9cm}Q_\gamma^{(\frac12,\beta)}(z)=\frac{\Gamma(\gamma+\frac32)\Gamma(\beta+\gamma+1)}{\Gamma(\gamma+1)\Gamma(\beta+\gamma+\frac32)}
\left(\frac{2}{z-1}\right)^\frac12
Q_{\gamma+\frac12}^{(-\frac12,\beta)}(z).
\end{equation}
\end{cor}

\begin{proof}
Equating the two relations in Theorem \ref{altquadQalpha}
completes the proof.
\end{proof}

\begin{thm}
Let $x\in\mathbb C\setminus((-\infty,-1]\cup[1,\infty))$. Then
the relation between the symmetric and antisymmetric
Jacobi functions of the second kind on-the-cut and
the Ferrers function of the second kind are given by
\begin{eqnarray}
&&\hspace{-6.9cm}
\label{Qcutsym}\boro{{\sf Q}_\gamma^{(\alpha,\alpha)}(x)
=\frac{2^\alpha\Gamma(\alpha+\gamma+1)}{\Gamma(\gamma+1)
(1-x^2)^{\frac12\alpha}}{\sf Q}_{\gamma+\alpha}^{-\alpha}(x),}
\label{FerQsym}\\
&&\hspace{-6.9cm}\boro{{\sf Q}_\gamma^{(\alpha,-\alpha)}(x)
=\frac{\Gamma(\alpha+\gamma+1)}{\Gamma(\gamma+1)}
\left(\frac{1+x}{1-x}\right)^{\frac12\alpha}
{\sf Q}_\gamma^{-\alpha}(x),}
\label{FerQasym}
\end{eqnarray}
where $\alpha+\gamma\not\in-\mathbb N$,
\begin{eqnarray}
&&\hspace{-7.1cm}\boro{{\sf Q}_\gamma^{(-\alpha,\alpha)}(x)
=\frac{\Gamma(\gamma-\alpha+1)}{\Gamma(\gamma+1)}
\left(\frac{1-x}{1+x}\right)^{\frac12\alpha}
{\sf Q}_\gamma^{\alpha}(x),}
\label{lastQcut}
\end{eqnarray}
where $\gamma-\alpha\not\in-\mathbb N$.
\end{thm}

\begin{proof}
The result follows by taking into account \eqref{Qcut1} and 
cf.~\cite[Theorem 3.2]{Cohletal2021}
\begin{eqnarray}
&&\hspace{-1.4cm} {\sf Q}_\nu^\mu (x) = \frac{\pi}{2\sin(\pi\mu)}
\Bigg[
\cos(\pi(\nu + \mu))
\frac{\Gamma(\nu+\mu+1)}{
\Gamma(\nu-\mu+1)}
{\left( \frac{1+x}{1-x} \right)}^{{\frac12\mu}}
\Ohyp21{-\nu, \nu+1}{1+\mu}{\frac{1\!+\!x}{2}}
\nonumber
\\&&\hspace{5.1cm}
{}-
\cos(\pi\nu)
{\left( \frac{1-x}{1+x} \right)}^{{\frac12\mu}}
\Ohyp21{-\nu, \nu+1}{1 - \mu}{\frac{1\!+\!x}{2}} 
\Bigg]
,
\label{FerrersQ}
\end{eqnarray}
where $\nu \in \mathbb C$, $\mu \in \mathbb C \setminus \mathbb Z$, such
that $\nu + \mu \notin -\mathbb N$.

The formula \eqref{FerQsym} is obtained by taking $\beta=\alpha$, then comparing \eqref{Qcut1} with \eqref{FerrersQ}. 
The other identities follow by applying an analogous method taking $\beta=-\alpha$.
This completes the proof.
\end{proof}
\section{Addition theorems for the Jacobi function of the first kind}
\label{reffirstKoorn}
The Flensted-Jensen--Koornwinder addition theorem for Jacobi functions of the first kind is the extension of the Koornwinder addition theorem for Jacobi polynomials when the degree is allowed to be a complex number. This addition theorem has two separate contexts 
and some interesting special cases. We will refer to the two separate 
contexts as the hyperbolic and trigonometric contexts.
The hyperbolic context
arises when the Jacobi function is analytically continued in 
the complex plane from the ray $[1,\infty)$.
The trigonometric context
arises when the argument of the Jacobi function 
is analytically continued from the
real segment $(-1,1)$. 
First we will present the addition theorem for the Jacobi function of the first kind in the hyperbolic {context}. As we will see, the Jacobi function in the trigonometric context can be obtained from the Jacobi functions in the hyperbolic context (and vice versa).
We now present the most general form of the addition theorem for Jacobi functions of the first kind
in the hyperbolic and trigonometric contexts.

\begin{thm}
\label{Koornaddnhyp}
Let $\gamma,\alpha,\beta\in\mathbb C$, $z_1,z_2\in\mathbb C\setminus(-\infty,1]$,
$x_1,x_2\in\mathbb C\setminus((-\infty,-1]\cup[1,\infty))$,
$x,w\in\mathbb C$,
\begin{eqnarray}
&&\hspace{-1.6cm}Z^\pm:=Z^\pm(z_1,z_2,w,x)=2z_1^2z_2^2+2w^2
(z_1^2-1)(z_2^2-1)\pm 4z_1z_2wx(z_1^2-1)^\frac12(z_2^2-1)^\frac12-1,\label{ZZdef}\\
&&\hspace{-1.56cm}{\sf X}^\pm:={\sf X}^\pm(x_1,x_2,w,x)=2x_1^2x_2^2+2w^2
(1-x_1^2)(1-x_2^2)
\pm 4x_1x_2wx(1-x_1^2)^\frac12
(1-x_2^2)^\frac12-1,\label{XXdef}
\end{eqnarray}
such that the complex variables $\gamma,\alpha,\beta,z_1,z_2,x_1,x_2,x,w$ are in some
yet to be determined neighborhood of the real line.
Then
\begin{eqnarray}
&&\hspace{-0.40cm}P_\gamma^{(\alpha,\beta)}(Z^\pm)
=\frac{\Gamma(\alpha+1)\Gamma(\gamma+1)}{\Gamma(\alpha+\gamma+1)}\sum_{k=0}^\infty
\frac{(\alpha+1)_k(\alpha+\beta+\gamma+1)_k}
{(\alpha+k)(\beta+1)_k(-\gamma)_k}
\nonumber\\
&&\hspace{0.2cm}\times\sum_{l=0}^k
(\mp 1)^{k-l}
\frac{(\alpha+k+l)(-\beta-\gamma)_l}{(\alpha+\gamma+1)_l}
(z_1z_2)^{k-l}\left(({z_1^2-1})({z_2^2-1})\right)^{\frac{k+l}{2}}
\nonumber\\[0.1cm]
&&\hspace{0.7cm}\times
P_{\gamma-k}^{(\alpha+k+l,\beta+k-l)}(2z_1^2\!-\!1)
P_{\gamma-k}^{(\alpha+k+l,\beta+k-l)}(2z_2^2\!-\!1)
w^{k-l}P_{l}^{(\alpha-\beta-1,\beta+k-l)}(2w^2\!-\!1)\frac{\beta\!+\!k\!-\!l}{\beta}C_{k-l}^\beta(x),
\label{addPhypgen}
\end{eqnarray}
\begin{eqnarray}
&&\hspace{-0.29cm}{\sf P}_\gamma^{(\alpha,\beta)}({\sf X}^\pm)
=\frac{\Gamma(\alpha+1)\Gamma(\gamma+1)}{\Gamma(\alpha+\gamma+1)}\sum_{k=0}^\infty
\frac{(\alpha+1)_k(\alpha+\beta+\gamma+1)_k}
{(\alpha+k)(\beta+1)_k(-\gamma)_k}
\nonumber\\
&&\hspace{0.3cm}\times\sum_{l=0}^k
(\mp 1)^{k-l}
\frac{(\alpha+k+l)(-\beta-\gamma)_l}{(\alpha+\gamma+1)_l}
(x_1x_2)^{k-l}\left(({1-x_1^2})({1-x_2^2})\right)^{\frac{k+l}{2}}
 \nonumber\\[0.1cm]
&&\hspace{0.7cm}\times
{\sf P}_{\gamma-k}^{(\alpha+k+l,\beta+k-l)}(2x_1^2\!-\!1)
{\sf P}_{\gamma-k}^{(\alpha+k+l,\beta+k-l)}(2x_2^2\!-\!1)
w^{k-l}{P}_{l}^{(\alpha-\beta-1,\beta+k-l)}(2w^2\!-\!1)\frac{\beta\!+\!k\!-\!l}{\beta}C_{k-l}^\beta(x).
\label{addPtriggen}
\end{eqnarray}
\end{thm}
\begin{proof}
Start with the form of the Flensted-Jensen--Koornwinder addition theorem in \cite[Theorem 2.1]{FlenstedJensenKoorn79} (see also \cite[(24)]{LiPeng2007}). Define the Flensted-Jensen--Koornwinder--Jacobi function of the first kind \cite[(2.1)]{FlenstedJensenKoorn79} (Flensted-Jensen--Koornwinder refer to this function as the Jacobi function of the first kind) 
\begin{equation}
\varphi_\lambda^{(\alpha,\beta)}(t):=\hyp21{\frac12(\alpha+\beta+1+i\lambda),\frac12(\alpha+\beta+1-i\lambda)}{\alpha+1}{-\sinh^2t},
\label{FJKJdef}
\end{equation}
and express it in terms of 
the Jacobi function of the first kind using 
\begin{equation}
\hspace{0.0cm}\varphi_\lambda^{(\alpha,\beta)}(t)=\frac{\Gamma(\alpha+1)
\Gamma(-\tfrac12(\alpha+\beta-1+i\lambda))}
{\Gamma(\tfrac12(\alpha-\beta+1-i\lambda))}
P_{-\tfrac12(\alpha+\beta+1+i\lambda)}^{(\alpha,\beta)}(\cosh(2t)),
\label{phiPrel}
\end{equation}
which follows by comparing the Gauss hypergeometric
representations of the functions. Replacing
$\lambda=i(\alpha+\beta+2\gamma+1)$ and setting 
$z_1=\cosh t_1$, $z_2=\cosh t_2$ and $w=\cos\psi$ 
produces the form of the addition theorem \eqref{addPhypgen}. Then analytically continuing 
\eqref{addPhypgen} to ${\sf X}^\pm\in(-1,1)$ 
using
\eqref{JacPcutdef}
produces \eqref{addPtriggen}. This
completes the proof.
\end{proof}

\begin{rem}
It is worth mentioning that in the definitions of $Z^\pm$ \eqref{ZZdef} and ${\sf X}^\pm$ \eqref{XXdef}, the
influence of the $\pm 1$ factor on the addition theorems in
Theorem \ref{Koornaddnhyp} and elsewhere in this paper is simply due
to the influence of the parity relation for ultraspherical polynomials \eqref{parGeg} upon the
reflection map $x\mapsto -x$.
\end{rem}

\begin{rem}
Note that there are various ways of expressing the variables $Z^\pm$ \eqref{ZZdef} and ${\sf X}^\pm$ \eqref{XXdef},
which are useful in different applications. For instance, we may also write
\begin{eqnarray}
&&\hspace{-2.95cm}Z^\pm=
2z_1^2z_2^2(1-x^2)-1+2(z_1^2-1)(z_2^2-1)\left(w\pm\frac{xz_1z_2}{\sqrt{(z_1^2-1)(z_2^2-1)}}\right)^2\nonumber\\
&&\hspace{-2.3cm}=2(z_1^2-1)(z_2^2-1)\left(
\frac{2z_1^2z_2^2(1-x^2)-1}{2(z_1^2-1)(z_2^2-1)}+\left(w\pm \frac{xz_1z_2}{\sqrt{(z_1^2-1)(z_2^2-1)}}\right)^2\right),\nonumber
\end{eqnarray}
\begin{eqnarray}
&&\hspace{-2.85cm}{\sf X}^\pm=
2x_1^2x_2^2(1-x^2)-1+2(1-x_1^2)(1-x_2^2)\left(w\pm \frac{xx_1x_2}{\sqrt{(1-x_1^2)(1-x_2^2)}}\right)^2\nonumber\\
&&\hspace{-2.25cm}=2(1-x_1^2)(1-x_2^2)\left(
\frac{2x_1^2x_2^2(1-x^2)-1}{2(1-x_1^2)(1-x_2^2)}+\left(w\pm\frac{xx_1x_2}{\sqrt{(1-x_1^2)(1-x_2^2)}}\right)^2\right).\nonumber
\end{eqnarray}
\end{rem}

First we will develop some tools which will help us prove the correct form of the double summation addition theorem for the Jacobi function of the second kind. Consider the orthogonality of the ultraspherical polynomials and the Jacobi polynomials with the argument 
$2w^2-1$.
\begin{lem}
\label{orthogGegJac}
Let $m,n,p\in\N_0$, $\mu\in(-\frac12,\infty)$
$\alpha,\beta\in(-1,\infty)$, $\alpha>\beta$.
Then the ultraspherical and Jacobi polynomials satisfy the following orthogonality relations
\begin{eqnarray}
&&\hspace{-3.3cm}\int_0^\pi\! C_m^\mu(\cos\phi)C_n^\mu(\cos\phi)(\sin\phi)^{2\mu}\,\dd\phi\!=\!\frac{\pi\,\Gamma(2\mu+n)}{2^{2\mu-1}(\mu+n)n!\,\Gamma(\mu)^2}\delta_{m,n},
\label{orthogGeg}
\\
&&\hspace{-3.3cm}\int_0^1 \!P_m^{(\alpha-\beta-1,\beta+p)}(2w^2\!-\!1)
P_n^{(\alpha-\beta-1,\beta+p)}(2w^2\!-\!1)w^{2\beta+2p+1}(1\!-\!w^2)^{\alpha-\beta-1} \,\dd w \nonumber\\
&&\hspace{2.5cm}
=\frac{\Gamma(\alpha-\beta+n)\Gamma(\beta+1+p+n)}{2(\alpha+p+2n)\Gamma(\alpha+p+n)n!}\delta_{m,n}.
\label{orthogJac}
\end{eqnarray}
\end{lem}
\begin{proof}
These orthogonality relations follow easily from \cite[(9.8.20), (9.8.2)]{Koekoeketal} upon making the straightforward substitutions.
\end{proof}

\subsection{The parabolic biangle orthogonal polynomial system}
\noindent Define the 2-variable orthogonal polynomial system 
which are sometimes referred to as parabolic biangle polynomials \cite{KoornwinderSchwartz97}
\begin{equation}
\label{defParabipoly}
{\mathcal P}_{k,l}^{(\alpha,\beta)}(w,\phi):=
w^{k-l}P_l^{(\alpha-\beta-1,\beta+k-l)}(2w^2\!-\!1)C_{k-l}^\beta(\cos\phi),
\end{equation}
where $k,l\in\N_0$ such that $l\le k$.
These 2-variable polynomials are orthogonal over $(w,\phi)\in(0,1)\times(0,\pi)$ with
orthogonality measure $\dd m^{(\alpha,\beta)}(w,\phi)$ defined by 
\begin{equation}
\hspace{0.3cm}\dd m^{(\alpha,\beta)}(w,\phi):=(1-w^2)^{\alpha-\beta-1}w^{2\beta+1}(\sin\phi)^{2\beta}\,\dd w\,\dd \phi.
\label{measdef}
\end{equation}
The orthogonal polynomial system ${\mathcal P}_{k,l}^{(\alpha,\beta)}(w,\phi)$ is deeply connected to the addition theorem for Jacobi functions of the first and second kind. Using 
the orthogonality relations in Lemma \ref{orthogGegJac} we can derive the orthogonality relation for the 2-variable parabolic biangle polynomials.
\begin{lem}Let $k,l,k',l'\in\N_0$ such that $l\le k$, $l'\le k'$, $\alpha,\beta\in(-1,\infty)$, $\alpha>\beta$. Then the 2-variable parabolic biangle polynomials
satisfy the following orthogonality relation
\begin{eqnarray}
&&\hspace{-0.8cm}\int_0^1\int_0^\pi {\mathcal P}_{k,l}^{(\alpha,\beta)}(w,\phi)
 {\mathcal P}_{k',l'}^{(\alpha,\beta)}(w,\phi)\, \dd m^{(\alpha,\beta)}(w,\phi)=\frac{\pi\,\Gamma(\beta\!+\!1\!+\!k)\Gamma(2\beta\!+\!k\!-\!l)\Gamma(\alpha\!-\!\beta\!+\!l)}{2^{2\beta}\Gamma(\beta)^2(\alpha\!+\!k\!+\!l)(\beta\!+\!k\!-\!l)\Gamma(\alpha\!+\!k)(k\!-\!l)!l!}\delta_{k,k'}\delta_{l,l'}.
\end{eqnarray}
\end{lem}
\begin{proof}
Starting with the definition of the 2-variable 
parabolic biangle polynomials \eqref{defParabipoly} and integrating over 
$(w,\phi)\in(0,1)\times(0,\pi)$ with 
measure \eqref{measdef} and using
the orthogonality relations in Lemma \ref{orthogGegJac} completes the proof.
\end{proof}

\noindent The following result is a Jacobi function of the first kind generalization
of \cite[(4.10)]{Koornwinder75} for Jacobi polynomials.

\begin{thm}
\label{IntPdblmsr}
Let $k,l\in\N_0$ with $l\le k$, $\gamma,\alpha,\beta\in\mathbb C$, $z_1,z_2\in\mathbb C\setminus(-\infty,1]$,
$Z^\pm$
defined in \eqref{ZZdef},
such that $x=\cos\phi$ and the complex variables $\gamma,\alpha,\beta,z_1,z_2$ are in some
yet to be determined neighborhood of the real line.
Then
\begin{eqnarray}
&&\hspace{-0.7cm}
\int_0^1 \!\int_0^\pi P_\gamma^{(\alpha,\beta)}(Z^\pm)\, w^{k-l}P_l^{(\alpha-\beta-1,\beta+k-l)}(2w^2\!-\!1)C_{k-l}^\beta(\cos\phi)
 \,\dd m^{(\alpha,\beta)}(w,\phi)\nonumber\\[0.15cm]
&&\hspace{0.0cm}={(\mp 1)^{k+l}}\,{\sf A}^{(\alpha,\beta,\gamma)}_{k,l}
(z_1z_2)^{k-l} ((z_1^2-\!1)(z_2^2\!-\!1))^{\frac12(k+l)} P_{\gamma-k}^{(\alpha+k+l,\beta+k-l)}(2z_1^2\!-\!1)
P_{\gamma-k}^{(\alpha+k+l,\beta+k-l)}(2z_2^2\!-\!1),
\end{eqnarray}
where
\begin{equation}
\label{Akl}
\hspace{0.2cm}{\sf A}^{(\alpha,\beta,\gamma)}_{k,l}:=\frac{\pi\Gamma(\gamma+1)(\alpha+\beta+\gamma+1)_k\Gamma(2\beta+k-l)\Gamma(\alpha-\beta+l)(-\beta-\gamma)_l}{2^{2\beta}\Gamma(\beta)(-\gamma)_k\,(k-l)!\,l!\,\Gamma(\alpha+\gamma+1+l)}.
\end{equation}
\end{thm}
\begin{proof}
Start with the addition theorem 
for the Jacobi function of the first kind
\eqref{addPhypgen} and consider the $(k,l)$-th term in the double series. It involves
a product of two Jacobi functions of the first
kind with degree $\gamma-k$ and parameters $(\alpha+k+l,\beta+k-l)$. 
Replace in \eqref{addPhypgen} the summation indices $k,l$ by $k',l'$, multiply both sides of \eqref{addPhypgen} by
${\mathcal P}_{k,l}^{(\alpha,\beta)}(w,\phi)\,\dd m^{(\alpha,\beta)}(w,\phi)$,
and integrate both sides over $(w,\phi)\in(0,1)\times(0,\pi)$
using \eqref{orthogGeg},
\eqref{orthogJac}. This completes the proof.
\end{proof}

We will return to the parabolic biangle polynomials in Section \ref{secfunsecond}.

\subsection{Special cases of the addition theorem for the Jacobi function of the first kind}

In the case when $z_1,z_2,x_1,x_2,w,x=\cos\phi$ are real numbers then the argument of the Jacobi
function of the first kind in the addition theorem takes a simpler form convenient form and was proved in Flensted-Jensen--Koornwinder \cite{FlenstedJensenKoorn79}. 

\begin{rem}
In the case where the variables $z_1,z_2,x_1,x_2,x,w$ are real then you may
write $Z^\pm$ and ${\sf X}^\pm$ as follows
\begin{eqnarray}
&&\hspace{-8.6cm}{Z}^\pm=2\left|z_1z_2\pm\expe^{i\phi}w\sqrt{z_1^2-1}\sqrt{z_2^2-1}\right|^2-1,
\label{ZZrealdef}
\\[0.10cm]
&&\hspace{-8.55cm}{\sf X}^\pm=2\left|x_1x_2\pm\expe^{i\phi}w\sqrt{1-x_1^2}\sqrt{1-x_2^2}\right|^2-1.
\label{XXrealdef}
\end{eqnarray}
\end{rem}

\noindent We now give a result which appears to be identical to Theorem \ref{Koornaddnhyp}, but it must be emphasized that it is only in the real case that we are able to write
$Z^\pm$, ${\sf X}^\pm$ using 
\eqref{ZZrealdef}, \eqref{XXrealdef}.
Otherwise one must use \eqref{ZZdef}, \eqref{XXdef}.

\begin{thm}
\label{Koornaddnhypd1}
Let $\gamma,\alpha,\beta\in\mathbb C$, $z_1,z_2\in(1,\infty)$,
$x_1,x_2\in(-1,1)$,
$w\in\mathbb R$, $\phi\in[0,\pi]$,
and $Z^\pm$, ${\sf X}^\pm$ is defined as 
in \eqref{ZZrealdef}, \eqref{XXrealdef} respectively.
Then
\begin{eqnarray}
&&\hspace{-0.60cm}P_\gamma^{(\alpha,\beta)}(Z^\pm)
=\frac{\Gamma(\alpha+1)\Gamma(\gamma+1)}{\Gamma(\alpha+\gamma+1)}\sum_{k=0}^\infty
\frac{(\alpha+1)_k(\alpha+\beta+\gamma+1)_k}
{(\alpha+k)(\beta+1)_k(-\gamma)_k}
\nonumber\\
&&\hspace{-0.2cm}\times\sum_{l=0}^k
(\mp 1)^{k-l}
\frac{(\alpha+k+l)(-\beta-\gamma)_l}{(\alpha+\gamma+1)_l}
(z_1z_2)^{k-l}\left(({z_1^2-1}) ({z_2^2-1})\right)^{\frac{k+l}{2}}\nonumber\\[0.1cm]
&&\hspace{0.3cm}\times
P_{\gamma-k}^{(\alpha+k+l,\beta+k-l)}(2z_1^2\!-\!1)
P_{\gamma-k}^{(\alpha+k+l,\beta+k-l)}(2z_2^2\!-\!1)
w^{k-l}P_{l}^{(\alpha-\beta-1,\beta+k-l)}(2w^2\!-\!1)
\frac{\beta\!+\!k\!-\!l}{\beta}C_{k-l}^\beta(\cos\phi),
\label{addPhyp}
\end{eqnarray}
\begin{eqnarray}
&&\hspace{-0.45cm}{\sf P}_\gamma^{(\alpha,\beta)}({\sf X}^\pm)
=\frac{\Gamma(\alpha+1)\Gamma(\gamma+1)}{\Gamma(\alpha+\gamma+1)}\sum_{k=0}^\infty
\frac{(\alpha+1)_k(\alpha+\beta+\gamma+1)_k}
{(\alpha+k)(\beta+1)_k(-\gamma)_k}
\nonumber\\
&&\hspace{-0.2cm}\times\sum_{l=0}^k
(\mp 1)^{k-l}
\frac{(\alpha+k+l)(-\beta-\gamma)_l}{(\alpha+\gamma+1)_l}
(x_1x_2)^{k-l}\left((1-x_1^2)(1-x_2^2)\right)^{\frac{k+l}{2}}
 \nonumber\\[0.1cm]
&&\hspace{0.3cm}\times
{\sf P}_{\gamma-k}^{(\alpha+k+l,\beta+k-l)}(2x_1^2\!-\!1)
{\sf P}_{\gamma-k}^{(\alpha+k+l,\beta+k-l)}(2x_2^2\!-\!1)
w^{k-l}{P}_{l}^{(\alpha-\beta-1,\beta+k-l)}(2w^2\!-\!1)\frac{\beta\!+\!k\!-\!l}{\beta}C_{k-l}^\beta(\cos\phi).
\label{addPtrig}
\end{eqnarray}
\end{thm}
\begin{proof}
Starting with Theorem \ref{Koornaddnhyp}
and restricting such that the variables
$z_1,z_2,x_1,x_2,w,x=\cos\phi$ are real
completes the proof.
\end{proof}

Next we have a specialization of Theorem \ref{Koornaddnhypd1}
when $w=1$.

\begin{cor}
\label{limKoorn}
Let $\gamma,\alpha,\beta\in\mathbb C$, $r_1,r_2\in[0,\infty)$, $\theta_1,\theta_2\in[0,\frac{\pi}{2}]$, 
$\phi\in[0,\pi]$, 
\begin{eqnarray}
&&\hspace{-0.25cm}Z^\pm:=2\left|\cosh r_1\cosh r_2\pm\expe^{i\phi}
\sinh r_1\sinh r_2\right|^2-1
=\cosh(2 r_1)\cosh(2 r_2)\pm\sinh(2 r_1)\sinh(2 r_2)\cos\phi,\label{spZZdef}\\
&&\hspace{-0.25cm}{\sf X}^\pm:=2\left|\cos\theta_1\cos\theta_2\pm\expe^{i\phi}
\sin\theta_1\sin\theta_2\right|^2-1
=\cos(2\theta_1)\cos(2\theta_2)\pm\sin(2\theta_1)\sin(2\theta_2)\cos\phi.
\label{spXXdef}
\end{eqnarray}
Then
\begin{eqnarray}
&&\hspace{-0.85cm}P_\gamma^{(\alpha,\beta)}(Z^\pm)
=\frac{\Gamma(\alpha+1)\Gamma(\gamma+1)}{\Gamma(\alpha+\gamma+1)}\sum_{k=0}^\infty
\frac{(\alpha)_k(\frac{\alpha}{2}+1)_k(-\beta-\gamma)_k(\alpha+\beta+\gamma+1)_k}
{(\frac{\alpha}{2})_k(\beta+1)_k(-\gamma)_k(\alpha+\gamma+1)_k}
(\sinh r_1\sinh r_2)^{2k}\nonumber\\
&&\hspace{0.5cm}\times\sum_{l=0}^k
\frac{(\mp 1)^l(\alpha-\beta)_{k-l}(-\alpha-\gamma-k)_l(-\alpha-2k+1)_l}
{(k-l)!(-\alpha-2k)_l(\beta+\gamma+1)_l}
(\coth r_1\coth r_2
)^l\,
\nonumber\\[0.1cm]
&&\hspace{1.5cm}\times
P_{\gamma-k}^{(\alpha+2k-l,\beta+l)}(\cosh(2r_1))
P_{\gamma-k}^{(\alpha+2k-l,\beta+l)}(\cosh(2r_2))\frac{\beta+l}{\beta}C_l^\beta(\cos\phi) 
.
\end{eqnarray}
\begin{eqnarray}
&&\hspace{-0.85cm}{\sf P}_\gamma^{(\alpha,\beta)}({\sf X}^\pm)
=\frac{\Gamma(\alpha+1)\Gamma(\gamma+1)}{\Gamma(\alpha+\gamma+1)}\sum_{k=0}^\infty
\frac{(\alpha)_k(\frac{\alpha}{2}+1)_k(-\beta-\gamma)_k(\alpha+\beta+\gamma+1)_k}
{(\frac{\alpha}{2})_k(\beta+1)_k(-\gamma)_k(\alpha+\gamma+1)_k}
(\sin\theta_1\sin\theta_2)^{2k}\nonumber\\
&&\hspace{0.5cm}\times\sum_{l=0}^k
\frac{(\mp 1)^l(\alpha-\beta)_{k-l}(-\alpha-\gamma-k)_l(-\alpha-2k+1)_l}
{(k-l)!(-\alpha-2k)_l(\beta+\gamma+1)_l}
(\cot\theta_1\cot\theta_2
)^l\,
\nonumber\\[0.1cm]
&&\hspace{1.5cm}\times
P_{\gamma-k}^{(\alpha+2k-l,\beta+l)}(\cos(2\theta_1))
P_{\gamma-k}^{(\alpha+2k-l,\beta+l)}(\cos(2\theta_2))\frac{\beta+l}{\beta}C_l^\beta(\cos\phi) 
.
\end{eqnarray}
\end{cor}

\begin{proof}
Start with Theorem \ref{Koornaddnhypd1} and let $w=1$ using 
\eqref{Jacone}
and substituting $l\mapsto l'=k-l$ followed by relabeling $l'\mapsto l$ completes the proof.
\end{proof}

By letting $\alpha=\beta$ in Corollary \ref{limKoorn} we 
can relate the above result to associated Legendre and Gegenbauer functions of the first kind.
This is mentioned in \cite{Koornwinder1972AI}, 
namely that Koornwinder's addition theorem for 
Jacobi polynomials generalizes
Gegenbauer's addition theorem 
\eqref{Gegaddn}. Similarly, the extension to the 
Flensted-Jensen--Koornwinder addition theorem for Jacobi functions of the first kind generalizes the addition theorem for Gegenbauer functions of the first kind. First we define
the variables
\begin{eqnarray}
&&\label{ZGegdefc}\hspace{-5.8cm}{\mathcal Z}^\pm:={\mathcal Z}^\pm(r_1,r_2,\phi):=\cosh r_1\cosh r_2\pm \sinh r_1\sinh r_2\cos\phi,\\
&&\label{XGegdefc}\hspace{-5.8cm}{\mathcal X}^\pm:={\mathcal X}^\pm(\theta_1,\theta_2,\phi):=\cos\theta_1\cos\theta_2\pm \sin\theta_1\sin\theta_2\cos\phi.
\end{eqnarray}

\begin{cor}
Let $\gamma,\alpha\in\mathbb C$,
$r_1,r_2\in[0,\infty)$, 
$\theta_1,\theta_2\in[0,\frac{\pi}{2}]$, 
$\phi\in[0,\pi]$, 
and ${\mathcal Z}^\pm$, ${\mathcal X}^\pm$ as defined in \eqref{ZGegdefc}, \eqref{XGegdefc}, respectively.
Then
\begin{eqnarray}
\label{GegaddnZ}
&&\hspace{-1.5cm}C_\gamma^\alpha({\mathcal Z}^\pm)=\frac{\Gamma(2\alpha)\Gamma(\gamma+1)}{\Gamma(2\alpha+\gamma)}\sum_{k=0}^\infty \frac{(\mp 1)^k\,2^{2k}(\alpha)_k(\alpha)_k}
{(-\gamma)_k(2\alpha+\gamma)_k}(\sinh(2r_1)\sinh(2r_2))^k\nonumber\\
&&\hspace{3.75cm}\times C_{\gamma-k}^{\alpha+k}(\cosh(2r_1))C_{\gamma-k}^{\alpha+k}(\cosh(2r_2))\frac{\alpha\!-\!\tfrac12\!+\!k}{\alpha\!-\!\tfrac12}C_k^{\alpha-\frac12}(\cos\phi),
\end{eqnarray}
\begin{eqnarray}
\label{GegaddnX}
&&\hspace{-1.5cm}C_\gamma^\alpha({\mathcal X}^\pm)=\frac{\Gamma(2\alpha)\Gamma(\gamma+1)}{\Gamma(2\alpha+\gamma)}\sum_{k=0}^\infty \frac{(\mp 1)^k\,2^{2k}(\alpha)_k(\alpha)_k}
{(-\gamma)_k(2\alpha+\gamma)_k}(\sin(2\theta_1)\sin(2\theta_2))^k\nonumber\\
&&\hspace{3.75cm}\times C_{\gamma-k}^{\alpha+k}(\cos(2\theta_1))C_{\gamma-k}^{\alpha+k}(\cos(2\theta_2))\frac{\alpha\!-\!\tfrac12\!+\!k}{\alpha\!-\!\tfrac12}C_k^{\alpha-\frac12}(\cos\phi),
\end{eqnarray}
or equivalently
\begin{eqnarray}
&&\hspace{-0.8cm}\frac{1}{(1-{{\mathcal Z}^\pm}^2)^{\frac12\alpha}} P_\gamma^{-\alpha}({{\mathcal Z}}^\pm)=\frac{2^\alpha\Gamma(\alpha+1)}{(\sinh(2\theta_1)\sinh(2\theta_1))^\alpha}\nonumber\\
&&\hspace{1cm}\times\sum_{k=0}^\infty(\pm 1)^k(\alpha-\gamma)_k(\alpha+\gamma+1)_k{P}_\gamma^{-\alpha-k}(\cosh(2r_1)){P}_\gamma^{-\alpha-k}(\cosh(2r_2))\frac{\alpha+k}{\alpha}C_k^\alpha(\cos\phi),
\end{eqnarray}
\begin{eqnarray}
&&\hspace{-1.1cm}\frac{1}{(1-{{\mathcal X}^\pm}^2)^{\frac12\alpha}}{\sf P}_\gamma^{-\alpha}({\mathcal X}^\pm)=\frac{2^\alpha\Gamma(\alpha+1)}{(\sin(2\theta_1)\sin(2\theta_1))^\alpha}\nonumber\\
&&\hspace{1cm}\times\sum_{k=0}^\infty(\pm 1)^k(\alpha-\gamma)_k(\alpha+\gamma+1)_k{\sf P}_\gamma^{-\alpha-k}(\cos(2\theta_1)){\sf P}_\gamma^{-\alpha-k}(\cos(2\theta_2))\frac{\alpha+k}{\alpha}C_k^\alpha(\cos\phi).
\end{eqnarray}
\end{cor}
\begin{proof}
Start with Corollary \ref{limKoorn} and let $\alpha=\beta$ using
\eqref{relJacGeg}, 
\eqref{relJacFer} respectively for the Jacobi functions of the first kind on the left-hand side and on the right-hand side. Then mapping
$(2z_1^2-1,2z_2^2-1)\mapsto(z_1,z_2)$,
$(2x_1^2-1,2x_2^2-1)\mapsto(x_1,x_2)$,
where $z_1=\cosh r_1$, $z_2=\cosh r_2$, $x_1=\cos\theta_1$, $x_2=\cos\theta_2$,
and simplifying using \eqref{relJacGeg}--\eqref{relJacAssP} 
completes the proof.

Another way to prove this result is to take
$\beta=-\frac12$, $w=\cos\psi=1$, $\gamma\to2\gamma$ in \eqref{addPhyp} and use the quadratic transformation \eqref{P2g}. After using \eqref{relJacGeg}, this produces the left-hand side of \eqref{GegaddnZ} with degree $4\gamma$
and order given by $\alpha+\frac12$. Because we set $w=1$, the sum over $l$ only survives for $l=0,1$. By taking $4\gamma\mapsto\gamma$ and expressing the contribution due to each of these terms one can identify Gegenbauer's addition theorem through repeated application of \eqref{relJacGeg} on the right-hand side and that
\[
\hspace{0.0cm}\sum_{k=0}^{\infty}\left(f_{2k}+f_{2k+1}\right)=\sum_{k=0}^\infty f_k,
\]
for some sequence $\{f\}_{k\in\mathbb N_0}$, one 
arrives at \eqref{GegaddnZ}.
This other proof is similar for \eqref{GegaddnX}.
\end{proof}

\section{Addition theorems for the Jacobi function of the second kind}
\label{secfunsecond}
Now we present double summation addition theorems for the Jacobi functions of the second kind in the hyperbolic and trigonometric contexts. 
\subsection{The hyperbolic context for the addition theorem for the Jacobi function of the second kind}

\noindent Now, we present the double summation addition theorem for the Jacobi function of the second kind in the hyperbolic context.
\noindent 
Define 
\begin{equation}
\hspace{0.0cm}z_\lessgtr:=\hspace{-0.5cm}\begin{array}{rcl}
&\min{}&\\[-0.1cm]&\max{}&
\end{array}\hspace{-0.5cm}\{z_1,z_2\}, 
\label{zlgt}
\end{equation} 
where ${z_1,z_2}\in(1,\infty)$,
and in the case where ${z_1,z_2}\in\CC$, then if one
takes without loss of generality ${z_1}=z_>$ to lie on 
an ellipse with foci at $\pm 1$, then ${z_2}=z_<$ must 
be chosen to be in the interior of that ellipse.

\begin{thm}
\label{AddnhypQgen}
Let $\gamma,\alpha,\beta\in\mathbb C$, $z_1,z_2\in\mathbb C\setminus(-\infty,1]$,
$x,w\in\mathbb C$, $Z^\pm$
defined in \eqref{ZZdef},
such that the complex variables $\gamma,\alpha,\beta,z_1,z_2,x,w$ are in some
yet to be determined neighborhood of the real line.
Then
\begin{eqnarray}
&&\hspace{-0.40cm}Q_\gamma^{(\alpha,\beta)}(Z^\pm)
=\frac{\Gamma(\alpha+1)\Gamma(\gamma+1)}{\Gamma(\alpha+\gamma+1)}\sum_{k=0}^\infty
\frac{(\alpha+1)_k(\gamma+1)_k}
{(\alpha+k)(\beta+1)_k(1-\gamma)_k}
\nonumber\\
&&\hspace{0.3cm}\times\sum_{l=0}^k
(\pm 1)^{k-l}
(\alpha+k+l)
(z_1z_2)^{k-l}\left(({z_1^2-1})({z_2^2-1})\right)^{\frac{k+l}{2}}
 \nonumber\\[0.1cm]
&&\hspace{0.7cm}\times
P_{\gamma-k}^{(\alpha+k+l,\beta+k-l)}(2z_<^2\!-\!1)
Q_{\gamma-k}^{(\alpha+k+l,\beta+k-l)}(2z_>^2\!-\!1)
w^{k-l}P_{l}^{(\alpha-\beta-1,\beta+k-l)}(2w^2\!-\!1)\frac{\beta\!+\!k\!-\!l}{\beta}C_{k-l}^\beta(x).
\label{addQhyp}
\end{eqnarray}
\end{thm}
\begin{proof}
Start with Theorem \ref{IntPdblmsr} which is equivalent
to \eqref{addPhypgen}.
Now use the connection relation which relates the Jacobi function of the first kind with two Jacobi functions of the second kind \eqref{Durconrel} once in the integrand of the double integral and again on the Jacobi function of the first kind
with argument $2z_2^2-1$, assuming without
loss of generality that $z_2=z_>$.
This results in the following equation
\begin{eqnarray}
&&\hspace{-0.15cm} {\sf B}_\gamma^{(\alpha,\beta)}\int_0^1\!\!
\int_0^{\pi}
\!\!\!Q_{\gamma}^{(\alpha,\beta)}(Z^\pm)\,w^{k-l}P_l^{(\alpha-\beta-1,\beta+k-l)}(2w^2\!-\!1)C_{k-l}^\beta(\cos\phi)
\,\dd m^{(\alpha,\beta)}(w,\phi)\nonumber\\
&&\hspace{1.2cm}+{\sf C}_{\gamma,k,l}^{(\alpha,\beta)}(z_1z_2)^{k-l} ((z_1^2-1)(z_2^2-1))^{\frac12(k+l)}
P_{\gamma-k}^{(\alpha+k+l,\alpha+k-l)}(2z_<^2-1) Q_{\gamma-k}^{(\alpha+k+l,\alpha+k-l)}(2z_>^2-1)\nonumber\\
&&\hspace{0.4cm}={\sf D}_\gamma^{(\alpha,\beta)}\int_0^1\!\!
\int_0^{\pi}
\!\!\!Q_{-\alpha-\beta-\gamma-1}^{(\alpha,\beta)}(Z^\pm)w^{k-l}P_l^{(\alpha-\beta-1,\beta+k-l)}(2w^2\!-\!1)C_{k-l}^\beta(\cos\phi)
\,\dd m^{(\alpha,\beta)}(w,\phi)\nonumber\\
&&\hspace{1.2cm}+{\sf E}_{\gamma,k,l}^{(\alpha,\beta)}(z_1z_2)^{k-l}((z_1^2-1)(z_2^2-1))^{\frac12(k+l)}
P_{-\alpha-\beta-\gamma-k-1}^{(\alpha+k+l,\alpha+k-l)}(2z_<^2-1)
 Q_{-\alpha-\beta-\gamma-k-1}^{(\alpha+k+l,\alpha+k-l)}(2z_>^2-1),
\label{leadrel}
\end{eqnarray}
where
\begin{eqnarray}
&&\hspace{-0.8cm}{\sf B}_\gamma^{(\alpha,\beta)}:=\frac{-2\sin(\pi\gamma)\sin(\pi(\beta+\gamma))}{\pi\sin(\pi(\alpha+\beta+2\gamma+1))},
\\
&&\hspace{-0.8cm}{\sf C}_{\gamma,k,l}^{(\alpha,\beta)}:
=\frac{\sin(\pi\gamma)\sin(\pi(\beta+\gamma))\Gamma(\gamma+1)(\alpha+\beta+\gamma+1)_k\Gamma(2\beta+k-l)\Gamma(\alpha-\beta+l)(-\beta-\gamma)_l}{2^{2\beta-1}\sin(\pi(\alpha+\beta+2\gamma+1))\Gamma(\beta)(-\gamma)_k(k-l)!\,l!\,\Gamma(\alpha+\gamma+1+l)},\\
&&\hspace{-0.8cm}{\sf D}_\gamma^{(\alpha,\beta)}:
=
\frac{-2\sin(\pi(\alpha+\gamma))\sin(\pi(\beta+\gamma))\Gamma(\alpha+\gamma+1)\Gamma(\beta+\gamma+1)}
{\pi\sin(\pi(\alpha+\beta+2\gamma+1))\Gamma(\gamma+1)\Gamma(\alpha+\beta+\gamma+1)},\\
&&\hspace{-0.8cm}
{\sf E}_{\gamma,k,l}^{(\alpha,\beta)}:=\frac{\sin(\pi(\alpha+\gamma))\sin(\pi(\beta+\gamma))\Gamma(\beta+\gamma+1)\Gamma(2\beta+k-l)\Gamma(\alpha-\beta+l)}
{2^{2\beta-1}\sin(\pi(\alpha+\beta+2\gamma+1))\Gamma(\beta)\Gamma(\alpha+\beta+\gamma+1)(k-l)!\,l!}.
\end{eqnarray}
Now consider the asymptotics of all four terms 
as $z_2\to\infty$. 
The asymptotic behavior of $Z^\pm$ as $z_2\to\infty$ is $Z^\pm\sim z_2^2$. The behavior
of the Jacobi function of the second kind
as the argument $|z|\to\infty$ is \eqref{Qinfty}
\[
Q_\gamma^{(\alpha,\beta)}(z)\sim \frac{1}{z^{\alpha+\beta+\gamma+1}}.
\]
Therefore, one has
the following asymptotic behavior considered as functions of 
$\zeta=z_>$ with $z_<$ fixed, 
\begin{eqnarray}
&&\hspace{-8.6cm}Q_\gamma^{(\alpha,\beta)}(Z^\pm)
\sim {(Z^\pm)}^{-\gamma-\alpha-\beta-1}\sim \zeta^{-2\gamma-2\alpha-2\beta-2},\\
&&\hspace{-8.6cm}
Q_{-\alpha-\beta-\gamma-1}^{(\alpha,\beta)}(Z^\pm)\sim (Z^\pm)^\gamma\sim \zeta^{2\gamma},\\
&&\hspace{-8.6cm}Q_{\gamma-k}^{(\alpha+k+l,\beta+k-l)}(\zeta)
\sim\zeta^{-\gamma-\alpha-\beta-k-1},\\
&&\hspace{-8.6cm}Q_{-\alpha-\beta-\gamma-1-k}^{(\alpha+k+l,\beta+k-l)}(\zeta)\sim\zeta^{\gamma-k}.
\end{eqnarray}
The above relation \eqref{leadrel} with leading order asymptotic contribution as $z_2\to\infty$ taken out
can be written as a function of two analytic functions
${\sf f}(z_2^{-1}):={\sf f}_{\gamma,k,l}^{(\alpha,\beta)}(z_1,z_2^{-1})$,
${\sf g}(z_2^{-1}):={\sf g}_{\gamma,k,l}^{(\alpha,\beta)}(z_1,z_2^{-1})$,
as
\begin{equation}
z_2^{-2(\gamma+\alpha+\beta+1)} {\sf f}(z_2^{-1})=z_2^{2\gamma} {\sf g}(z_2^{-1}).
\label{fgeqnzero}
\end{equation}
For $4\Re\gamma\not\in -2(\alpha+\beta+1)+\mathbb{Z}$,
the only way the equation
can be true is if ${\sf f}={\sf g}$
identically vanishes. The case of general $\gamma$ then follows by analytic continuation in $\gamma$.
 Therefore we have now verified separately all terms in the double series expansion of the Jacobi function of the second kind given in \eqref{addQhyp}.
Given this, and proof of overall
convergence by Koornwinder and Flensted-Jensen \cite[Theorem 2.1]{FlenstedJensenKoorn79},
this completes the proof.
\end{proof}

\begin{rem}
If you apply \eqref{QtoP} to the Jacobi functions of the second kind on the left-hand side and right-hand side of \eqref{addQhyp}, then it becomes the hyperbolic context of the addition theorem for Jacobi polynomials. 
\end{rem}

\begin{cor}
\label{IntQdblmsr}
Let $k,l\in\N_0$ with $l\le k$, $\gamma,\alpha,\beta\in\mathbb C$, $z_1,z_2\in\mathbb C\setminus(-\infty,1]$,
$Z^\pm$
defined in \eqref{ZZdef},
such that $x=\cos\phi$ and the complex variables $\gamma,\alpha,\beta,z_1,z_2$ are in some
yet to be determined neighborhood of the real line.
Then
\begin{eqnarray}
&&\hspace{-0.7cm}
\int_0^1 \!\int_0^\pi Q_\gamma^{(\alpha,\beta)}(Z^\pm)\, w^{k-l}P_l^{(\alpha-\beta-1,\beta+k-l)}(2w^2\!-\!1)C_{k-l}^\beta(\cos\phi)
 \,\dd m^{(\alpha,\beta)}(w,\phi),\nonumber\\[0.15cm]
&&\hspace{0.00cm}={(\pm 1)^{k+l}}\,{\sf A}^{(\alpha,\beta,\gamma)}_{k,l}
(z_1z_2)^{k-l} ((z_1^2-\!1)(z_2^2\!-\!1))^{\frac12(k+l)} P_{\gamma-k}^{(\alpha+k+l,\beta+k-l)}(2z_<^2\!-\!1)
Q_{\gamma-k}^{(\alpha+k+l,\beta+k-l)}(2z_>^2\!-\!1),
\end{eqnarray}
where {${\sf A}_{k,l}^{(\alpha,\beta,\gamma)}$ is defined in 
\eqref{Akl}.}
\end{cor}
\begin{proof}
This follows directly from \eqref{fgeqnzero} since in the proof of Theorem \ref{AddnhypQgen}, we showed that ${\sf f}=0$. The result ${\sf g}=0$ is equivalent to this result under the transformation $\gamma\mapsto-\gamma-\alpha-\beta-1$.
\end{proof}

\begin{rem}
The above integral representations Theorem \ref{IntPdblmsr} and Corollary \ref{IntQdblmsr} are equivalent to the double summation addition theorems for the Jacobi function of the first kind \eqref{addPhypgen} and second kind, Theorem \ref{AddnhypQgen}.
\end{rem}

\begin{rem}
One has the following  well-known product representations which are 
the $k=l=0$ contribution
of the integral representations Theorem \ref{IntPdblmsr} and Corollary \ref{IntQdblmsr}.
Let $x=\cos\phi$, $\gamma,\alpha,\beta\in\mathbb 
C$, $z_1,z_2\in\mathbb C\setminus(-\infty,1]$, 
$Z^\pm$ defined in \eqref{ZZdef},
such that the complex variables 
$\gamma,\alpha,\beta,z_1,z_2$ are in some
yet to be determined neighborhood of the real line.
Then
\begin{eqnarray}
\label{PPint}
&&\hspace{-0.1cm}P_\gamma^{(\alpha,\beta)}(2z_1^2\!-\!1)
P_\gamma^{(\alpha,\beta)}(2z_2^2\!-\!1)\!=\!
\frac{2\Gamma(\alpha\!+\!\gamma\!+\!1)}{\sqrt{\pi}\,\Gamma(\gamma\!+\!1)\Gamma(\beta\!+\!\frac12)\Gamma(\alpha\!-\!\beta)}
\int_0^1
\int_0^{\pi}
\!P_\gamma^{(\alpha,\beta)}(Z^\pm)\,
\dd m^{(\alpha,\beta)}(w,\phi),\\
\label{PQint}
&&\hspace{-0.1cm}P_\gamma^{(\alpha,\beta)}(2z_<^2\!-\!1)
Q_\gamma^{(\alpha,\beta)}(2z_>^2\!-\!1)=
\frac{2\Gamma(\alpha\!+\!\gamma\!+\!1)}{\sqrt{\pi}\,\Gamma(\gamma\!+\!1)\Gamma(\beta\!+\!\frac12)\Gamma(\alpha\!-\!\beta)}
\int_0^1
\int_0^{\pi}
Q_\gamma^{(\alpha,\beta)}(Z^\pm)\,
\dd m^{(\alpha,\beta)}(w,\phi).
\end{eqnarray}
For independent verification of these product representations, see  \cite[Theorem 4.1]{FlenstedJensenKoornwinder73}  
for the product formula \eqref{PPint}
and
 \cite[p.~255]{FlenstedJensenKoornwinder73}
 for the product formula \eqref{PQint}.
\end{rem}

In the case when $z_1,z_2,w,x=\cos\phi$ are real numbers, then the argument of the Jacobi
function of the second kind in the addition theorem for the Jacobi function of the second kind takes a simpler and more convenient form. 
This is analogous to the Flensted-Jensen--Koornwinder addition theorem of the first kind \eqref{addPhyp}.
We present this result now.

\begin{thm}
\label{CohladdnhypQgen}
Let $\gamma,\alpha,\beta,w\in\mathbb R$, $\alpha+\gamma,\beta+\gamma\not\in-\N$, 
$\phi\in[0,\pi]$,
${z_1,z_2}\in(1,\infty)$,
and $Z^\pm$, $z_\lessgtr$,  as defined as in \eqref{ZZrealdef}, \eqref{zlgt}, respectively.
Then
\begin{eqnarray}
&&\hspace{-0.55cm}Q_\gamma^{(\alpha,\beta)}(Z^\pm)
=\frac{\Gamma(\alpha+1)\Gamma(\gamma+1)}{\Gamma(\alpha+\gamma+1)}\sum_{k=0}^\infty
\frac{(\alpha+1)_k(\alpha+\beta+\gamma+1)_k}
{(\alpha+k)(\beta+1)_k(-\gamma)_k}
\nonumber\\
&&\hspace{-0.2cm}\times\sum_{l=0}^k
(\pm 1)^{k+l}
\frac{(\alpha+k+l)(-\beta-\gamma)_l}{(\alpha+\gamma+1)_l}
({z_1z_2})^{k-l}\left({(z_1^2-1)}{(z_2^2-1)}\right)^{\frac{k+l}{2}} \nonumber\\[0.1cm]
&&\hspace{0.0cm}\times
P_{\gamma-k}^{(\alpha+k+l,\beta+k-l)}(2z_<^2\!-\!1)
Q_{\gamma-k}^{(\alpha+k+l,\beta+k-l)}(2z_>^2\!-\!1)
w^{k-l}P_{l}^{(\alpha-\beta-1,\beta+k-l)}(2w^2\!-\!1)
\frac{\beta\!+\!k\!-\!l}{\beta}C_{k-l}^\beta(\cos\phi).
\label{addQhypd1}
\end{eqnarray}
\end{thm}
\begin{proof}
This follows from
Theorem \ref{AddnhypQgen}
by setting the variables
$\gamma,\alpha,\beta,z_1,z_2,w,x=\cos\phi$ to real numbers.
\end{proof}
\subsection{The trigonometric context of the addition theorem for the Jacobi function of the second kind}

In the trigonometric context for the addition theorem for Jacobi functions of the second kind, one must then use the Jacobi function of the second kind on-the-cut ${\sf Q}_\gamma^{(\alpha,\beta)}(x)$ \eqref{Qcutdef}, which are
defined in Section \ref{Jac2cut} and
have a hypergeometric representation
given by \eqref{Qcut1}. Note that this representation is not unique and there are many other double Gauss hypergeometric representations of this function. For more about this see the discussion immediately above \eqref{Qcut1}.
Define 
\begin{equation}
\hspace{0.0cm}x_\lessgtr:=\hspace{-0.5cm}\begin{array}{rcl}
&\min{}&\\[-0.1cm]&\max{}&
\end{array}\hspace{-0.5cm}\{{x_1,x_2}\}, 
\label{xlgt}
\end{equation} 
where ${x_1,x_2}\in(-1,1)$,
and in the case where ${x_1,x_2}\in\CC$, then if one takes without loss of generality ${x_1}=x_>$ to lie on an ellipse with foci at $\pm 1$, then ${x_2}=x_<$ must be chosen to be in the interior of that ellipse.

{
\begin{thm}
\label{IntQcdblmsr}
Let $k,l\in\N_0$, $l\le k$, $\gamma,\alpha,\beta\in\mathbb C$, 
$x_1,x_2\in\mathbb C\setminus((-\infty,-1]\cup[1,\infty))$,
$\alpha\not\in\Z$, $\alpha+\gamma,\beta+\gamma\not\in-\N$, 
${\sf X}^\pm$, $x_\lessgtr$ as defined
in \eqref{XXdef}, \eqref{xlgt} respectively,
such that the complex variables $\gamma,\alpha,\beta,x_1,x_2$ are in some
yet to be determined neighborhood of the real line.
Then
\begin{eqnarray}
&&\hspace{-0.7cm}
\int_0^1 \!\int_0^\pi {\sf Q}_\gamma^{(\alpha,\beta)}(X^\pm) \,w^{k-l}P_l^{(\alpha-\beta-1,\beta+k-l)}(2w^2\!-\!1)C_{k-l}^\beta(\cos\phi)
 \,\dd m^{(\alpha,\beta)}(w,\phi)\nonumber\\[0.15cm]
&&\hspace{-0.20cm}=(\mp 1)^{k+l}{\sf A}^{(\alpha,\beta,\gamma)}_{k,l}
(x_1x_2)^{k-l} ((1\!-\!x_1^2)(1\!-\!x_2^2))^{\frac12(k+l)} {\sf Q}_{\gamma-k}^{(\alpha+k+l,\beta+k-l)}(2x_<^2\!-\!1)
{\sf P}_{\gamma-k}^{(\alpha+k+l,\beta+k-l)}(2x_>^2\!-\!1),
\end{eqnarray}
where {${\sf A}_{k,l}^{(\alpha,\beta,\gamma)}$ is defined in 
\eqref{Akl}.}
\end{thm}
}
{
\begin{proof}
Starting with 
Corollary \ref{IntQdblmsr} and directly applying 
\eqref{Qcutdef}
completes the proof.
\end{proof}
}

{
\begin{cor}
\label{IntQcdblmsrconst}
Let  $\gamma,\alpha,\beta\in\mathbb C$, 
$x_1,x_2\in\mathbb C\setminus((-\infty,-1]\cup[1,\infty))$,
$\alpha\not\in\Z$, $\alpha+\gamma,\beta+\gamma\not\in-\N$, 
${\sf X}^\pm$, $x_\lessgtr$ as defined
in \eqref{XXdef}, \eqref{xlgt} respectively,
such that the complex variables $\gamma,\alpha,\beta,x_1,x_2$ are in some
yet to be determined neighborhood of the real line.
Then
\begin{eqnarray}
&&\hspace{-0.7cm}
\int_0^1 \!\int_0^\pi {\sf Q}_\gamma^{(\alpha,\beta)}(X^\pm)
 \,\dd m^{(\alpha,\beta)}(w,\phi)
=\frac{\sqrt{\pi}\,\Gamma(\gamma+1)\Gamma(\beta+\frac12)\Gamma(\alpha-\beta)}{2\Gamma(\alpha+\gamma+1)}
 {\sf Q}_{\gamma}^{(\alpha,\beta)}(2x_<^2\!-\!1)
{\sf P}_{\gamma}^{(\alpha,\beta)}(2x_>^2\!-\!1).
\end{eqnarray}
\end{cor}
}
{
\begin{proof}
Starting with Theorem \ref{IntQcdblmsr} and setting $k=l=0$ completes the proof.
\end{proof}
}

\begin{thm}
\label{AddntrigQgen}
Let $\gamma,\alpha,\beta\in\mathbb C$, 
$x_1,x_2\in\mathbb C\setminus((-\infty,-1]\cup[1,\infty))$,
$\alpha\not\in\Z$, $\alpha+\gamma,\beta+\gamma\not\in-\N$, 
$x,w\in\mathbb C$ with
${\sf X}^\pm$, $x_\lessgtr$ as defined
in \eqref{XXdef}, \eqref{xlgt} respectively,
such that the complex variables $\gamma,\alpha,\beta,x_1,x_2,x,w$ are in some
yet to be determined neighborhood of the real line.
Then
\begin{eqnarray}
&&\hspace{-0.40cm}{\sf Q}_\gamma^{(\alpha,\beta)}({\sf X}^\pm)
=\frac{\Gamma(\alpha+1)\Gamma(\gamma+1)}{\Gamma(\alpha+\gamma+1)}\sum_{k=0}^\infty
\frac{(\alpha+1)_k(\alpha+\beta+\gamma+1)_k}
{(\alpha+k)(\beta+1)_k(-\gamma)_k}
\nonumber\\
&&\hspace{0.2cm}\times\sum_{l=0}^k
(\mp 1)^{k-l}
\frac{(\alpha+k+l)(-\beta-\gamma)_l}{(\alpha+\gamma+1)_l}
({x_1x_2})^{k-l}\left(({1-x_1^2})({1-x_2^2})\right)^{\frac{k+l}{2}}
\nonumber\\[0.1cm]
&&\hspace{0.7cm}\times
{\sf Q}_{\gamma-k}^{(\alpha+k+l,\beta+k-l)}(2x_<^2\!-\!1)
{\sf P}_{\gamma-k}^{(\alpha+k+l,\beta+k-l)}(2x_>^2\!-\!1)
w^{k-l}P_{l}^{(\alpha-\beta-1,\beta+k-l)}(2w^2\!-\!1)\frac{\beta\!+\!k\!-\!l}{\beta}C_{k-l}^\beta(x) .
\label{addQhypd2}
\end{eqnarray}
\end{thm}
\begin{proof}
The result follows by starting with the addition theorem 
for Jacobi functions of
the second kind \eqref{addQhypd1} and applying the definition 
\eqref{Qcutdef}
completes the proof.
\end{proof}

In the case when $z_1,z_2,w,x=\cos\phi$ are real numbers, then the argument of the Jacobi
function of the second kind on-the-cut in the addition theorem for the Jacobi function of the second kind takes a simpler and more convenient form. 
This is analogous to the addition theorem \eqref{addPhyp}.
We present this result now.

\begin{cor}
\label{Cohladdnhyp}
Let $\gamma,\alpha,\beta,w\in\mathbb R$, 
$\alpha\not\in\Z$, $\alpha+\gamma,\beta+\gamma\not\in-\N$, 
${x_1,x_2}\in(-1,1)$, $\phi\in[0,\pi]$, 
with ${\sf X}^\pm$, $x_\lessgtr$ as defined 
in \eqref{XXrealdef}, \eqref{xlgt} respectively.
Then
\begin{eqnarray}
&&\hspace{-0.55cm}{\sf Q}_\gamma^{(\alpha,\beta)}({\sf X}^\pm)
=\frac{\Gamma(\alpha+1)\Gamma(\gamma+1)}{\Gamma(\alpha+\gamma+1)}\sum_{k=0}^\infty
\frac{(\alpha+1)_k(\alpha+\beta+\gamma+1)_k}
{(\alpha+k)(\beta+1)_k(-\gamma)_k}
\nonumber\\
&&\hspace{0.0cm}\times\sum_{l=0}^k
(\mp 1)^{k-l}
\frac{(\alpha+k+l)(-\beta-\gamma)_l}{(\alpha+\gamma+1)_l}
({x_1x_2})^{k-l}\left({(1-x_1^2)(1-x_2^2)}\right)^{\frac{k+l}{2}}
 \nonumber\\[0.1cm]
&&\hspace{0.3cm}\times
{\sf Q}_{\gamma-k}^{(\alpha+k+l,\beta+k-l)}(2x_<^2\!-\!1)
{\sf P}_{\gamma-k}^{(\alpha+k+l,\beta+k-l)}(2x_>^2\!-\!1)
w^{k-l}P_{l}^{(\alpha-\beta-1,\beta+k-l)}(2w^2\!-\!1)\frac{\beta\!+\!k\!-\!l}{\beta}C_{k-l}^\beta(\cos\phi).
\label{addQtrig}
\end{eqnarray}
\end{cor}

\begin{proof}
This result follows from Theorem \ref{CohladdnhypQgen} by setting the complex variables to be real.
\end{proof}
\section{Olver normalized Jacobi functions and their addition theorems}
Koornwinder's addition theorem
for Jacobi polynomials with degree $n\in\mathbb N_0$ \eqref{addPtrigfirst} is terminating with the $k$ sum being over $k\in\{0,\ldots,n\}$. One can see this by examination of \eqref{addPhyp}, \eqref{addPtrig}
by recognizing that both Jacobi polynomials $P_{\gamma-k}^{(\alpha+k+l,\beta+k-l)}$ vanish for $\gamma=n\in\mathbb N_0$ and $k\ge n+1$. However, considering the limit as $\gamma\to n$ for all values of $k\in\mathbb N_0$ in Koornwinder's addition theorem, the factor $1/(-\gamma)_k$ blows up for $k\ge n+1$. On the other hand, this factor in the limit when multiplied by the 
Jacobi function of the first kind prefactor containing $1/\Gamma(\gamma-k+1)$, while considering the residues of the gamma function, the product will be
finite, namely
\begin{equation}
\hspace{0.0cm}\lim_{\gamma\to n}\frac{1}
{(-\gamma)_k\Gamma(\gamma-k+1)}=
\lim_{\gamma\to n}\frac{\Gamma(-\gamma)}
{\Gamma(-\gamma+k)\Gamma(\gamma-k+1)}=\frac{(-1)^k}{n!},
\label{limitfinite}
\end{equation}
for $k\ge n+1$.
But when this finite factor is multiplied by the 
second Jacobi polynomial $P_{\gamma-k}^{(\alpha+\beta
+k+l,\beta+k-l)}$ which vanishes for $k\ge n+1$, the 
resulting expression vanishes for all these $k$ values 
which results in a terminating sum over $k\in\{0,\ldots,n\}$.

Unlike Koornwinder's addition theorem for 
Jacobi polynomials, the addition theorem for the 
Jacobi functions of the second kind (see 
\S\ref{secfunsecond}) is not a terminating sum. 
One can see this by examination of \eqref{addQhypd2},
by recognizing that the Jacobi polynomials 
$P_{\gamma-k}^{(\alpha+k+l,\beta+k-l)}$ vanish for 
$\gamma=n\in\mathbb N_0$ and $k\ge n+1$. However, 
considering the limit as $\gamma\to n$ for all values 
of $k\in\mathbb N_0$ in Koornwinder's addition theorem, 
the factor $1/(-\gamma)_k$ blows up for $k\ge n+1$. 
On the other hand, this factor in the limit is multiplied by the 
Jacobi function of the first kind prefactor containing
$1/\Gamma(\gamma-k+1)$, while considering the residues 
of the gamma function, the product will be
finite, namely \eqref{limitfinite},
for $k\ge n+1$.
This finite factor is then multiplied by the Jacobi function 
of the second kind 
$Q_{\gamma-k}^{(\alpha+\beta+k+l,\beta+k-l)}(2z_>^2-1)$ 
which does not vanish for $k\ge n+1$ for 
$\alpha,\beta\not\in\mathbb Z$, unlike the case for 
Jacobi polynomials.

\subsection{Olver normalized Jacobi functions}

We previously introduced Olver's normalization of the Gauss \cite[\href{http://dlmf.nist.gov/15.2.E2}{(15.2.2)}]{NIST:DLMF} and generalized
hypergeometric function \eqref{regrFs}
(see also \cite[\href{http://dlmf.nist.gov/16.2.E5}{(16.2.5)}]{NIST:DLMF})
which results in these functions being entire functions of all of the parameters which appear including all denominator factors.
Olver applied this concept of special normalization previously to the associated Legendre function of 
the second kind
\cite[\href{http://dlmf.nist.gov/14.3.E10}{(14.3.10)}]{NIST:DLMF} (see also \cite[p.~170 
and 178]{Olver:1997:ASF}).
We now demonstrate how to apply this concept to the Jacobi functions of the first and second kind.

In the above description, instead of carefully determining the limits of the 
relevant functions when there are removable singularities due to the appearance of various gamma function prefactors, 
an alternative option is to use appropriately defined Olver normalized Jacobi 
functions and recast the addition theorems correspondingly. 
The benefit of using Olver normalized definitions of the Jacobi functions is that 
one avoids complications due to gamma functions with 
removable singularities. 
Typical examples of these benefits occur in the often 
appearing examples when one has degrees $\gamma$ and 
parameters $\alpha,\beta$ given by integers. 
In these cases, using the standard definitions such as 
those which appear in Theorems \ref{Firstthm}, \ref{thmQ}, 
the appearing functions are not defined and careful limits must be taken. However, if one adopts carefully chosen Olver normalized definitions where 
only the Olver normalized Gauss hypergeometric 
functions are used, then these functions will be 
entire for all values of the parameters. 
As we will see, by using these definitions, we arrive 
at formulas for the addition theorems which are elegant 
and highly useful! First we give our new choice of 
the Olver normalization and then give the relations of the 
Olver normalized definitions in terms of the usual definitions.
Our definitions of Olver normalized Jacobi functions of 
the first and second kind in the hyperbolic and trigonometric contexts are given by
\begin{eqnarray}
\label{hypPbdef}
&&\hspace{-1.10cm}\bm{P}_\gamma^{(\alpha,\beta)}(z):=\Ohyp21{-\gamma,\alpha+\beta+\gamma+1}{\alpha+1}{\frac{1-z}{2}},\\
\label{hypQbdef}
&&\hspace{-1.10cm}\bm{Q}_\gamma^{(\alpha,\beta)}(z):=\frac{2^{\alpha+\beta+\gamma}}{(z-1)^{\alpha+\gamma+1}(z+1)^\beta}\Ohyp21{\gamma+1,\alpha+\gamma+1}{\alpha+\beta+2\gamma+2}{\frac{2}{1-z}},\\
\label{trigPbdef}
&&\hspace{-1.10cm}\bm{\mathsf{P}}_\gamma^{(\alpha,\beta)}(x):=\Ohyp21{-\gamma,\alpha+\beta+\gamma+1}{\alpha+1}{\frac{1-x}{2}},\\
\label{trigQbdef}
&&\hspace{-1.05cm}\bm{\mathsf{Q}}_\gamma^{(\alpha,\beta)}(x):=
\tfrac12\Gamma(\alpha+1)\left(\frac{1+x}{2}\right)^\gamma
\biggl(
\frac{\cos(\pi\alpha)\Gamma(\alpha+\gamma+1)}{\Gamma(\gamma+1)}
\Ohyp21{-\gamma,-\beta-\gamma}{1+\alpha}{\frac{x-1}{x+1}}\nonumber\\
&&\hspace{1.7cm}-\frac{\Gamma(\beta+\gamma+1)}{\Gamma(\alpha+\beta+\gamma+1)}\!
\left(\frac{1+x}{1-x}\right)^\alpha\!
\Ohyp21{-\alpha-\gamma,-\alpha-\beta-\gamma}{1-\alpha}{\frac{x-1}{x+1}}\biggr).
\end{eqnarray}
Therefore one has the following connection relations between the Jacobi functions of the first and second kinds and their Olver normalized counterparts, namely
\begin{eqnarray}
&&\hspace{-7.20cm}{P}_\gamma^{(\alpha,\beta)}(z)=\frac{\Gamma(\alpha+\gamma+1)}{\Gamma(\gamma+1)}\bm{P}_\gamma^{(\alpha,\beta)}(z),
\label{bmPconn}\\
&&\hspace{-7.20cm}{Q}_\gamma^{(\alpha,\beta)}(z)=\Gamma(\alpha+\gamma+1)\Gamma(\beta+\gamma+1)\,\bm{Q}_\gamma^{(\alpha,\beta)}(z),
\label{bmQconn}\\
&&\hspace{-7.20cm}{\sf P}_\gamma^{(\alpha,\beta)}(x)=\frac{\Gamma(\alpha+\gamma+1)}{\Gamma(\gamma+1)}\bm{\mathsf P}_\gamma^{(\alpha,\beta)}(x).
\label{bmPcutconn}
\end{eqnarray}
Note that 
\begin{equation}
\hspace{0.5cm}\bm{\mathsf P}_\gamma^{(\alpha,\beta)}(x)
=\bm{P}_\gamma^{(\alpha,\beta)}(x\pm i0),
\end{equation}
as in \eqref{JacPcutdef}.
Furthermore in the special case $\gamma=0$ one has
\begin{eqnarray}
&&\hspace{-5.6cm}\bm{Q}_0^{(\alpha,\beta)}(z):=\frac{2^{\alpha+\beta}}{(z-1)^{\alpha+1}(z+1)^\beta}\Ohyp21{1,\alpha+1}{\alpha+\beta++2}{\frac{2}{1-z}}.
\end{eqnarray}

\begin{rem}
As of the date of publication of this manuscript, we have been unable to find an Olver normalized version of the Jacobi function of the second kind on-the-cut
$\bm{\mathsf{Q}}_\gamma^{(\alpha,\beta)}(x)$. However we did find
a special normalization of this
function which works well 
when $\gamma=0$ and the $\beta$ parameters
take integer values which is of 
particular importance because they appear in a very important application
(see Section \ref{application} below).
Let $b\in\N_0$. Define
\begin{equation}
\hspace{0.0cm}\mathcal{Q}_{-k}^{(\alpha+k+l,b+k-l)}(x):=\lim_{\gamma\to0,\beta\to b}
(-\beta-\gamma)_l{\sf Q}_{\gamma-k}^{(\alpha+k+l,\beta+k-l)}(x),
\label{limfunQ}
\end{equation}
which is a well-defined function for all $\alpha,b,x,k,l$ in its domain.
\end{rem}

\subsection{Addition theorems for the Olver normalized Jacobi functions}

Now that we've introduced the Olver normalized Jacobi functions of the first and second kind in the hyperbolic and trigonometric contexts, we are in a position to perform the straightforward derivation of the corresponding addition theorems for these functions.

\begin{thm}
\label{AddnhypQgend1}
Let $\gamma,\alpha,\beta\in\mathbb C$, $z_1,z_2\in\mathbb C\setminus(-\infty,1]$,
$x_1,x_2\in\mathbb C\setminus((-\infty,-1]\cup[1,\infty))$,
$x,w\in\mathbb C$, and 
$Z^\pm$, ${\sf X}^\pm$ as defined in 
\eqref{ZZdef}, \eqref{XXdef}, respectively,
such that the complex variables $\gamma,\alpha,\beta,z_1,z_2,x_1,x_2,x,w$ are in some
yet to be determined neighborhood of the real line.
Then
\begin{eqnarray}
&&\hspace{-0.55cm}P_\gamma^{(\alpha,\beta)}(Z^\pm)=
\frac{\Gamma(\alpha+1)\Gamma(\alpha+\gamma+1)}
{\Gamma(\gamma+1)}
\sum_{k=0}^\infty
\frac{(\alpha+1)_k(\alpha+\beta+\gamma+1)_k(-\gamma)_k}{(\alpha+k)(\beta+1)_k}\nonumber\\[-0.15cm]
&&\hspace{0.25cm}\times
\sum_{l=0}^k
(\mp 1)^{k-l}(\alpha+k+l)(\alpha+\gamma+1)_l(-\beta-\gamma)_l(z_1z_2)^{k-l}((z_1^2-1)(z_2^2-1))^{\frac{k+l}{2}}\nonumber\\
&&\hspace{0.25cm}\times
\bm{P}_{\gamma-k}^{(\alpha+k+l,\beta+k-l)}(2z_1^2\!-\!1)
\bm{P}_{\gamma-k}^{(\alpha+k+l,\beta+k-l)}(2z_2^2\!-\!1)w^{k-l}P_l^{(\alpha-\beta-1,\beta+k-l)}(2w^2\!-\!1)\frac{\beta\!+\!k\!-\!l}{\beta} C_{k-l}^\beta(x),\\[0.3cm]
&&\hspace{-0.55cm}Q_\gamma^{(\alpha,\beta)}(Z^\pm)=
\Gamma(\alpha+1)\Gamma(\alpha+\gamma+1)
\Gamma(\beta+\gamma+1)
\sum_{k=0}^\infty
\frac{(\alpha+1)_k(\alpha+\beta+\gamma+1)_k}{(\alpha+k)(\beta+1)_k}\nonumber\\[-0.1cm]
&&\hspace{0.25cm}\times
\sum_{l=0}^k
(\mp 1)^{k-l}(\alpha+k+l)(\alpha+\gamma+1)_l(z_1z_2)^{k-l}((z_1^2-1)(z_2^2-1))^{\frac{k+l}{2}}\nonumber\\
&&\hspace{0.25cm}\times
\bm{Q}_{\gamma-k}^{(\alpha+k+l,\beta+k-l)}(2z_>^2\!-\!1)
\bm{P}_{\gamma-k}^{(\alpha+k+l,\beta+k-l)}(2z_<^2\!-\!1)w^{k-l}P_l^{(\alpha-\beta-1,\beta+k-l)}(2w^2\!-\!1)\frac{\beta\!+\!k\!-\!l}{\beta} C_{k-l}^\beta(x),
\label{Qregularized}
\end{eqnarray}

\begin{eqnarray}
&&\hspace{-0.55cm}{\sf P}_\gamma^{(\alpha,\beta)}({\sf X}^\pm)=
\frac{\Gamma(\alpha+1)\Gamma(\alpha+\gamma+1)}
{\Gamma(\gamma+1)}
\sum_{k=0}^\infty
\frac{(\alpha+1)_k(\alpha+\beta+\gamma+1)_k(-\gamma)_k}{(\alpha+k)(\beta+1)_k}\nonumber\\[-0.15cm]
&&\hspace{0.25cm}\times
\sum_{l=0}^k
(\mp 1)^{k-l}(\alpha+k+l)(\alpha+\gamma+1)_l(-\beta-\gamma)_l(x_1x_2)^{k-l}((1-x_1^2)(1-x_2^2))^{\frac{k+l}{2}}\nonumber\\
&&\hspace{0.25cm}\times
{\bm{\mathsf P}}_{\gamma-k}^{(\alpha+k+l,\beta+k-l)}(2x_1^2\!-\!1)
{\bm{\mathsf P}}_{\gamma-k}^{(\alpha+k+l,\beta+k-l)}(2x_2^2\!-\!1)w^{k-l}P_l^{(\alpha-\beta-1,\beta+k-l)}(2w^2\!-\!1)\frac{\beta\!+\!k\!-\!l}{\beta} C_{k-l}^\beta(x),\\[0.3cm]
&&\hspace{-0.55cm}
{\mathsf{Q}}_\gamma^{(\alpha,\beta)}({\sf X}^\pm)=
\Gamma(\alpha+1)
\sum_{k=0}^\infty
\frac{(-1)^k(\alpha+\beta+\gamma+1)_k(\alpha+1)_k}{(\alpha+k)(\beta+1)_k}\nonumber\\[-0.1cm]
&&\hspace{0.25cm}\times
\sum_{l=0}^k(\mp 1)^{k-l}
(\alpha+k+l)(-\beta-\gamma)_l(x_1x_2)^{k-l}((1-x_1^2)(1-x_2^2))^{\frac{k+l}{2}}\nonumber\\
&&\hspace{0.25cm}\times{\mathsf{Q}}_{\gamma-k}^{(\alpha+k+l,\beta+k-l)}(2x_<^2\!-\!1)
\bm{\mathsf P}_{\gamma-k}^{(\alpha+k+l,\beta+k-l)}(2x_>^2\!-\!1)w^{k-l}P_l^{(\alpha-\beta-1,\beta+k-l)}(2w^2\!-\!1)\frac{\beta\!+\!k\!-\!l}{\beta} C_{k-l}^\beta(x).
\label{Qcutregularized}
\end{eqnarray}
\end{thm}

\begin{proof}
Substituting 
\eqref{bmPconn}--\eqref{bmPcutconn}
into 
\eqref{addPhypgen}, 
\eqref{addPtriggen},
\eqref{addQhyp},
\eqref{addQhypd2} as necessary
completes the proof.
\end{proof}

There are also the corresponding expansions that are 
sometimes useful with the $l$ sum reversed, i.e., 
making the replacement $l'=k-l$ and then replacing 
$l'\mapsto l$ in Theorem \ref{AddnhypQgen}. 
These are given as follows.
\begin{cor}
\label{AddnhypPQrev}
Let $\gamma,\alpha,\beta\in\mathbb 
C$, $z_1,z_2\in\mathbb C\setminus(-\infty,1]$,
$x_1,x_2\in\mathbb C\setminus((-\infty,-1]\cup[1,\infty))$,
$x,w\in\mathbb C$, with 
$Z^\pm$, ${\sf X}^\pm$ as defined in 
\eqref{ZZdef}, \eqref{XXdef}, respectively,
and
the complex variables $\gamma,\alpha,\beta,z_1,z_2,x_1,x_2,x,w$ 
are in some yet to be determined neighborhood of the real line.
Then
\begin{eqnarray}
&&\hspace{-0.55cm}P_\gamma^{(\alpha,\beta)}(Z^\pm)=
\frac{\Gamma(\alpha+1)\Gamma(\alpha+\gamma+1)}
{\Gamma(\gamma+1)}
\sum_{k=0}^\infty
\frac{(\alpha+1)_k(\alpha+\beta+\gamma+1)_k(-\gamma)_k}{(\alpha+k)(\beta+1)_k}\nonumber\\[-0.15cm]
&&\hspace{0.25cm}\times
\sum_{l=0}^k
(\mp 1)^{l}(\alpha+2k-l)(\alpha+\gamma+1)_{k-l}(-\beta-\gamma)_{k-l}(z_1z_2)^{l}((z_1^2-1)(z_2^2-1))^{\frac{2k-l}{2}}\nonumber\\
&&\hspace{1.25cm}\times\bm{P}_{\gamma-k}^{(\alpha+2k-l,\beta+l)}(2z_1^2\!-\!1)
\bm{P}_{\gamma-k}^{(\alpha+2k-l,\beta+l)}(2z_2^2\!-\!1)
w^{l}P_{k-l}^{(\alpha-\beta-1,\beta+l)}(2w^2\!-\!1)\frac{\beta+l}{\beta} C_{l}^\beta(x),\\[0.3cm]
&&\hspace{-0.55cm}Q_\gamma^{(\alpha,\beta)}(Z^\pm)=
\Gamma(\alpha+1)\Gamma(\alpha+\gamma+1)
\Gamma(\beta+\gamma+1)
\sum_{k=0}^\infty
\frac{(\alpha+1)_k(\alpha+\beta+\gamma+1)_k}{(\alpha+k)(\beta+1)_k}\nonumber\\[-0.1cm]
&&\hspace{0.25cm}\times
\sum_{l=0}^k
(\mp 1)^{l}(\alpha+2k-l)(\alpha+\gamma+1)_{k-l}(z_1z_2)^{l}((z_1^2-1)(z_2^2-1))^{\frac{2k-l}{2}}\nonumber\\
&&\hspace{1.25cm}\times
\bm{Q}_{\gamma-k}^{(\alpha+2k-l,\beta+l)}(2z_>^2\!-\!1)
\bm{P}_{\gamma-k}^{(\alpha+2k-l,\beta+l)}(2z_<^2\!-\!1)w^{l}P_l^{(\alpha-\beta-1,\beta+l)}(2w^2\!-\!1)\frac{\beta+l}{\beta} C_{l}^\beta(x),
\label{Qregularizedd1}
\end{eqnarray}

\begin{eqnarray}
&&\hspace{-0.55cm}{\sf P}_\gamma^{(\alpha,\beta)}({\sf X}^\pm)=
\frac{\Gamma(\alpha+1)\Gamma(\alpha+\gamma+1)}
{\Gamma(\gamma+1)}
\sum_{k=0}^\infty
\frac{(\alpha+1)_k(\alpha+\beta+\gamma+1)_k(-\gamma)_k}{(\alpha+k)(\beta+1)_k}\nonumber\\[-0.15cm]
&&\hspace{0.25cm}\times
\sum_{l=0}^k
(\mp 1)^{l}(\alpha+2k-l)(\alpha+\gamma+1)_{k-l}(-\beta-\gamma)_{k-l}(x_1x_2)^{l}((x_1^2-1)(x_2^2-1))^{\frac{2k-l}{2}}\nonumber\\
&&\hspace{1.25cm}\times
\bm{\mathsf P}_{\gamma-k}^{(\alpha+2k-l,\beta+l)}(2x_1^2\!-\!1)
\bm{\mathsf P}_{\gamma-k}^{(\alpha+2k-l,\beta+l)}(2x_2^2\!-\!1)
w^{l}P_{k-l}^{(\alpha-\beta-1,\beta+l)}(2w^2\!-\!1)\frac{\beta+l}{\beta} C_{l}^\beta(x),
\label{Pcutregularizedd1}\\[0.3cm]
&&\hspace{-0.55cm}{\mathsf Q}_\gamma^{(\alpha,\beta)}({\sf X}^\pm)=
\Gamma(\alpha+1)
\sum_{k=0}^\infty(-1)^k
\frac{(\alpha+1)_k(\alpha+\beta+\gamma+1)_k}{(\alpha+k)(\beta+1)_k}\nonumber\\[-0.1cm]
&&\hspace{0.25cm}\times
\sum_{l=0}^k
(\mp 1)^{l}(\alpha+2k-l)(-\beta-\gamma)_{k-l}(x_1x_2)^{l}((x_1^2-1)(x_2^2-1))^{\frac{2k-l}{2}}\nonumber\\
&&\hspace{1.25cm}\times
{\mathsf Q}_{\gamma-k}^{(\alpha+2k-l,\beta+l)}(2x_<^2\!-\!1)
\bm{\mathsf P}_{\gamma-k}^{(\alpha+2k-l,\beta+l)}(2x_>^2\!-\!1)
w^{l}P_{k-l}^{(\alpha-\beta-1,\beta+l)}(2w^2\!-\!1)\frac{\beta+l}{\beta} C_{l}^\beta(x).
\label{Qcutregularizedd1}
\end{eqnarray}
\end{cor}
\begin{proof}
Making the replacement $l\mapsto k-l$ in 
Theorem \ref{AddnhypQgend1} completes the proof.
\end{proof}

By examining the expansion of the Jacobi function of 
the first kind, one can see that the $(-\gamma)_k$ 
shifted factorial in this alternative expansion is 
moved from the denominator to the numerator, so it is 
more natural for Jacobi polynomials where the sum is 
terminating.
One can see the benefit is that all the functions 
involved in the expansions are well-defined for all 
values of the parameters, including integer values. 
These expansions are extremely useful for expansions 
of fundamental solutions of rank-one symmetric spaces 
where all the degrees and parameters are given by integers. 
One no longer has any difficulties with the various 
functions not being defined for certain parameter values. 
This is completely resolved. One example is that in 
the integer context for the Jacobi function of 
the second kind, the functions appear with degree equal 
to $\gamma-k$ for all $k\in\mathbb N_0$. 
These quickly become undefined for negative values 
of the degree. However, since the Olver normalized Jacobi 
functions are entire functions, there is no longer 
any problem here. These alternative expansions are 
highly desirable!

\subsection{Application to the non-compact and compact symmetric spaces of rank one}
\label{application}

Let $d=\dim_{\mathbb{R}}\mathbb{K}$, where ${\mathbb{K}}$ is equal to either the real numbers $\R$, the complex numbers $\CC$, the quaternions $\HH$ or the octonions $\OO$.
For $d=1$, namely the real case, these are Riemannian manifolds of constant curvature which include Euclidean $\mathbb R^n$ space, real hyperbolic geometry $\mathbb R\Hi_R^n$ (noncompact) and real hyperspherical (compact) geometry $\mathbb R\Si_R^n$, in various models. For $d\in\{2,4,8\}$,
it is well known that there exists isotropic Riemannian manifolds of both noncompact and compact type which are referred to as the rank one symmetric spaces, see for instance \cite{Helgason84}.
These include the symmetric spaces given by the complex hyperbolic ${\CC\Hi^n_R}$, quaternionic hyperbolic ${\HH\Hi^n_R}$, and the octonionic hyperbolic plane ${\OO\Hi^2_R}$ and the 
complex projective space ${\CC\Pii^n_R}$, quaternionic projective space ${\HH\Pii^n_R}$, and the octonionic projective (Cayley) plane ${\OO\Pii^2_R}$,
where $R>0$ is their corresponding radii of curvatures.
The complex, quaternionic and octonionic rank one symmetric spaces have real dimension given by 
$2n$, 
$4n$ 
and 
$16$.
For a description of the Riemannian manifolds given by the rank one symmetric spaces, 
see for instance \cite{Helgason59,Helgason78,Helgason84} and the references therein.

Riemannian symmetric spaces, compact and non-compact, come in infinite series (4 corresponding to simple complex groups and 7 corresponding to real simple groups) and a finite class of exceptional spaces, see \cite[p.~516, 518]{Helgason59}. Each of
those come with a commutative algebra of invariant differential operators and correspondingly, a class of eigenfunctions, the spherical functions. Which in the above discussed rank one examples are hypergeometric functions. This has been used as
a motivation for several generalizations of the classical Gauss hypergeometric functions. In particular the Heckman-Opdam
hypergeometric functions \cite{HeckmanSchlichtkrull94,
HeckmanOpdam87,HeckmanOpdam2021,Opdam1988}. Another direction for generalization is the work by Macdonald
 \cite{Macdonald2013} and related publications.

Due to the isotropy of the symmetric spaces of rank one, a fundamental solution of the Laplace-Beltrami operator on these manifolds can be obtained by solving a one-dimensional ordinary differential equation given in terms of the geodesic distance. Laplace's equation is satisfied on these manifolds when the Laplace-Beltrami operator acts on an unknown function and the result is zero. 
In geodesic polar coordinates  the Laplace-Beltrami operator is given in the rank one noncompact (hyperbolic) symmetric spaces by
\begin{eqnarray}
&&\hspace{-3.5cm}\Delta= \frac{1}{R^2}\left\{\frac{\partial^2}{\partial r^2}+\left[d(n-1)\coth r+2(d-1)\coth (2 r )\right]
\frac{\partial}{\partial r}+\frac{1}{\sinh^2r}\Delta_{K/M}
\right\}\label{eq-Laplace1}\\
&&\hspace{-3.1cm}=
\frac{1}{R^2}
\left\{
\frac{\partial^2}{\partial r^2}+\left[(dn-1)\coth r+(d-1)\tanh r\right]
\frac{\partial}{\partial r}+\frac{1}{\sinh^2r}\Delta_{K/M}
\right\}\\
&&\hspace{-3.1cm}=:\frac{1}{R^2}\left(\Delta_r+\frac{1}{\sinh^2r}\Delta_{K/M}\right),
\end{eqnarray}
and on the rank one compact (projective) spaces it is given by 
\begin{eqnarray}
&&\hspace{-3.5cm}\Delta= \frac{1}{R^2}\left\{\frac{\partial^2}{\partial \theta^2}+\left[d(n-1)\cot \theta+2(d-1)\cot (2 \theta )\right]
\frac{\partial}{\partial \theta}+\frac{1}{\sin^2\theta}\Delta_{K/M}
\right\}\label{eq-Laplace1}\\
&&\hspace{-3.1cm}=
\frac{1}{R^2}
\left\{
\frac{\partial^2}{\partial \theta^2}+\left[(dn-1)\cot \theta+(d-1)\tan \theta\right]
\frac{\partial}{\partial \theta}+\frac{1}{\sin^2\theta}\Delta_{K/M}
\right\} \\
&&\hspace{-3.1cm}=:\frac{1}{R^2}\left(\Delta_\theta+\frac{1}{\sin^2\theta}\Delta_{K/M}\right),
\end{eqnarray}
where $r$ and $\theta$ are the geodesic distance on the noncompact and compact rank one symmetric spaces respectively (see \cite[Lemma 21]{Helgason59}).
For a spherically symmetric solution such as a fundamental solution, the contribution from $\Delta_{K/M}$ vanishes and one needs to solve Laplace's equation for radial solutions, namely
\begin{equation}
\hspace{1.0cm}\Delta_r \,u(r)=0,\quad
\Delta_\theta\, v(\theta)=0.
\end{equation}
For the solution to these equations, the homogeneous solutions to the second order ordinary differential equation which appears are given by Jacobi/hypergeometric functions (see \cite[p.~484]{Helgason84}).
It can be easily verified that a basis for radial solutions can be given by 
\begin{eqnarray}
&&\hspace{-7cm}u(r)=a\,P_0^{(\alpha,\beta)}(\cosh(2r))+b\,Q_0^{(\alpha,\beta)}(\cosh(2r)).\\
&&\hspace{-7cm}v(\theta)=c\,{\sf  P}_0^{(\alpha,\beta)}(\cos(2\theta))+d\,{\sf Q}_0^{(\alpha,\beta)}(\cos(2\theta)),
\end{eqnarray}
where on the complex, quaternionic and octonionic rank one symmetric spaces one has $\alpha\in\{n-1,2n-1,7\}$
and $\beta\in\{0,1,3\}$ respectively \cite[Table 1, p.~265]{Koornwinder79}.
Furthermore, for a fundamental solution, the solutions need to be singular at the origin and match up to a Euclidean fundamental solutions locally. This requires that the solutions should be irregular at the origin ($r=0$ and $\theta=0$). 
Therefore fundamental solutions must correspond to the solutions which are the functions of the second kind. Hence for a fundamental solution of Laplace's equation $a=c=0$ and we must determine $b$ and $d$ which will be a function of $d$, $n$ and $R$.
\begin{rem}
Note that the general homogeneous solution as a function of the geodesic coordinate includes contribution from both the function of the first kind and function of the second kind.  However, the function of the first kind with $\gamma=0$ is simply the constants $a$ and $c$ since $P_0^{(\alpha,\beta)}(z)=1$ \eqref{Pzero} (same for the functions on-the-cut). 
On the other hand, in the case of non-spherically symmetric homogeneous solutions there will be contributions due to the function of the first kind because then the contribution to the $\Delta_{K/M}$ term will be non-zero.
\end{rem}
Let $\bfx,\bfxp\in\R^s$, then a Euclidean fundamental solution of Laplace's equation is given by (see for instance \cite[p.~202]{GelfandShilov})
\begin{equation}
\mcg^s({\bf x},{\bf x}^\prime)=
\left\{ \begin{array}{ll}
\displaystyle\frac{\Gamma(s/2)}{2\pi^{s/2}(s-2)}\|{\bf x}-{\bf x}^\prime\|^{2-s}
& \qquad\mathrm{if}\ s=1\mathrm{\ or\ }s\ge 3,\\[10pt]
\displaystyle\frac{1}{2\pi}\log\|{\bf x}-{\bf x}^\prime\|^{-1}
& \qquad\mathrm{if}\  s=2. 
\end{array} \right.
\label{Eucfundsol}
\end{equation}
For a description of opposite antipodal fundamental solutions on the real hypersphere see \cite{Cohlhypersphere}.
The above analysis leads us to the following theorem.
\begin{thm}
A fundamental solution and an opposite antipodal fundamental solution of the
Laplace-Beltrami operator on
the rank one noncompact and compact 
symmetric spaces respectively  given in terms of the geodesic radii $r\in[0,\infty)$, $\theta\in[0,\pi/2]$ on these manifolds are given by 
\begin{eqnarray}
&&\hspace{-8.0cm}
\mcC(r)
=
\frac{(n-1)!}{2\pi^nR^{2n-2}}Q_0^{(n-1,0)}(\cosh(2r)),\\[0.05cm]
&&\hspace{-8.0cm}\mcH(r)=
\frac{(2n)!}{2\pi^{2n}R^{4n-2}}Q_0^{(2n-1,1)}(\cosh(2r)),\\[0.05cm]
&&\hspace{-8.0cm}\mcO
(r)=\frac{302\,400}{\pi^8 R^{14}}Q_0^{(7,3)}(\cosh(2r)),\\
&&\hspace{-8.0cm}\mpC(\theta)=
\frac{(n-1)!}{2\pi^nR^{2n-2}}{\mathsf Q}_0^{(n-1,0)}(\cos(2\theta)),\\
&&\hspace{-8.0cm}\mpH(\theta)=
\frac{(2n)!}{2\pi^{2n}R^{4n-2}}{\mathsf Q}_0^{(2n-1,1)}(\cos(2\theta)),\\
&&\hspace{-8.0cm}\mpO(\theta)=\frac{302\,400}{\pi^8 R^{14}}{\mathsf Q}_0^{(7,3)}(\cos(2\theta)).
\end{eqnarray}
\end{thm}
\begin{proof}
The complex, quaternionic and octonionic rank one symmetric spaces all have even dimensions, namely $s\in\{2n,4n,16\}$, respectively. It is easy to verify that the homogeneous spherically symmetric solutions of Laplace's equation on the complex, quaternionic and octonionic rank one symmetric spaces are given by Jacobi functions of the first and second kind for the noncompact manifolds and are given by Jacobi functions of the first and second kind on-the-cut for the compact manifolds, both having $\gamma=0$, $\alpha\in\{n-1,2n-1,7\}$ and $\beta\in\{0,1,3\}$ respectively. Furthermore, one requires that locally these fundamental solutions match up to a Euclidean fundamental solution.
Using \eqref{Qnear1}, \eqref{Qcutnear1}, assuming $\gamma=0$, $\alpha=a$, $\beta=b$, $a\in\N$, $b\in\N_0$, one has
the following behaviors near the singularity at unity for the Jacobi function of the second kind and the Jacobi function of the second kind on-the-cut, for $\epsilon\to0^{+}$,
\begin{eqnarray}
\label{Qnear10int}
&&\hspace{-8.2cm}Q_0^{(a,b)}(1+\epsilon)\sim
{\sf Q}_0^{(a,b)}(1-\epsilon)
\sim\frac{2^{a-1}(a-1)!b!}{(a+b)!\,\epsilon^a}.
\end{eqnarray}
Referring to the geodesic distance on the hyperbolic manifolds as $r\in[0,\infty)$ and on the compact manifolds as $\theta\in[0,\pi/2]$, 
one has 
\begin{eqnarray}
&&\hspace{-10.3cm}\cosh(2r)\sim\cosh(2\tfrac{\rho}{R})
\sim 1+\tfrac{2\rho^2}{R^2} ,\\
&&\hspace{-10.3cm}
\cos(2\theta)\sim\cos(\tfrac{2\rho}{R})\sim1-\tfrac{2\rho^2}{R^2},
\end{eqnarray}
where $\rho$ is the Euclidean geodesic distance.
Matching locally to a Euclidean fundamental solution 
\eqref{Eucfundsol}
using the 
flat-space limit (see for instance \cite[\S2.4]{CohlPalmer}), one is able to determine the constants of proportionality which are multiplied by the Jacobi functions of the second kind. This completes the proof.
\end{proof}
\noindent 
Since fundamental solutions on the rank one symmetric spaces all have $\gamma=0$, we first present the expansions in these cases. For the Jacobi functions of the first kind the $\gamma=0$ case just corresponds with unity. However, for the Jacobi functions of the second kind, these functions are quite rich, and the expansions are quite useful in that they allow one to produce separated eigenfunction expansions of a fundamental solution of Laplace's equation on these isotropic spaces.

\begin{rem}The reader should be aware that 
the addition theorems presented below for the Jacobi functions of the second kind with $\gamma=0$ are well-defined except in the case where the $\alpha$ and $\beta$ parameters on the left-hand sides are non-negative integers. In that case, special care must be taken (refer to Theorems \ref{Qnotcutsumthm}, \ref{Qcutsub}), 
even though the functions at these parameter values may be obtained by taking the appropriate limit.
\end{rem}

\begin{cor}
\label{AddnhypQgend2g0}
Let $\alpha,\beta\in\mathbb C$, 
$\beta\not\in\Z$, $z_1,z_2\in\mathbb C\setminus(-\infty,1]$,
$x_1,x_2\in\mathbb C\setminus((-\infty,-1]\cup[1,\infty))$,
$x,w\in\mathbb C$, 
with 
$Z^\pm$, ${\sf X}^\pm$ as defined in 
\eqref{ZZdef}, \eqref{XXdef}, respectively,
such that the complex variables $\alpha,\beta,z_1,z_2,x_1,x_2,x,w$ are in some
yet to be determined neighborhood of the real line.
Then
\begin{eqnarray}
&&\hspace{-0.55cm}Q_0^{(\alpha,\beta)}(Z^\pm)=
\Gamma(\alpha+1)\Gamma(\alpha+1)
\Gamma(\beta+1)
\sum_{k=0}^\infty
\frac{(\alpha+1)_k(\alpha+\beta+1)_k}{(\alpha+k)(\beta+1)_k}\nonumber\\[-0.1cm]
&&\hspace{0.25cm}\times
\sum_{l=0}^k
(\mp 1)^{k-l}(\alpha+k+l)(\alpha+1)_l(z_1z_2)^{k-l}((z_1^2-1)(z_2^2-1))^{\frac{k+l}{2}}\nonumber\\
&&\hspace{0.25cm}\times
\bm{Q}_{-k}^{(\alpha+k+l,\beta+k-l)}(2z_>^2\!-\!1)
\bm{P}_{-k}^{(\alpha+k+l,\beta+k-l)}(2z_<^2\!-\!1)
w^{k-l}P_l^{(\alpha-\beta-1,\beta+k-l)}(2w^2\!-\!1)\frac{\beta\!+\!k\!-\!l}{\beta} C_{k-l}^\beta(x),
\label{Qregularizedb}
\end{eqnarray}

\begin{eqnarray}
&&\hspace{-0.55cm}
{\mathsf{Q}}_0^{(\alpha,\beta)}({\sf X}^\pm)=
\Gamma(\alpha+1)
\sum_{k=0}^\infty(-1)^k
\frac{(\alpha+1)_k(\alpha+\beta+1)_k}{(\alpha+k)(\beta+1)_k}\nonumber\\[-0.1cm]
&&\hspace{0.25cm}\times
\sum_{l=0}^k(\mp 1)^{k-l}
(\alpha+k+l)(-\beta)_l(x_1x_2)^{k-l}((1-x_1^2)(1-x_2^2))^{\frac{k+l}{2}}\nonumber\\
&&\hspace{0.25cm}\times{\mathsf{Q}}_{-k}^{(\alpha+k+l,\beta+k-l)}(2x_<^2\!-\!1)
\bm{\mathsf P}_{-k}^{(\alpha+k+l,\beta+k-l)}(2x_>^2\!-\!1)w^{k-l}P_l^{(\alpha-\beta-1,\beta+k-l)}(2w^2\!-\!1)\frac{\beta\!+\!k\!-\!l}{\beta} C_{k-l}^\beta(x).
\label{Qcutregularizedg0}
\end{eqnarray}
\end{cor}

\begin{proof}
Substituting 
$\gamma=0$ in Theorem 
\ref{AddnhypQgend1} for the
Jacobi functions of the 
second kind
completes the proof.
\end{proof}

Next we give examples of the expansions for complex and quaternionic hyperbolic spaces where $\beta\in\{0,1,3\}$ respectively.
First we treat the complex case which
corresponds to complex hyperbolic space and complex projective space.
In order to do this we start with 
Corollary \ref{AddnhypQgend1} and 
take the limit as $\beta\to 0$ using
\cite[(6.4.13)]{AAR}
\begin{equation}
\lim_{\mu\to0}\frac{n+\mu}{\mu}C_n^\mu(x)=\epsilon_n T_n(x),
\label{ChebyGeg}
\end{equation}
where $\epsilon_n:=2-\delta_{n,0}$ is the Neumann factor commonly
appearing in Fourier series.
\begin{cor}
\label{comcase}
Let $\alpha\in\mathbb C$, $z_1,z_2\in\mathbb C\setminus(-\infty,1]$,
$x_1,x_2\in\mathbb C\setminus((-\infty,-1]\cup[1,\infty))$,
$x,w\in\mathbb C$, 
with 
$Z^\pm$, ${\sf X}^\pm$ as defined in 
\eqref{ZZdef}, \eqref{XXdef}, respectively,
such that the complex variables $\alpha,z_1,z_2,x_1,x_2,x,w$ are in some
yet to be determined neighborhood of the real line.
Then
\begin{eqnarray}
&&\hspace{-0.55cm}Q_0^{(\alpha,0)}(Z^\pm)=
\Gamma(\alpha+1)\Gamma(\alpha+1)
\sum_{k=0}^\infty
\frac{(\alpha+1)_k(\alpha+1)_k}{(\alpha+k)k!}\nonumber\\[-0.1cm]
&&\hspace{0.25cm}\times
\sum_{l=0}^k
(\mp 1)^{k-l}(\alpha+k+l)(\alpha+1)_l(z_1z_2)^{k-l}((z_1^2-1)(z_2^2-1))^{\frac{k+l}{2}}\nonumber\\
&&\hspace{1.25cm}\times
\bm{Q}_{-k}^{(\alpha+k+l,k-l)}(2z_>^2\!-\!1)
\bm{P}_{-k}^{(\alpha+k+l,k-l)}(2z_<^2\!-\!1)w^{k-l}P_l^{(\alpha-1,k-l)}(2w^2\!-\!1)\epsilon_{k-l}T_{k-l}(x),
\label{Qregularizedc}
\end{eqnarray}
\begin{eqnarray}
&&\hspace{-0.55cm}
{\mathsf{Q}}_0^{(\alpha,0)}({\sf X}^\pm)=
\Gamma(\alpha+1)
\sum_{k=0}^\infty(-1)^k
\frac{(\alpha+1)_k(\alpha+1)_k}{(\alpha+k)k!}\nonumber\\[-0.1cm]
&&\hspace{0.25cm}\times
\sum_{l=0}^k(\mp 1)^{k-l}
(\alpha+k+l)(x_1x_2)^{k-l}((1-x_1^2)(1-x_2^2))^{\frac{k+l}{2}}\nonumber\\
&&\hspace{1.25cm}\times
{\mathcal{Q}}_{-k}^{(\alpha+k+l,k-l)}(2x_<^2\!-\!1)
\bm{\mathsf P}_{-k}^{(\alpha+k+l,k-l)}(2x_>^2\!-\!1)w^{k-l}P_l^{(\alpha-1,k-l)}(2w^2\!-\!1)\epsilon_{k-l}T_{k-l}(x).
\label{Qcutregularized0}
\end{eqnarray}
\end{cor}

\begin{proof}
Take the limit as $\beta\to 0$
in Corollary \ref{AddnhypQgend1}
using \eqref{ChebyGeg} completes the proof.
\end{proof}

Now we treat the case corresponding to the quaternionic hyperbolic and projective spaces which correspond to $\beta=1$.

\begin{cor}
\label{quatcase}
Let $\alpha\in\mathbb C$, $z_1,z_2\in\mathbb C\setminus(-\infty,1]$,
$x_1,x_2\in\mathbb C\setminus((-\infty,-1]\cup[1,\infty))$,
$x,w\in\mathbb C$,
with 
$Z^\pm$, ${\sf X}^\pm$ as defined in 
\eqref{ZZdef}, \eqref{XXdef}, respectively,
such that the complex variables $\alpha,z_1,z_2,x_1,x_2,x,w$ are in some
yet to be determined neighborhood of the real line.
Then
\begin{eqnarray}
&&\hspace{-0.55cm}Q_0^{(\alpha,1)}(Z^\pm)=
\Gamma(\alpha+1)\Gamma(\alpha+1)
\sum_{k=0}^\infty
\frac{(\alpha+1)_k(\alpha+2)_k}{(\alpha+k)(2)_k}\nonumber\\[-0.1cm]
&&\hspace{0.25cm}\times
\sum_{l=0}^k
(\mp 1)^{k-l}(1+k-l)(\alpha+k+l)(\alpha+1)_l(z_1z_2)^{k-l}((z_1^2-1)(z_2^2-1))^{\frac{k+l}{2}}\nonumber\\
&&\hspace{1.25cm}\times
\bm{Q}_{-k}^{(\alpha+k+l,1+k-l)}(2z_>^2\!-\!1)
\bm{P}_{-k}^{(\alpha+k+l,1+k-l)}(2z_<^2\!-\!1)w^{k-l}P_l^{(\alpha-2,1+k-l)}(2w^2\!-\!1)U_{k-l}(x),
\label{Qregularized1}
\end{eqnarray}
\begin{eqnarray}
&&\hspace{-0.55cm}
{\mathsf{Q}}_0^{(\alpha,1)}({\sf X}^\pm)=
\Gamma(\alpha+1)
\sum_{k=0}^\infty(-1)^k
\frac{(\alpha+1)_k(\alpha+2)_k}{(\alpha+k)(2)_k}\nonumber\\[-0.1cm]
&&\hspace{0.25cm}\times
\sum_{l=0}^k(\mp 1)^{k-l}
(1+k-l)(\alpha+k+l)(x_1x_2)^{k-l}((1-x_1^2)(1-x_2^2))^{\frac{k+l}{2}}\nonumber\\
&&\hspace{1.25cm}\times
{\mathcal{Q}}_{-k}^{(\alpha+k+l,1+k-l)}(2x_<^2\!-\!1)
\bm{\mathsf P}_{-k}^{(\alpha+k+l,1+k-l)}(2x_>^2\!-\!1)w^{k-l}P_l^{(\alpha-2,1+k-l)}(2w^2\!-\!1)U_{k-l}(x).
\label{Qcutregularized1}
\end{eqnarray}
\end{cor}

\begin{proof}
Take the limit as $\beta\to 1$
in Corollary \ref{AddnhypQgend1}
using \cite[\href{http://dlmf.nist.gov/18.7.E4}{(18.7.4)}]{NIST:DLMF}
which connects the Chebyshev polynomial of the second kind to the Gegenbauer polynomial with parameter equal to unity, namely $C_n^1(x)=U_n(x)$. This completes the proof.
\end{proof}

Now we treat the case corresponding to the octonionic hyperbolic space and octonionic projective space. This corresponds to $\beta=3$.

\begin{cor}
\label{octcase}
Let $\alpha\in\mathbb C$, $z_1,z_2\in\mathbb C\setminus(-\infty,1]$,
$x_1,x_2\in\mathbb C\setminus((-\infty,-1]\cup[1,\infty))$,
$x,w\in\mathbb C$, 
with 
$Z^\pm$, ${\sf X}^\pm$ as defined in 
\eqref{ZZdef}, \eqref{XXdef}, respectively,
such that the complex variables $\alpha,z_1,z_2,x_1,x_2,x,w$ are in some
yet to be determined neighborhood of the real line.
Then
\begin{eqnarray}
&&\hspace{-0.55cm}Q_0^{(\alpha,3)}(Z^\pm)=
2\Gamma(\alpha+1)\Gamma(\alpha+1)
\sum_{k=0}^\infty
\frac{(\alpha+1)_k(\alpha+4)_k}{(\alpha+k)(4)_k}\nonumber\\[-0.1cm]
&&\hspace{0.25cm}\times
\sum_{l=0}^k
(\mp 1)^{k-l}(3+k-l)(\alpha+k+l)(\alpha+1)_l(z_1z_2)^{k-l}((z_1^2-1)(z_2^2-1))^{\frac{k+l}{2}}\nonumber\\
&&\hspace{1.25cm}\times
\bm{Q}_{-k}^{(\alpha+k+l,3+k-l)}(2z_>^2\!-\!1)
\bm{P}_{-k}^{(\alpha+k+l,3+k-l)}(2z_<^2\!-\!1)w^{k-l}P_l^{(\alpha-4,3+k-l)}(2w^2\!-\!1)C_{k-l}^3(x),
\label{Qregularized3}
\end{eqnarray}
\begin{eqnarray}
&&\hspace{-0.55cm}
{\mathsf{Q}}_0^{(\alpha,3)}({\sf X}^\pm)=
\frac13\Gamma(\alpha+1)
\sum_{k=0}^\infty(-1)^k
\frac{(\alpha+1)_k(\alpha+4)_k}{(\alpha+k)(4)_k}\nonumber\\[-0.1cm]
&&\hspace{0.25cm}\times
\sum_{l=0}^k(\mp 1)^{k-l}
(3+k-l)(\alpha+k+l)(x_1x_2)^{k-l}((1-x_1^2)(1-x_2^2))^{\frac{k+l}{2}}\nonumber\\
&&\hspace{1.25cm}\times
{\mathcal{Q}}_{-k}^{(\alpha+k+l,3+k-l)}(2x_<^2\!-\!1)
\bm{\mathsf P}_{-k}^{(\alpha+k+l,3+k-l)}(2x_>^2\!-\!1)w^{k-l}P_l^{(\alpha-4,3+k-l)}(2w^2\!-\!1)C^3_{k-l}(x).
\label{Qcutregularized3}
\end{eqnarray}
\end{cor}

\begin{proof}
Setting $\beta=3$
in Corollary \ref{AddnhypQgend1}
completes the proof.
\end{proof}

The above calculations look almost trivial in that they are simply substitutions of the values of $\beta\in\{0,1,3\}$ and $\gamma=0$ in the addition theorems given by Theorem \ref{AddnhypQgend1}. However, it should be understood that ordinarily these computations would be extremely difficult, particularly if one was to use the standard normalizations of the Jacobi functions. With standard normalizations of Jacobi functions these particular values, and in fact for values of integer parameters $(\alpha,\beta)$ and degrees $\gamma$, the Jacobi functions are not even defined. It is only because of the strategic choice of the particular normalization that we have chosen that the evaluation of these particular values becomes quite easy. We will further take advantage of these expansions in later publications.

\subsection*{Acknowledgements}
We would like to thank Tom Koornwinder for so many things:~first for being such a great source of ideas, inspiration, insight and experience over the years; for very useful conversations which significantly improved this manuscript; for his essential help in describing and editing for accuracy, his story regarding the addition theorem for Jacobi polynomials and his interactions with Dick Askey; for informing us about Moriz All\'{e} and his pioneering work on the addition theorem ultraspherical polynomials; and for his assistance and instruction in constructing a rigorous proof of Theorem \ref{AddnhypQgen}. Thanks also to Jan Derezi\'nsky for valuable discussions and in particular about Olver normalization.


\begin{thebibliography}{10}

\bibitem{Alle1865}
M.~All\'{e}.
\newblock {\"{U}ber die Eigenschaften derjenigen Gattung von Functionen, welche
  in der Entwicklung von $(1-2qx+q^2)^{-\frac{m}{2}}$ nach aufsteigenden
  Potenzen von $q$ auftreten, und \"{u}ber die Entwicklung des Ausdruckes
  $\{1-2q[\cos\theta\cos\theta'+\sin\theta\sin\theta'\cos(\psi-\psi')]+q^2\}^{-\frac{m}{2}}$}.
\newblock {\em Sitzungsberichte der mathematisch-naturwissenschaftlichen Classe
  der kaiserlichen Akademie der Wissenschaften Wien}, 51:429--458, 1865.

\bibitem{AAR}
G.~E. Andrews, R.~Askey, and R.~Roy.
\newblock {\em Special functions}, volume~71 of {\em Encyclopedia of
  Mathematics and its Applications}.
\newblock Cambridge University Press, Cambridge, 1999.

\bibitem{MR385197}
R.~Askey.
\newblock Jacobi polynomials. {I}. {N}ew proofs of {K}oornwinder's {L}aplace
  type integral representation and {B}ateman's bilinear sum.
\newblock {\em SIAM Journal on Mathematical Analysis}, 5:119--124, 1974.

\bibitem{Askeyetal86}
R.~Askey, T.~H. Koornwinder, and M.~Rahman.
\newblock An integral of products of ultraspherical functions and a
  {$q$}-extension.
\newblock {\em Journal of the London Mathematical Society. Second Series},
  33(1):133--148, 1986.

\bibitem{Cartan1929}
E.~Cartan.
\newblock Sur la d{\'e}termination d'un syst{\`e}me orthogonal complet dans un
  espace de {Riemann} sym{\'e}trique clos.
\newblock {\em Rendiconti del Circolo Matematico di Palermo}, 53:217--252,
  1929.

\bibitem{Cartan1931}
E.~Cartan.
\newblock Lecons sur la g{\'e}om{\'e}trie projective complexe.
\newblock Paris: {Gauthier}-{Villars}. {VII}. 325 {S}. (1931)., 1931.

\bibitem{Cohlhypersphere}
H.~S. Cohl.
\newblock {Opposite antipodal fundamental solution of Laplace's equation in
  hyperspherical geometry}.
\newblock {\em Symmetry, Integrability and Geometry: Methods and Applications},
  7(108):14, 2011.

\bibitem{Cohl12pow}
H.~S. {Cohl}.
\newblock {Fourier, Gegenbauer and Jacobi expansions for a power-law
  fundamental solution of the polyharmonic equation and polyspherical addition
  theorems}.
\newblock {\em Symmetry, Integrability and Geometry: Methods and Applications},
  9(042):26, 2013.

\bibitem{CohlPalmer}
H.~S. {Cohl} and R.~M. {Palmer}.
\newblock {Fourier and Gegenbauer expansions for a fundamental solution of
  Laplace's equation in hyperspherical geometry}.
\newblock {\em Symmetry, Integrability and Geometry:~Methods and Applications,
  Special Issue on Exact Solvability and Symmetry Avatars in honour of Luc
  Vinet}, 11:Paper 015, 23, 2015.

\bibitem{Cohletal2021}
H.~S. {Cohl}, J.~{Park}, and H.~{Volkmer}.
\newblock Gauss hypergeometric representations of the {F}errers function of the
  second kind.
\newblock {\em SIGMA. Symmetry, Integrability and Geometry. Methods and
  Applications}, 17:Paper No. 053, 33, 2021.

\bibitem{Durand78}
L.~Durand.
\newblock Product formulas and {N}icholson-type integrals for {J}acobi
  functions. {I}. {S}ummary of results.
\newblock {\em SIAM Journal on Mathematical Analysis}, 9(1):76--86, 1978.

\bibitem{Durand79}
L.~Durand.
\newblock Addition formulas for {J}acobi, {G}egenbauer, {L}aguerre, and
  hyperbolic {B}essel functions of the second kind.
\newblock {\em SIAM Journal on Mathematical Analysis}, 10(2):425--437, 1979.

\bibitem{DurandFishSim}
L.~Durand, P.~M. Fishbane, and L.~M. Simmons, Jr.
\newblock Expansion formulas and addition theorems for {G}egenbauer functions.
\newblock {\em Journal of Mathematical Physics}, 17(11):1933--1948, 1976.

\bibitem{ErdelyiHTFII}
A.~Erd{\'e}lyi, W.~Magnus, F.~Oberhettinger, and F.~G. Tricomi.
\newblock {\em {Higher Transcendental Functions. {V}ol. {II}}}.
\newblock Robert E. Krieger Publishing Co. Inc., Melbourne, Fla., 1981.

\bibitem{FlenstedJensenKoornwinder73}
M.~Flensted-Jensen and T.~Koornwinder.
\newblock The convolution structure for {J}acobi function expansions.
\newblock {\em Arkiv f\"{o}r Matematik}, 11:245--262, 1973.

\bibitem{FlenstedJensenKoorn79}
M.~Flensted-Jensen and T.~H. Koornwinder.
\newblock Jacobi functions: the addition formula and the positivity of the dual
  convolution structure.
\newblock {\em Arkiv f\"{o}r Matematik}, 17(1):139--151, 1979.

\bibitem{Koornwinder79}
M.~Flensted-Jensen and T.~H. Koornwinder.
\newblock Positive definite spherical functions on a noncompact, rank one
  symmetric space.
\newblock In {\em Analyse harmonique sur les groupes de {L}ie ({S}\'{e}m.,
  {N}ancy-{S}trasbourg 1976--1978), {II}}, volume 739 of {\em Lecture Notes in
  Math.}, pages 249--282. Springer, Berlin, 1979.

\bibitem{Gegenbauer1874}
L.~{Gegenbauer}.
\newblock {\"{U}ber einige bestimmte Integrale}.
\newblock {\em Sitzungsberichte der Kaiserlichen Akademie der Wissenschaften.
  Mathematische-Naturwissenschaftliche Classe.}, 70:433--443, 1874.

\bibitem{Gegenbauer1893}
L.~{Gegenbauer}.
\newblock {Das Additionstheorem der Functionen $C_n^\nu(x)$}.
\newblock {\em Sitzungsberichte der Kaiserlichen Akademie der Wissenschaften.
  Mathematische-Naturwissenschaftliche Classe.}, 102:942--950, 1893.

\bibitem{GelfandShilov}
I.~M. Gel'fand and G.~E. Shilov.
\newblock {\em Generalized functions. {V}ol. 1}.
\newblock Academic Press [Harcourt Brace Jovanovich Publishers], New York, 1964
  [1977].
\newblock Properties and operations, Translated from the Russian by Eugene
  Saletan.

\bibitem{HeckmanSchlichtkrull94}
G.~Heckman and H.~Schlichtkrull.
\newblock {\em Harmonic analysis and special functions on symmetric spaces},
  volume~16 of {\em Perspectives in Mathematics}.
\newblock Academic Press, Inc., San Diego, CA, 1994.

\bibitem{HeckmanOpdam87}
G.~J. Heckman and E.~M. Opdam.
\newblock Root systems and hypergeometric functions. {I}.
\newblock {\em Compositio Mathematica}, 64(3):329--352, 1987.

\bibitem{HeckmanOpdam2021}
G.~J. Heckman and E.~M. Opdam.
\newblock Jacobi polynomials and hypergeometric functions associated with root
  systems.
\newblock In {\em Encyclopedia of special functions: the {A}skey-{B}ateman
  project. {V}ol. 2. {M}ultivariable special functions}, pages 217--257.
  Cambridge Univ. Press, Cambridge, 2021.

\bibitem{Heine1878}
E.~Heine.
\newblock {\em Handbuch der {K}ugelfunctionen, {T}heorie und {A}nwendungen
  (volume 1)}.
\newblock Druck und Verlag von G. Reimer, Berlin, 1878.

\bibitem{Helgason59}
S.~Helgason.
\newblock Differential operators on homogenous spaces.
\newblock {\em Acta Mathematica}, 102:239--299, 1959.

\bibitem{Helgason78}
S.~Helgason.
\newblock {\em Differential geometry, {L}ie groups, and symmetric spaces},
  volume~80 of {\em Pure and Applied Mathematics}.
\newblock Academic Press Inc. [Harcourt Brace Jovanovich Publishers], New York,
  1978.

\bibitem{Helgason84}
S.~Helgason.
\newblock {\em Groups and geometric analysis: Integral geometry, invariant
  differential operators, and spherical functions}, volume 113 of {\em Pure and
  Applied Mathematics}.
\newblock Academic Press Inc., Orlando, FL, 1984.

\bibitem{Jacobi1859}
C.~G.~J. Jacobi.
\newblock Untersuchungen \"{u}ber die {D}ifferentialgleichung der
  hypergeometrischen {R}eihe.
\newblock {\em J. Reine Angew. Math.}, 56:149--165, 1859.

\bibitem{Koekoeketal}
R.~Koekoek, P.~A. Lesky, and R.~F. Swarttouw.
\newblock {\em Hypergeometric orthogonal polynomials and their
  {$q$}-analogues}.
\newblock Springer Monographs in Mathematics. Springer-Verlag, Berlin, 2010.
\newblock With a foreword by Tom H. Koornwinder.

\bibitem{Koornwinder73}
T.~Koornwinder.
\newblock The addition formula for {J}acobi polynomials and spherical
  harmonics.
\newblock {\em SIAM Journal on Applied Mathematics}, 25:236--246, 1973.

\bibitem{Koornwinder74}
T.~Koornwinder.
\newblock Jacobi polynomials. {II}. {A}n analytic proof of the product formula.
\newblock {\em SIAM Journal on Mathematical Analysis}, 5:125--137, 1974.

\bibitem{Koornwinder75}
T.~Koornwinder.
\newblock Jacobi polynomials. {III}. {A}n analytic proof of the addition
  formula.
\newblock {\em SIAM Journal on Mathematical Analysis}, 6:533--543, 1975.

\bibitem{Koornwinder77}
T.~Koornwinder.
\newblock Yet another proof of the addition formula for {J}acobi polynomials.
\newblock {\em Journal of Mathematical Analysis and Applications},
  61(1):136--141, 1977.

\bibitem{Koornwinder1972AI}
T.~H. Koornwinder.
\newblock The addition formula for {J}acobi polynomials. {I}. {S}ummary of
  results.
\newblock {\em Nederl. Akad. Wetensch. Proc. Ser. A {\bf 75}=Indag. Math.},
  34:188--191, 1972.

\bibitem{Koornwinder72A}
T.~H. Koornwinder.
\newblock {The addition formula for Jacobi polynomials. I. Summary of results}.
\newblock {\em Stichting Mathematisch Centrum, Afdeling Toegepaste Wiskunde.
  TW: 131/71}, November 1972.

\bibitem{Koornwinder72B}
T.~H. Koornwinder.
\newblock {The addition formula for Jacobi polynomials, II : the Laplace type
  integral representation and the product formula}.
\newblock {\em Stichting Mathematisch Centrum, Afdeling Toegepaste Wiskunde.
  TW: 133/72,
  \href{http://persistent-identifier.org/?identifier=urn:nbn:nl:ui:18-7722
  }{\tt http://persistent-identifier.org/?identifier=urn:nbn:nl:ui:18-7722}},
  April 1972.

\bibitem{Koornwinder72C}
T.~H. Koornwinder.
\newblock {The addition formula for Jacobi polynomials, III : Completion of the
  proof}.
\newblock {\em Stichting Mathematisch Centrum, Afdeling Toegepaste Wiskunde.
  TW: 135/72,
  \href{http://persistent-identifier.org/?identifier=urn:nbn:nl:ui:18-12598}{\tt
  http://persistent-identifier.org/?identifier=urn:nbn:nl:ui:18-12598}},
  December 1972.

\bibitem{Koornwinder18}
T.~H. {Koornwinder}.
\newblock {Dual addition formulas associated with dual product formulas}.
\newblock In {\em {Frontiers in Orthogonal Polynomials and $q$-Series}},
  chapter~19, pages 373--392. World Scientific Publishing, Hackensack, NJ,
  2018.
\newblock Zuhair Nashed and Xin Li, editors,
  \href{https://arxiv.org/abs/1607.06053v4} {\bf\tt\normalsize
  arXiv:1607.06053v4}.

\bibitem{KoornwinderSchwartz97}
T.~H. Koornwinder and A.~L. Schwartz.
\newblock Product formulas and associated hypergroups for orthogonal
  polynomials on the simplex and on a parabolic biangle.
\newblock {\em Constructive Approximation. An International Journal for
  Approximations and Expansions}, 13(4):537--567, 1997.

\bibitem{LiPeng2007}
Z.~Li and L.~Peng.
\newblock Some representations of translations of the product of two functions
  for {H}ankel transforms and {J}acobi transforms.
\newblock {\em Constructive Approximation. An International Journal for
  Approximations and Expansions}, 26(1):115--125, 2007.

\bibitem{Macdonald2013}
I.~G. {Macdonald}.
\newblock {Hypergeometric Functions, I.}
\newblock {\em \href{https://arxiv.org/abs/1309.4568} {\bf\tt\normalsize
  arXiv:1309.4568}}, 2013.

\bibitem{MOS}
W.~Magnus, F.~Oberhettinger, and R.~P. Soni.
\newblock {\em Formulas and theorems for the special functions of mathematical
  physics}.
\newblock Third enlarged edition. Die Grundlehren der mathematischen
  Wissenschaften, Band 52. Springer-Verlag New York, Inc., New York, 1966.

\bibitem{Miller}
W.~Miller, Jr.
\newblock {\em Symmetry and separation of variables}.
\newblock Addison-Wesley Publishing Co., Reading, Mass.-London-Amsterdam, 1977.
\newblock With a foreword by Richard Askey, Encyclopedia of Mathematics and its
  Applications, Vol. 4.

\bibitem{Olver:1997:ASF}
F.~W.~J. Olver.
\newblock {\em Asymptotics and Special Functions}.
\newblock AKP Classics. A K Peters Ltd., Wellesley, MA, 1997.
\newblock Reprint of the 1974 original [Academic Press, New York].

\bibitem{Opdam1988}
E.~{Opdam}.
\newblock {\em {{Generalized Hypergeometric Functions Associated with Root
  Systems}}}.
\newblock dissertation, Leiden University, Leiden, The Netherlands, 1988.

\bibitem{NIST:DLMF}
{\it NIST Digital Library of Mathematical Functions}.
\newblock \href{https://dlmf.nist.gov/}{{\bf\tt\normalsize
  https://dlmf.nist.gov/}}, Release 1.1.9 of 2023-03-15.
\newblock F.~W.~J. Olver, A.~B. {Olde Daalhuis}, D.~W. Lozier, B.~I. Schneider,
  R.~F. Boisvert, C.~W. Clark, B.~R. Miller, B.~V. Saunders, H.~S. Cohl, and
  M.~A. McClain, eds.

\bibitem{Szego}
G.~Szeg{\H{o}}.
\newblock {\em Orthogonal polynomials}.
\newblock American Mathematical Society Colloquium Publications, Vol. 23.
  Revised ed. American Mathematical Society, Providence, R.I., fourth edition,
  1975.

\bibitem{VilenkinSapiro1967}
N.~Ja. Vilenkin and R.~L. \v{S}apiro.
\newblock Irreducible representations of the group {${\rm SU}(n)$} of class {I}
  relative to {${\rm SU}(n-1)$}.
\newblock {\em Izvestija Vys\v{s}ih U\v{c}ebnyh Zavedeni\u{\i} Matematika},
  1967(7 (62)):9--20, 1967.

\bibitem{Sapiro1968}
R.~L. \v{S}apiro.
\newblock Special functions related to representations of the group {${\rm
  SU}(n)$}, of class {I} with respect to {${\rm SU}(n-1)$} {$(n\geqq 3)$}.
\newblock {\em Izvestija Vys\v{s}ih U\v{c}ebnyh Zavedeni\u{\i} Matematika},
  1968(4 (71)):97--107, 1968.

\bibitem{WimpMcCabeConnor97}
J.~Wimp, P.~McCabe, and J.~N.~L. Connor.
\newblock Computation of {J}acobi functions of the second kind for use in
  nearside-farside scattering theory.
\newblock {\em Journal of Computational and Applied Mathematics},
  82(1-2):447--464, 1997.
\newblock Seventh 96 International Congress on Computational and Applied
  Mathematics (Leuven).

\end{thebibliography}

\def\cprime{$'$} \def\dbar{\leavevmode\hbox to 0pt{\hskip.2ex \accent"16\hss}d}

\end{document}